\documentclass[11pt, dvipsnames]{article} %%%% 11pt pro Combinatoricu

\newif\ifcomb 
\combfalse % Verze pro arXiv

\ifcomb
\usepackage[letterpaper, margin=1in]{geometry} 		%%%%% Combinatorica 
\else
\usepackage[a4paper, margin=1in]{geometry}				
\fi

\usepackage[utf8]{inputenc}

\usepackage{lmodern}

\usepackage{amsmath}        % extensions for typesetting of math
\usepackage{amsfonts}       % math fonts
\usepackage{amsthm}         % theorems, definitions, etc.
\usepackage{amssymb}
\usepackage{bbding}         % various symbols (squares, asterisks, scissors, ...)
\usepackage{bm}             % boldface symbols (\bm)
\usepackage{bbm}			% for blackboard bold symbols
\usepackage{graphicx}       % embedding of pictures
\usepackage{fancyvrb}       % improved verbatim environment
\usepackage[style=alphabetic,natbib=true,useprefix=true,sorting=nyt]{biblatex}

\usepackage{dcolumn}        % improved alignment of table columns
\usepackage{booktabs}       % improved horizontal lines in tables
\usepackage{paralist}       % improved enumerate and itemize
\usepackage{xcolor}  % typesetting in color

\usepackage{enumerate}
\usepackage{mathtools} % for DeclarePairedDelimiter
\usepackage{tikz}

\usepackage[mathlines, modulo]{lineno}

\usepackage[margin=1cm]{caption}

\usepackage{url}

\usepackage{color}
\definecolor{grey}{rgb}{.7,.7,.7}
\definecolor{blue}{rgb}{0,0,.8}
\definecolor{red}{rgb}{.8,0,0}
\definecolor{orange}{rgb}{1,.3,0}
\definecolor{green}{rgb}{0,.4,0}
\definecolor{gold}{rgb}{0.8,0.6,0.1}
\definecolor{brown}{rgb}{0.8,0.4,0.1}
\def\marrow{{\marginpar[\hfill$\Rrightarrow$]{$\Lleftarrow$}}}

\def\vk#1{\color{brown} {\textsc VK: }{\marrow\textsf{#1}} \normalcolor}
\def\martin#1{\color{blue} {\textsc MARTIN: }{\marrow\textsf{#1}} \normalcolor}
\def\todo#1{\color{red} {\textsc TO~DO: }{\marrow\textsf{#1}} \normalcolor}
\def\red#1{\color{red} #1 \normalcolor}

\def\green#1{\color{green} #1 \normalcolor}

\usepackage{hyperref}
\definecolor{blue3}{rgb}{.1,.0,.4}
\hypersetup{colorlinks=true, linkcolor=red, urlcolor=blue3, citecolor=green, pdfpagemode=UseNone, pdfstartview=FitH}
\hypersetup{unicode}
\hypersetup{breaklinks=true}
\hypersetup{hypertexnames=false} %without this the hyperreferences got messed up

\hypersetup{
  pdftitle    = {Even maps, the Colin de Verdière number and representations of graphs},
  pdfsubject  = {},
  pdfauthor   = {Vojtěch Kaluža, Martin Tancer},
  pdfkeywords = {Colin de Verdière parameter, van der Holst–Pendavingh parameter, even maps, spectrum of a graph, minor-closed families, adjacency matrix, representations of graphs} ,
  pdfcreator  = {pdflatex},
  pdfproducer = {LaTeX with hyperref}
}

\makeatletter
\def\@makechapterhead#1{
  {\parindent \z@ \raggedright \normalfont
   \Huge\bfseries \thechapter. #1
   \par\nobreak
   \vskip 20\p@
}}
\def\@makeschapterhead#1{
  {\parindent \z@ \raggedright \normalfont
   \Huge\bfseries #1
   \par\nobreak
   \vskip 20\p@
}}
\makeatother

\newcounter{inlineenum}
\renewcommand{\theinlineenum}{\alph{inlineenum}}
\newenvironment{inlineenum}
  {\unskip\ignorespaces\setcounter{inlineenum}{0}%
   \renewcommand{\item}{\refstepcounter{inlineenum}{\textit{\theinlineenum})~}}}
  {\ignorespacesafterend}

\theoremstyle{plain}
\newtheorem{thm}{Theorem}
\newtheorem{lemma}[thm]{Lemma}
\newtheorem{prop}[thm]{Proposition}

\newtheorem*{quest*}{Question}
\newtheorem{cor}[thm]{Corollary}
\newtheorem{obs}[thm]{Observation}

\newcounter{claimcounter}[thm]
\numberwithin{claimcounter}{thm}
\newtheorem*{claim*}{Claim}
\newtheorem{claim}[claimcounter]{Claim}

\theoremstyle{plain}
\newtheorem{defn}[thm]{Definition}

\theoremstyle{remark}
\newtheorem*{cor*}{Corollary}
\newtheorem*{example*}{Example}
\newtheorem{remark}[thm]{Remark}

\newtheorem*{preremark*}{Remark}
  \newenvironment{rem*}%
    {\begin{preremark*}\itshape}{\end{preremark*}}

\DefineVerbatimEnvironment{code}{Verbatim}{fontsize=\small, frame=single}

\newcommand{\R}{\mathbb{R}}
\newcommand{\N}{\mathbb{N}}
\newcommand{\Q}{\mathbb{Q}}
\newcommand{\Z}{\mathbb{Z}}

\DeclareMathOperator{\supp}{supp}

\DeclareMathOperator{\corank}{corank}

\DeclareMathOperator{\sign}{sign}
\DeclareMathOperator{\st}{st}

\newcommand{\mb}[1]{\bm{#1}}

\newcommand{\mc}[1]{\mathcal{#1}}
\newcommand{\mt}[1]{\mathtt{#1}}

\newcommand{\abs}[1]{\left|#1\right|}
\newcommand{\lnorm}[2]{\left\|#2\right\|_#1}

\newcommand{\set}[1]{\left\{#1\right\}}
\newcommand{\cl}[1]{\overline{#1}}
\newcommand{\rest}[2]{#1|_{#2}}% restriction of something to something else
\newcommand{\indg}[2]{#1[#2]}% induced subgraph/subcomplex

\newcommand{\E}{\mc{E}}

\DeclarePairedDelimiter{\floor}{\lfloor}{\rfloor}
\DeclarePairedDelimiter{\br}{(}{)}

\def\XXint#1#2#3{{\setbox0=\hbox{$#1{#2#3}{\int}$ }
		\vcenter{\hbox{$#2#3$ }}\kern-.6\wd0}}
\makeatletter
\newcommand{\mylabel}[2]{{\normalfont #2}\def\@currentlabel{#2}\label{#1}}
\makeatother

\renewcommand{\C}{\mathcal{C}}
\newcommand{\D}{\mathcal{D}}
\newcommand{\K}{\mathcal{K}}
\renewcommand{\L}{\mathcal{L}}
\newcommand{\Nc}{\mathcal{N}}
\renewcommand{\P}{\mathcal{P}}
\renewcommand{\S}{\mathcal{S}}
\newcommand{\F}{\mathcal{F}}

\newcommand{\CC}{\mathbf{C}}
\newcommand{\DD}{\mathbf{D}}
\newcommand{\EE}{\mathbf{E}}
\newcommand{\FF}{\mathbf{F}}
\newcommand{\HH}{\mathbf{H}}
\newcommand{\PP}{\mathbf{P}}
\newcommand{\QQ}{\mathbf{Q}}
\renewcommand{\aa}{\mathbf{a}}
\newcommand{\bb}{\mathbf{b}}
\newcommand{\ee}{\mathbf{e}}
\newcommand{\uu}{\mathbf{u}}
\newcommand{\ww}{\mathbf{w}}

\newcommand{\Gt}{\mathsf{G}}
\newcommand{\gmu}[1]{\Gt_{\mu \leq #1}}
\newcommand{\gsigma}[1]{\Gt_{\sigma \leq #1}}

\renewcommand{\mb}[1]{\mathbf{#1}}

\DeclareMathOperator{\sd}{sd}

\author{Vojtěch Kaluža\thanks{V.K. gratefully acknowledges the support of Austrian Science Fund (FWF): P 30902-N35.
During the early stage of this research, V.K. was partially supported by Charles University project GAUK 926416.}
\\ \small{Institut für Mathematik, Universität Innsbruck, Austria}
\\ \small{\href{mailto:vojtech.kaluza@uibk.ac.at}{vojtech.kaluza@uibk.ac.at}}
\and
Martin Tancer\thanks{M.T. is supported by the grant
no.~19-04113Y of the Czech Science Foundation (GA\v{C}R) and partially
supported by Charles University project UNCE/SCI/004.}
\\ \small{Department of Applied Mathematics, Charles University, Prague, Czech
Republic}
\\ \small{\href{mailto:tancer@kam.mff.cuni.cz}{tancer@kam.mff.cuni.cz}}
}

\date{}

\title{Even maps, the Colin de~Verdi\`{e}re number and representations of graphs}

\begin{document}

\maketitle

\begin{abstract}
  Van der Holst and Pendavingh introduced a graph parameter $\sigma$,
  which coincides with the more famous Colin de Verdi\`{e}re graph parameter
  $\mu$ for small values. However, the definition of $\sigma$ is much more
  geometric/topological directly
  reflecting
%  \vk{Tady navrhuju rozdělení na 2 věty, protože v původní verzi
%  nevím, jaká tam měla být interpunkce. Přijde mi, že tam mělo něco být, ale
%  nebyl jsem si jist, tak jsem raději navrhnul reformulaci, která by měla být
%  snad bezproblémová.}\martin{Vsechny ostatni zmeny OK a diky za ne. Tady by se
%  mi malinko vic libila varianta v jedne vete, protoze se to stale vaze na to
%  `however'. V cestine mi prijde, ze by se to dalo interpretovat bud jako
%  privlastek postupne rozvijejici bez carky nebo pripadne pouze ta cast od
%  `directly reflecting' jako privlastek volny s carkou. Ale vyznamove bych se
%  spis klonil k tomu postupne rozvijejicimu. V anglictine bych to mel tendenci
%  v tomhle pripade vyhodnotit stejne, ale zadny podklad k tomu nemam a muzu byt
%  ovlivneny cestinou. Jeste by kompromis mohl byt dat tam pomlcku---}
  embeddability properties of the graph. They proved
  $\mu(G) \leq \sigma(G) + 2$ and conjectured $\mu(G) \leq \sigma(G)$ for any
  graph $G$. We confirm this conjecture.
%  \vk{Tady mi přijde, že je lepší "to confirm conjecture, než "to prove
%  conjecture". Taky this mi přišlo určitější než their}. 
  As far as we know,
  this is the first topological upper bound on $\mu(G)$ which is, in general,
  tight.

 Equality between $\mu$ and $\sigma$ does not hold in general as van der Holst and Pendavingh showed that there is a graph $G$ with
  $\mu(G) \leq 18$ and $\sigma(G)\geq 20$. We show that the gap appears on much
  smaller values, namely, we exhibit a graph $H$ for which $\mu(H)\leq 7$ and
  $\sigma(H)\geq 8$. We also prove %\vk{tady mi přišlo, že
%  těch "show" je zasebou už moc} 
  that, in general, the gap can be large: The
  incidence graphs $H_q$ of finite projective planes of order $q$ satisfy
  $\mu(H_q) \in O(q^{3/2})$ and $\sigma(H_q) \geq q^2$.
\end{abstract}

%\todo{Před sumbission předělat na letter-size}
%\todo{Taky odstranit blackboxy pro preteceni radku}

%\newpage

%\pdfbookmark{\contentsname}{toc} % to appear also in the pdf bookmarks
%\tableofcontents

\ifcomb
\bigskip
%\new{
An earlier version of this work appeared in the proceedings of the 2020
ACM-SIAM Symposium on Discrete Algorithms (SODA20)~\citep{kaluza_tancer_soda}.
The earlier version contains full proofs of the results mentioned in the abstract.
%also contained the full proofs of all three theorems listed
%above. 
The present version contains additionally an Appendix~\ref{a:mu} giving an
explicit PSPACE algorithm for computing $\mu$. Moreover, we alter the core definition of our work (a semivalid representation;
see Definition~\ref{def:ex_represent}). In our opinion, the new definition is
conceptually simpler and fits better into the existing theory.  We use it to
introduce a new graph parameter, which we call $\eta$ and which did not appear
in the previous version. It allows us to establish an extended version of
the main result showing that $\mu(G)\leq\eta(G)\leq\sigma(G)$ holds for every graph $G$.
%} \vk{Tady přidávám svůj návrh na upozornění, že už to bylo publikováno dříve.
%Snažil jsem se to napsat v souladu s tím, na čem jsme se dohodli dříve. Nevím,
%jestli to chceme uzavřít do vlastního paragraphu, nebo do prostředí poznámky,
%nebo tak něco.}

\linenumbers
\else
\fi

\section{Introduction}
%\martin{Tady se snazim napsat nejaky poradnejsi uvod. Ale jeste to neni hotove
%a asi to nema moc vyznam zatim cist.}

In 1990 \citet{CdV_orig} (English translation \citep{CdV_orig_en}) introduced a
graph parameter $\mu(G)$. 
It arises from the study of the multiplicity of the second smallest
%\vk{smallest} 
% \martin{Ja nikdy nemam jasno v tom, kdy je ta matice mimo diagonalu jeste
% nezaporna a kdy uz nekladna.}\vk{U Schrödingerových operátorů se očekává matice mimo diagonálu nekladná.} \martin{Takze u Schr. operatoru se ocekava, ze uz je
% jedine zaporne vlastni cislo?}\vk{Ne, neočekává. Jen je potřeba říct, jak člověk vlastní čísla řadí, jinak výraz "druhé vlastní číslo" nemá jasný význam.}
eigenvalue
of certain matrices associated to a graph $G$ (discrete Schr\"{o}dinger operators);
however, it turns out that this parameter is 
closely related to geometric and topological properties of $G$. 
In particular, this parameter is minor monotone, and moreover, it satisfies:
%\martin{Tady jsem jeste pridal tu charakterizaci pro $\mu \leq 0$ radeji
%explicitne, kdyz ji chceme pak vyuzit.}
%\vk{Souhlas. Jen jsem raději u té nuly napsal rovnost, aby se někdo nemyslel, že ten parametr může být záporný. Může být?}
\begin{enumerate}[(i)]
\item $\mu(G) = 0$ if and only if $G$ embeds in $\R^0$;
\item $\mu(G) \leq 1$ if and only if $G$ embeds in $\R^1$;
\item $\mu(G) \leq 2$ if and only if $G$ is outer planar;
\item $\mu(G) \leq 3$ if and only if $G$ is planar; and
\item $\mu(G) \leq 4$ if and only if $G$ admits a linkless embedding into
  $\R^3$.
\end{enumerate}
%\martin{Je zminka o diskretnich Schr\"{o}dingerovych operatorach vyse v poradku
%z hlediska terminologie?}
%\vk{Myslím, že ano.}

The characterization up to the value 3 as well as the minor monotonicity of $\mu$ was shown by \citet{CdV_orig, CdV_orig_en}. The characterization of graphs with $\mu(G)\leq 4$ was established by \citet{CdV_linkless}.
%\vk{Tady navrhuji přidat zelený text, zdálo se mi, že taková reference tady
%chybí.}\martin{Ja bych klidne dal jenom tri reference za slovo `satisfies' pred
%vyctem, aby
%tok textu byl plynulejsi. Beztak je to `general knowledge', kterou lze najit v
%mnoha zdorjich. Ale jesti to chces takhle, tak taky OK.}
%\vk{Už jsme se o tom bavili dříve a vím, že se na to díváš jinak, mně ale přijde, že něco jako \citep{CdV_orig} je součástí věty, podobně jako třeba matematická formule. Pokud to tak je, tak potom dát to za "satisfies" mi nedává moc smysl. Ale přiznám se, že nevím, jak to doopravdy je. Musím to někdy zkusit zjisit, teď se mi tím ale nechce zdržovat, tak pokud Ti to tedy nevadí, nechám tady variantu, co jsem navrhoval já.}
%\red{The description of graph classes with $\mu(G) \leq k$ for $k \geq 5$
%remains mysterious, though it is an open problem whether the graphs with
%$\mu(G) \leq 5$ coincide with knotlessly embeddable
%graphs \citep[Sec 14.5]{knotless_and_CdV},~\citep[Sec.
%7]{Thomas_knotless_and_CdV}.}
Beyond this, any description is known only for the classes of graphs
with $\mu(G)\geq \abs{V(G)}-k$ for $k=1,2,3$ and partial results are known also for $k=4,5$; see \citep{CdV_sphere_rep}.
It used to be an open problem whether the graphs with
$\mu(G) \leq 5$ coincide with knotless embeddable
graphs \citep[Sec.~14.5]{knotless_and_CdV},~\citep[Sec.~7]{Thomas_knotless_and_CdV}.
However, a graph $H$ constructed
by~\citet{foisy03} satisfies $\mu(H) \leq 5$ whereas it is not knotless
embeddable.\footnote{The inequality $\mu(H) \leq 5$ follows from the fact that
there is a vertex $v$ of $H$ such that $H - v$ is a linkless embeddable graph,
that is, $\mu(H - v) \leq 4$.} We are very thankful to Rose McCarty for sharing this example
with us~\citep{mccarty19pc}.
%\vk{Ještě existuje popis pro třídy grafů s $\mu(G)$ blízko $\abs{V(G)}$. Možná bych to tady jednou větou zmínil, tj. nahradil červený text zeleným. (Ne že by to pro nás bylo důležité, ale aby předchozí tvrzení bylo pravdivé.)}
%\martin{Jak chces. Predchozi tvrzeni podle me bylo pravdive, protoze to, co
%pises, neni popis trid s $\mu \geq k$. Ty tridy, kde je $\mu$ blizko vrcholu me
%taky napadly, ale pak jsem je tam nedal, protoze mi prislo, ze zase az tolik
%nesouviseji s tim, co delame. Ale jestli to chces zminit radeji v zelene
%variante, tak OK. Protoze jsem uz rovnou chtel zakmponovat zminku o knotless,
%tak jsem ji radeji pridal do cervene i zelene verze. U zelene verze}
%\martin{...by jeste
%stalo za zvazeni, jestli k tem dodatecnym charakterizacim nepridat odkaz na
%\citep{CdV_main}?}
%\vk{Přidal jsem odkaz na \citep{CdV_sphere_rep}, což je originální reference.}

Due to the aforementioned properties, the study of $\mu$ gained a lot of
popularity (e.g., \citep{Bacher_CdV, vdHolst_planar, CdV_clique_sums, CdV_sphere_rep, CdV_linkless, 
%Pendavingh_separation, 
CdV_main, CdV_nullspace,
CdV_Steinitz, %sigma, 
Izmestiev, Goldberg, flat_has_Arnold, McCarty, Tait}).
A precise definition of the parameter $\mu$ is given at the end
of Subsection~\ref{s:CdV}.
Later,
%\vk{Řekl bych, že 19 let je v matematice spíše o dost později,
%než o něco později.}\martin{Myslel jsem, ze to `somewhat' uvozuje, ze to nebylo
%hned po tom (jakoze je to silnejsi nez jenom `later'. Chtel jsem tedy napsat
%neco jako `Ponekud pozdeji'. Jestli mas jakykoliv navrh, jak to zmenit, tak ho
%vitam.}
%\vk{Mně to "somewhat" spíše zní, že to to "later" oslabuje, než zesiluje, spíše jako "trochu později" nebo "o něco později". Já bych to prostě vynechal.}
in 2009, \citet{sigma} introduced another minor monotone
parameter $\sigma(G)$, whose definition is much closer to the topological
properties of $G$. Roughly speaking, $\sigma(G)$ is defined as a minimal
integer $k$ such that every CW-complex $\C$ whose $1$-skeleton is $G$ admits
a so-called even mapping into $\R^k$. This is a mapping $f$ such that 
whenever $\vartheta$ and $\tau$ are disjoint cells of $\C$, then $f(\vartheta)
\cap f(\tau) = \emptyset$ if $\dim \vartheta + \dim \tau < k$, and $f(\vartheta)$ and
$f(\tau)$ cross in an even number of points if $\dim \vartheta + \dim \tau = k$.
For a precise definition, we refer to~\citep{sigma}.

%Let $\gmu k$ and $\gsigma k$ denote the class of graphs with $\mu \leq k$ and
%$\sigma \leq k$ respectively. It turns out that $\gmu k = \gsigma k$ for $k \in
%\{0, 1, 2, 3, 4\}$. 

It turns out that $\sigma(G) \leq k$ if and only if $\mu(G) \leq k$ for $k \in
\{0, 1, 2, 3, 4\}$.
In addition, \citet[Conj.~43]{sigma} conjectured that this is true
also for $k=5$. However, in general, $\sigma$ and $\mu$ differ.
%\Citet{sigma} 
They provide an example of a graph with $\mu(G) \leq 18$, but
$\sigma(G) \geq 20$ based on a previous work
of~\citet{Pendavingh_separation}.
%\martin{Vyse `citet' moc nefunguje, kdyz
%neumi velke pismeno na zacatku vety :-(} 
On the other hand, %they 
\citet[Cor.~41]{sigma} 
proved that $\mu(G) \leq \sigma(G) + 2$, while they conjectured that $\mu(G)
\leq \sigma(G)$. We confirm this conjecture.
%\vk{Tady se ve třech větách ze sebou opakuje "van der Holst and Pendavingh" a
%vždycky se odkazuje na stejný článek. Myslím, že aspoň poslední výskyt by se
%dal nahradit "they".}\martin{Opraveno. Nakonec jsem teda nahradil spis druhy
%vyskyt, protoze u toho tretiho tam byl `v ceste' jeste sam Pendavingh. Kdybys
%na to mel silnejsi nazor, tak to, prosim, zmen.}

\begin{thm}
\label{t:main}
For any graph $G$, $\mu(G) \leq \sigma(G)$.
\end{thm}

%\martin{Zatim jsem se nepokousel odkazy z intra zapracovat do dalsich sekci. To
%bude vhodne, az muj navrh intra bude aspon trochu schvaleny.}

Our tools that we use for the proof of Theorem~\ref{t:main} also allow us to
show that the gap between $\mu$ and $\sigma$ appears at much smaller values.

\begin{thm}
\label{t:gap}
There is a graph $G$ such that $\mu(G)\leq 7$ and $\sigma(G)\geq 8$.
\end{thm}
%\vk{Nezmínit ještě někde, že když se $\gsigma k$ a $\gmu k$ liší, tak se liší i $\gsigma l$ a $\gmu l$ pro každé $l\geq k$?}
%\martin{Jak chces, nepovazuju to za klicove, ale klidne to zmin. Je to mozna
%zase neco za neco vzhledem k plynulosti textu.}
%\vk{OK, přidal jsem odstavec níže.}

%\green{
We remark here that adding a new vertex to a graph $G$ and connecting it to all
vertices of $G$ increases both $\mu(G)$ and $\sigma(G)$ by exactly one (unless
$G$ is the complement of $K_2$); see \citep[Thm.~2.7]{CdV_main} and
\citep[Thm.~28]{sigma}. Consequently, Theorem~\ref{t:gap} immediately implies
that for every $k\in\N, k\geq 7$ there is a graph $G_k$ with $\mu(G_k)\leq k$
and $\sigma(G_k)\geq k+1$.%}

The key step in the proof of Theorem~\ref{t:gap} is to provide a lower bound on $\sigma$; otherwise we
follow~\citep{Pendavingh_separation}. We remark that the example of $G$ with
$\mu(G) \leq 18$ but $\sigma(G) \geq 20$ coming from \cite{sigma,
Pendavingh_separation} is highly regular Tutte's 12-cage. The important
property is that the second largest eigenvalue of the adjacency matrix of Tutte's 12-cage has very high
multiplicity. We use instead the incidence graphs of finite projective planes, which
enjoy the same property. Namely, if $H_q$ is the incidence graph of a finite
projective plane of order $q$, we will show that $\mu(H_3) \leq 9$, whereas
$\sigma(H_3) \geq 11$; see Proposition~\ref{p:9_11}. Then, by further modification of this graph, we obtain the
graph from Theorem~\ref{t:gap}.

As a complementary result, based on properties of finite projective planes, we
also show that the gap between $\mu$ and $\sigma$ is asymptotically large.

\begin{thm}
\label{t:asymptotic}
Let $q\in\N$ be such that a finite projective plane of order $q$
exists\footnote{This includes all prime powers $q$ (see, e.g., \citep[Sec.~2.3]{Stinson_designs}).}. Then $\mu(H_q)\in
  O\br*{q^{3/2}}$, while $\sigma(H_q)\geq\lambda(H_q)\geq q^2$, where
  $\lambda$ is the graph parameter of~\citet{lambda}, which we overview in
  Section~\ref{s:gaps}.
\end{thm}

\paragraph{Further motivation and computational aspects.} If we are interested
only in the properties of the Colin de Verid\`{e}re parameter $\mu$,
Theorem~\ref{t:main} can be reformulated as: If $\sigma(G) \leq k$, then
$\mu(G) \leq k$. In other words, if $G$ has a nice geometric
description\footnote{In fact, Theorem~30 of~\citep{sigma} reveals that an even
mapping of a CW-complex $\C$ (in the definition of $\sigma$) can be exchanged
with an even mapping of the $\lfloor k /2 \rfloor$-skeleton of $\C$ into $\R^{k-1}$,
provided that in addition $\C$ is a so-called closure (which can be assumed in
the definition of $\sigma$). This explains the shift of the dimension in
the geometric description of the classes with $\mu(G) \leq 3$ or $\mu(G) \leq 4$,
equivalently, the classes with $\sigma(G) \leq 3$ or $\sigma(G) \leq 4$.} in
$\R^k$, then $\mu(G) \leq k$. This is tight in general
%as tight as possible 
%\vk{"as tight as possible" mi zní, jako že to nejde zlepšit. Asi bych raději
%napsal, že je to "in general tight", nebo tak něco.}\martin{Nahrazeno pomoci
%`tight in general'}
because $\mu(K_n) =
\sigma(K_n) = n-1$, where $K_n$ is the complete graph on $n$
vertices~\cite{CdV_main, sigma}. As far as we can say this is the first tight
upper bound on $\mu(G)$ taking into account embeddability properties of $G$ for general value
of the parameter.\footnote{For comparison, there is a result of
\citet{Izmestiev} providing a quite different lower bound on $\mu$: If $G$ is a $1$-skeleton of
convex $d$-polytope, then $\mu(G) \geq d$. However, as Izmestiev points out,
this result already follows from the minor monotonicity of $\mu$ and the fact
that the $1$-skeleton of a $d$-polytope contains $K_{d+1}$ as a minor.}
%\vk{No, řekl bych že \citep{Izmestiev} se dá taky považovat za
%výsledek tohoto typu. Taky dává dolní odhad na $\mu(G)$ a taky dává
%$\mu(K_n)=n-1$.}\martin{Diky za dulezite upozorneni! Prvne technicka: Izmestiev
%dava dolni odhad, ale my davame horni, ne? Bud tuhle vetu muzu smazat, nebo
%muzeme Izmestieva zminit a snazit se odlisit. Treba, ze takove geometricke
%vysledky jsou vzacne pro obecne $\mu$, jedine o cem vime je Izmestiev, ktery
%dava dolni odhad v trochu jine situaci. UPDATE: To co jsem si predtim myslel o
%rovnosti je nesmysl, obratil jsem jednu nerovnost. Ale zase ten Izmestieuv
%vysledek neni nejak zazracny, plyne to mj. z toho, ze grafy hranic $k$-mnohostenu
%maji $K_{k+1}$ jako minor.}
%Tak me jeste napada: nedava ten nas
%vysledek dohromady s nejakou Borsuk-Ulamovou vetou presne rovnost v tom
%Izmestieove pripade? Mnohosten je specialni pripad bunecneho komplexu, a
%hranici ma sferu... Jen nevim, jestli se nejak neposunou hodnoty. Jeste se
%zamyslim.}
%\vk{Komentováno v emailu. Možná by stačilo změnit "bound" na "upper bound"?}
%\martin{Pridano `upper' a footnote. Je to OK, nebo to chceme jeste nejak
%vylepsit?}

On the other hand, we would also like to argue that the parameter $\sigma$ deserves
comparable attention as $\mu$. 

First of all, it provides a much more direct geometric generalization of graph
planarity than the parameter $\mu$; more in a spirit of the Hanani--Tutte type
characterization of graph planarity (see, e.g., ~\citep{schaefer13}). 
%\vk{Geometričtější než
%co?}\martin{No mel jsem implicitne na mysli srovnani s $\mu$. Pokud se Ti zda,
%ze to neplyne z kontextu, tak by to asi slo explicitneji, jen mi ta veta uz
%prisla dostatecne komplikovana} \vk{Navrhuji přidat zelený text.}
%\vk{Dal jsem tu referenci do závorky za "e.g.", aby bylo jasné, že to není originální reference.}

%\martin{TODO: Citovat asi Schaefer: Hanani-Tutte and related
%results.}

Next, it seems that it might be computationally much more tractable to
determine the graphs with $\sigma \leq k$ when compared to graphs with $\mu
\leq k$. From now on, let $\gmu k$ and $\gsigma k$ denote the class of graphs
with $\mu \leq k$ and $\sigma \leq k$ respectively.
Of course, once we fix an integer $k$, there is a polynomial time
algorithm for recognition of graphs in $\gmu k$ and $\gsigma k$ via
the Robertson--Seymour theory \citep{robertson_seymour95, robertson_seymour04} as there is a finite list of
forbidden minors for these classes. 
%for the class of graphs $G$
%satisfying $\mu(G) \leq k$ or $\sigma(G) \leq k$. 
The minors are well known if
$k \leq 4$; however the catch of this approach is
that it seems to be out of reach to find the minors as soon as $k \geq 5$.

Let us focus on the interesting case $k=5$.
%\red{As far as we can say, there does
%not seem to be a simple explicit algorithm determining graphs in $\gmu 5$ 
%(or more in generally $\gmu k$) directly from the definition of $\mu$.}
We are not aware of any explicit algorithm for determining the graphs in $\gmu 5$ in the literature. The best algorithm we could come up with is a PSPACE algorithm based on the existential theory of the reals. (This algorithm recognizes the graphs in $\gmu k$ for general $k\in\N$.)
For completeness we describe it in Appendix~\ref{a:mu}.

On the other hand, there is a completely
different set of tools for recognition of
graphs $G$ from $\gsigma 5$.
According to~\cite[Thm.~30]{sigma}
it is sufficient to verify whether the $2$-skeleton of a so-called
closure of $G$ admits an even mapping into $\R^4$. We do not describe
here a closure of $G$ in general; however, according to the definition
in~\cite{sigma}, it can be chosen in such a way that its $2$-skeleton coincides
with the complex obtained by gluing a disk to each cycle of $G$; let us
denote this complex by $\C^2(G)$. It is in general well known that it can be
determined whether a $2$-complex admits an even mapping to $\R^4$ (even in
polynomial time in the size of the complex). From the point of view of algebraic
topology, this is equivalent to vanishing of the $\Z_2$-reduction of the 
so-called van
Kampen obstruction. An explicit algorithm can be found
in~\citep{matousek_tancer_wagner11} %\martin{Jde to jednoduseji bez zminovani obstrukce?} 
modulo
small modifications caused by the facts that $\C^2(G)$ is not a simplicial
complex and that we work with the $\Z_2$-reduction.
%\footnote{Here we do not go
%into details as this algorithm is not the main contents of the paper---it
%serves as part of the motivation.} 
The algorithm runs 
in time polynomial in size of $\C^2(G)$, which might be exponential in size of
$G$. However, the naive implementation of the algorithm seems to perform many
redundant checks.
By removing some of these redundancies, we can get
an explicit polynomial time certificate for $\sigma(G) > 5$, that is, a
certificate for co\nobreakdash-NP membership.
\ifcomb
A proof of this is provided only in the full version of the present paper~\citep[App.~B]{kaluza_tancer19arxiv} on arXiv. 
\else
A proof of this is given in Appendix~\ref{a:sigma}.
\fi
Optimistically, we may hope that this algorithm could be
adapted to an explicit polynomial time algorithm.

Now, if the conjecture $\gmu 5 = \gsigma 5$ of van der Holst and Pendavingh is
true, then the algorithm above also determines graphs with $\mu \leq 5$.
Theorem~\ref{t:main} gives one implication.

%\martin{Jeste jsem sem dopsal zelenou vetu jako motivaci pres knotless.} 

Similar ideas can perhaps be used for the recognition of graphs from $\gsigma k$
with general $k$, though this requires working with the $\lfloor (k
-1)/2\rfloor$-skeleton of the closure, which is more complicated. (Of course,
the impact on $\mu$ is then more limited due to Theorems~\ref{t:gap} and \ref{t:asymptotic}.)

\paragraph{Overview of our proofs.} Here we briefly overview the key steps in our main proofs. We start with Theorem~\ref{t:main}.
On high level, we follow \citet{CdV_linkless}, who showed that if $G$ is a
linklessly embeddable graph, then $\mu(G) \leq 4$. First we sketch (in our
words) their strategy and then we point out the important differences.

For contradiction, Lov\'{a}sz and Schrijver assume that there is linklessly
embeddable $G$ with $\mu(G) \geq 5$. According to the definition of $\mu$
(given in the next section), there is a certain matrix $M \in \R^{V \times
V}$ of corank $5$
associated to $G = (V,E)$ which witnesses $\mu(G) \geq 5$. 
Given a vector $x\in\R^V$, we denote by $\supp(x)$ the set $\set{v\in V \colon
x_v\neq 0}$. Correspondingly, we define $\supp_+(x):=\set{v\in V\colon x_v> 0}$
and $\supp_-(x):=\set{v\in V\colon x_v<0}$. Then $\ker(M)$, the kernel of $M$, 
%\vk{Jen pro konzistenci: ve zbytku textu jsem psal $\ker(M)$, místo $\ker M$. Tady už jsem to změnil.} 
can be decomposed
into equivalence classes of vectors for which $\supp_+$ and $\supp_-$ coincide.
Each equivalence class is a (relatively open) cone (see Definition~\ref{d:PL}).
Then, by choosing a suitably dense set of unit vectors in each of the cones
and %then \vk{Vynechal bych "then", je tady zbytečné a opakuje se.} 
taking the convex hull, Lov\'{a}sz and Schrijver obtain a
$5$-dimensional polytope $\PP$ such that every relatively open face of
$\partial \PP$ is in one of the cones. 

Given a linkless embedding of $G$ (more
precisely, a flat embedding), it is possible now to define an embedding $f$ of
the $1$-skeleton $\PP^{(1)}$ into $\R^3$ in such a way that for every vertex
$\uu$ of $\PP$, which is also a vector of $\ker (M)$, $f(\uu)$ is
%\red{mapped close
%to a vertex} \vk{asi bych tady přidal "of the embedding"}\martin{Preformuloval
%jsem to vic---je to ted zelene. OK?} \red{of the graph
%$G[\supp_+(\uu)]$ induced by $\supp_+(\uu)$.}
mapped close to a vertex of $\supp_+(\uu)$ (this vertex is embedded in
$\R^3$ by the given linkless/flat embedding of $G$).

Also, for
every edge $\ee = \uu\ww$ of $\PP$, we have $\supp_+(\ee) \supseteq
\supp_+(\uu), \supp_+(\ww)$. If $G[\supp_+(\ee)]$, the subgraph induced by
$\supp_+(\ee)$, is connected for every such
$\ee$, then Lov\'{a}sz and Schrijver pass $f(\ee)$ close to some path
connecting $f(\uu)$ and $f(\ww)$ in $G[\supp_+(\ee)]$. An existence of such $f$
then reveals that the original embedding of $G$ was not linkless via a
Borsuk--Ulam type theorem by \citet{CdV_linkless}, which is the required
contradiction. 

It, however, still remains to resolve the case when some edges $\ee$ do not
satisfy that $G[\supp_+(\ee)]$ is connected. Such edges are called
\emph{broken} edges and it is the main technical part of the proof to take care
of them. %broken edges\vk{nahradil bych opakování "broken edges" slovem "them".}. 
Via structural properties of $G$, including the usage of one of the forbidden
minors for linkless embeddability (see \citep{linkless_minors} for the list of minimal such graphs),
%\vk{Citovat tady
%charakterizaci linkless grafů pomocí minorů z
%\citep{linkless_minors}?}\martin{OK, citace presunuta z komentare do textu.}
Lov\'{a}sz and Schrijver show how to route the broken edges without introducing
new linkings, which again yields the required contradiction.

Our main technical contribution is that we design a strategy how to route
broken edges without any requirements on the structure of $G$. Namely, we show
that if we make several very careful choices in the very beginning when placing
the vertices of $\PP$ as well as if we carefully route the nonbroken edges of
$\PP$, then we are able to make enough space for broken edges as well. The
important property is that when $\FF$ and $\FF'$ are (so-called) antipodal
faces, then the edges of $\FF$ and the edges of $\FF'$ are routed close to
disjoint subgraphs. (The precise statement is given by
Proposition~\ref{p:lemma37}, and we actually map $\PP^{(1)}$ into the graph $G$.)

Now, we could aim to conclude in a similar way as Lov\'{a}sz and Schrijver via
a suitable Borsuk--Ulam type theorem, which would require to extend the map to
higher skeletons and to perturb it a bit. However, we instead use a lemma of
\citet{sigma} tailored to such a setting, which they used in the proof of the
inequality $\mu(G) \leq \sigma(G) + 2$ (see the proof of
Proposition~\ref{p:sigma_ex_rep_new}).
%\martin{To je tedy tak, ze ukazuji $\lambda(G) \leq \sigma(G)$. Tohle je asi
%jedine potencialni misto, kdy je mozne prirozene zminit $\lambda$. Ale zase to
%ani tady neni nutne. Chceme to $\lambda$ teda nejak vubec zminovat v uvodu?
%Nemam v tom jasno. Je to pojem navic, ktery primo nevyuzijeme, ale zase pro nas
%predstavoval kus motivace/navodu.}
%\vk{Asi bych se přiklonil k tomu nezmiňovat $\lambda$ v úvodu, ale až v
%poslední sekci. Pokud budeš souhlasit, musí se odstranit $\lambda$ ze znění
%Theoremu~\ref{t:asymptotic} výše.}\martin{Ano, souhlasim. Jak jsem vyresil
%$\lambda$ u vety~\ref{t:asymptotic} pisu primo tam.}

Last but not least, instead of working directly with matrices from the
definition of $\mu$, we abstract their properties required in the proof of
Theorem~\ref{t:main} into a notion of \emph{semivalid representation};
see Definition~\ref{def:ex_represent}. (The main difference is that we replace the so-called Strong Arnold hypothesis by more combinatorial properties.) 
This abstraction turns out to be very useful in the proof of Theorem~\ref{t:gap}
because then it is possible to provide lower bounds on $\sigma$ also with aid
of matrices not satisfying the Strong Arnold hypothesis, which is essential if
we want to separate $\mu$ and $\sigma$. 
%\vk{Tady bych raději místo "if we want to separate" napsal
%něco jako "for our separation of". Ta původní formulace jakoby naznačuje, že z
%matic splňujících SAH žádná separace vyjít nemůže.}\martin{Klidne se to tak
%muze preformulovat. Ja to i puvodne myslel tak, ze se SAH separace vyjit
%nemuze, protoze by to vyslo stejne jako pro $\mu$. Prehlizim neco?}
%\vk{Máš pravdu, pak ale navrhuji smazat slovo "also" (červeně)}
%\martin{No duvod pro `also' byl, ze veta~\ref{t:main} v podstate poskytuje
%dolni odhad na $\mu$ pomoci matic splnujicich SAH (takze delame oboji). Ale
%jestli se Ti to also nelibi, tak ho klidne smaz (vcetne okolnich komentaru).}

Recall that by $H_q$ we denote the incidence graph of a finite projective plane of order $q$.
We add a short description of how our bound on $\sigma(G)$ is used in the
proof of Theorem~\ref{t:gap}; here we only sketch how to show a slightly weaker
result $\sigma(H_3) \geq 11$, discussed below the statement of
Theorem~\ref{t:gap}. We first note that semivalid representations are
defined as certain linear subspaces of $\R^V$ and we will introduce a parameter
$\eta(G)$ which is the maximal dimension of a semivalid representation. We will
also show $\mu(G) \leq \eta(G) \leq \sigma(G)$, where $\mu(G) \leq \eta(G)$
follows easily from the known results on $\mu$ whereas showing the inequality
$\eta(G) \leq \sigma(G)$ is the core of the proof of Theorem~\ref{t:main}.

Now let us consider a matrix $M_3$ which is a suitable shift of the
adjacency matrix of $H_3$. This matrix satisfies $\corank(M_3) = \dim \ker (M_3) =
12$ and $\ker M_3$ is `almost' a semivalid representation of $G$. Namely, by a
trick that we learnt from~\citet{Pendavingh_separation} we can
find a codimension 1 subspace $L$ of $\ker(M)$ which is a semivalid
representation. This shows $\eta(H_3) \geq 11$ and the key inequality
$\sigma(G) \geq \eta(G)$ gives the required bound $\sigma(H_3) \geq 11$.

The proof of Theorem~\ref{t:asymptotic} follows the same high-level strategy as
the proof of Theorem~\ref{t:gap}, except we do not work there with a semivalid
representation, but rather with a so-called \emph{valid representation}, which
is a concept used to define the parameter $\lambda$ (see
Subsection~\ref{ss:semivalid}).
We use a simple general position argument to show that if $G$ has a low maximum
degree, then a large subspace of $\ker(M_q)$ has to be a valid
representation of $G$ where $M_q$ is, in analogy to the previous case, 
a suitable shift of the adjacency matrix of $H_q$.  

\paragraph{Organization.} In Section~\ref{s:representations} we overview (or introduce)
various representations of graphs and establish some of their
properties. Then we prove Theorem~\ref{t:main} in Section~\ref{s:mu_sigma} and
Theorems~\ref{t:gap} and~\ref{t:asymptotic} in Section~\ref{s:gaps}.

%For a set $S\subset V$ we denote by $x_S$ the restriction of the vector $x$ to the subset $S$, that is, $x_S:=(x_v)_{v\in S}$. Similarly, for a matrix $M\in\R^{V\times V}$ we denote by $M_S$ the submatrix of the form $\br*{M_{u,v}}_{u,v\in S}$.

\section{Representations of graphs}
\label{s:representations}
\subsection{The Colin de Verdi\`ere graph parameter}
\label{s:CdV}
%\martin{Pokud se schvali uvod, tahle sekce bude potrebovat zkraceni, vyrazeni
%toho, co uz bude v uvodu. Zatim jsem nic nemenil.}

If not stated otherwise, we work with a graph $G=(V,E)$.
%such that
%$V=[n]$.\martin{Pouziva se vubec nekde, ze $V = [n]$?}
%\vk{Myslím, že nakonec ani ne. Jen o kousek níže se používaly indexy $i,j$, což nejsou obvyklá písmena pro vrcholy. Tak jsem je přepsal na $u,v$ a tahle konvence se může smazat.}
We use the usual graph-theoretic notation $N(v)$ for all vertices adjacent to
$v\in V$ and $N(S)$ for all vertices in $V\setminus S$ adjacent to a vertex in
$S\subseteq V$. Moreover, we use $\indg{G}{S}$ to denote the subgraph of $G$ induced by $S$.
%\red{Given a vector $x\in\R^V$, we denote by $\supp(x)$ the set
%$\set{v\in V \colon x_v\neq 0}$. Correspondingly, we define
%$\supp_+(x):=\set{v\in V\colon x_v> 0}$ and $\supp_-(x):=\set{v\in V\colon
%x_v<0}$.}
%\martin{Chceme opakovat/pripomenout definici supportu v porovanani s
%uvodem. Snad to neni nutne?}
%\vk{Asi ne. Zakomentoval jsem to.}
For a set $S\subset V$ we denote by $x_S$ the restriction of the
vector $x$ to the subset $S$, that is, $x_S:=(x_v)_{v\in S}$.
%Similarly, for a
%matrix $M\in\R^{V\times V}$ we denote by $M_S$ the submatrix of the form
%$\br*{M_{u,v}}_{u,v\in S}$.
%\vk{Zdá se mi, že notace $M_S$ se už nikde nepoužívá.}
%\martin{Snad OK. Hledal jsem vyskyty `M\_' a neobjevil zadny s $M_S$, ktery by
%jeste byl aktualni.}

Let $\mc{M}(G)$ be the set of \emph{symmetric} matrices $M$ in $\R^{V\times V}$ satisfying
\begin{enumerate}[(i)]
	\item $M$ has exactly one negative eigenvalue of multiplicity one,
	\item for any $u\neq v\in V$, $uv\in E$ implies $M_{uv}<0$ and $uv\notin E$ implies $M_{u,v}=0$.
\end{enumerate}
The matrices satisfying only the second of the properties above are sometimes called \emph{discrete Schrödinger operators} in the literature.

Note that there is no condition on the diagonal entries of $M$. Despite this, a
part of the Perron--Frobenius theory is still applicable to $M\in\mc{M}(G)$,
assuming that $G$ is connected (if not, the same reasoning can be applied
component-wise). This is because the matrix $-M+cI_V$, where $I_V$ denotes the
identity matrix of size $V\times V$, has nonnegative entries for $c>0$ large
enough. Since this transformation preserves all eigenspaces, the
Perron--Frobenius theorem implies that the smallest eigenvalue of $M$ has
multiplicity one and the corresponding eigenvector is strictly positive (or
strictly negative). For instance, as $M$ has an orthogonal eigenbasis,
this implies that every nonzero vector $x\in\ker(M)$ must have both
$\supp_+(x)$ and $\supp_-(x)$ nonempty.
%; this is used several times later on.

%\vk{Přidal bych odstavec zeleně. Účelem je vyřešit několik poznámek v článku.}
%\martin{Vesmes OK. Jenom nerozumim, jak presne z tohohle plyne, ze oba
%$\supp_+$ i $\supp_-$ jsou neprazdne (bez nejake dalsi uvahy). Nebylo by prece
%jen lepsi se odkazat na bod 1 v nasledujicim lemmatu. Tim se to sice trochu
%schova pod koberec---udelal to nekdo jiny. Ale zase se na to pujde jasneji
%odkazat (a asi taky neni uplne nutne opakovat poradne zduvodneni).}
%\vk{Zkusil jsem to ještě jednou dovysvětlit tou zelenou vsuvkou. Podle mě je to natolik jednoduché, že odkaz na nějaké lemma to spíše zatemňuje. Přijde mi užitečné, když si tady čtenář uvědomí, že matice z $\mc{M}(G)$ mají spoustu speciálních vlastností.}

%\red{If $M\in\mc{M}(G)$ satisfies the so-called Strong Arnold hypothesis (SAH), which asserts that
%$$
%MX=0 \qquad\Longrightarrow\qquad X=0
%$$
%for every symmetric $X\in\R^{V\times V}$ such that $X_{u,v}=0$ whenever $u=v$
%or $uv\in E$, we will call such an $M$ a \emph{CdV-matrix} of $G$.
%}
A matrix $M\in\mc{M}(G)$ satisfies the so-called \emph{Strong Arnold
hypothesis} (SAH),
if
\[
MX=0 \qquad\Longrightarrow\qquad X=0
\]
for every symmetric $X\in\R^{V\times V}$ such that $X_{u,v}=0$ whenever $u=v$
or $uv\in E$.
%\martin{Pojem `CdV-matrix' ma, zda se mi, jeden vyskyt. V takovem pripade bych ten pojem
%nezavadel, ale rekl jen matice z $\mc{M}(G)$ splnujici SAH. Muze byt?}
%\vk{Souhlasím. Přeformuloval jsem to výše dle tohoto návrhu.}
%\martin{Zakomentoval jsem starsi verzi a odarvil novou. Pokud nevadi, to
%pasivum mi prislo az velmi odtazite, tak je vyse navrh bez nej. Taky jsem
%`Strong Arnold hypothesis' dal do `emph'.}
%\martin{Navrhuju cerveny text zmenit na zeleny kvuli uvodu. Taky jsem vynechal
%zkraceni `CdV parameter', ktere se myslim vubec nepouziva.}
The \emph{Colin de Verdi\`ere graph parameter} $\mu(G)$ is defined as the
maximum of $\corank(M)$ over matrices $M\in\mc{M}(G)$ satisfying SAH.

\subsection{Semivalid representations of graphs}
\label{ss:semivalid}
%\martin{Pokud tohle bude podsekce, a jeste nasledujici tri podsekce zustanou
%podsekcemi, tak by se mel podle me prejmenovat nazev sekce~\ref{s:CdV}, protoze
%CdV parametr je jen drobny aspect. Zvazoval bych neco jako `The Colin de
%Verdi\`{e}re parameter and representations of graphs' Byt asi hrozi, ze to
%pretece.}
%\vk{Rozhodně netrvám na tom, že tohle bude podsekce. Někde jsi v nějaké komentáři zmiňoval, že chceš z podsekcí udělat sekce. Nic proti tomu nenamítám. Klidně to udělám já, jen mi teď není jasné, jestli to tak stále chceš udělat nebo ne?}
%\martin{Viz komentar o podsekci drive.}

We collect some of the easy, but important properties of matrices in
$\mc{M}(G)$ in the following lemma. The proofs can be found, for instance, in a
survey by \citet[Sec.~2.5]{CdV_main}\footnote{A global convention of
\citep[Sec.~2.5]{CdV_main} is that the matrices $M$ considered there satisfy
SAH. However, SAH is not used in the proof of the properties asserted in Lemma~\ref{lemma:basic}.}.

\begin{lemma}\label{lemma:basic}
Let $G=(V,E)$ be a connected graph and $M\in\mc{M}(G)$. Let $x\in\ker(M)$ be nonzero, then
\begin{enumerate}
	\item[{\normalfont(\mylabel{lemma:basic_1}{i})}] $N(\supp(x))=N(\supp_-(x))\cap N(\supp_+(x))$,
	\item[{\normalfont(\mylabel{lemma:basic_2}{ii})}] if $\indg{G}{\supp_+(x)}$ is disconnected, then there is no edge between $\supp_+(x)$ and $\supp_-(x)$, and moreover, for every connected component $C$ of $\indg{G}{\supp(x)}$ we have $N(C)=N(\supp(x))$,
	\item[{\normalfont(\mylabel{lemma:basic_3}{iii})}] if $\supp(x)$ is inclusion-minimal among nonzero vectors in $\ker(M)$, then both graphs $\indg{G}{\supp_+(x)}$ and $\indg{G}{\supp_-(x)}$ are nonempty and connected.\footnote{This part is originally due to \citet{vdHolst_planar}.}
\end{enumerate}
\end{lemma}

Motivated by the parameter $\mu$, \citet{lambda} introduced the invariant $\lambda(G)$ defined as follows. We say that a linear subspace $X\subseteq \R^V$ is a \emph{valid representation} of the graph $G$, if for every nonzero $x\in X$ the graph $G[\supp_+(x)]$ is nonempty and connected. Then $\lambda(G)$ is defined as the maximum of $\dim(X)$ over all valid representations $X$ of $G$.
%\vk{Přesunul jsem definici validní reprezentace ze začátku Sekce~\ref{s:gaps}
%tady, protože teď mi to tady dává větší smysl.}\martin{OK.}

Among other properties, \citet{lambda} proved that $\lambda$ is minor monotone and characterized the classes of graphs with $\lambda(G)\leq 1,2,3$. From this characterization it follows that the parameters $\lambda$ and $\mu$ differ already for those small values. In general, $\lambda$ can be both greater or smaller than $\mu$ (see \citep{lambda, Pendavingh_separation}).

Somewhat analogously to the notion of a valid representation, we introduce the following definition:
%\vk{Do nové definice není zahrnuta vlasnost \eqref{lemma:basic_1} z Lemma~\ref{lemma:basic}, protože se bez ní obejdu a navíc tahle vlastnost mi pak dělá problémy, když chci ukázat, že každá valid reprezentace je i semivalid reprezentace. Tahle vlastnost je ale zachována pro rozbité vektory, což stačí.}
%\vk{UPDATE: Definice je zúžena jen na souvislé grafy. Vysvětlení je u definice
%parametru $\eta$ níže.}\martin{OK.}
%\martin{Budeme chtit tu definici `extended representation' teda nejak
%prejmenovat, aby se to nepletlo s predchozi verzi?}
%\vk{Asi by to bylo lepší. Otázka je, jak to pojmenovat.}
%\vk{Jeden návrh bych měl: já tu současnou definici vidím hlavně jako zobecnění definice validní reprezentace, tak by se mi líbilo, kdyby to na ni nějak odkazovalo. Do budoucna bych si uměl představit nějakou hierarchii takových reprezentací, které by se lišily tím, kolik komponent souvislosti by bylo dovoleno. Z tohoto pohledu by definice validní reprezentace mohla být něco jako ``validní 1-reprezentace'', ta naše definice níže by pak byla ``validní 2-reprezentace'', někdy v budoucnu by se pak daly prozkoumat třeba ``validní $k$-reprezentace''. Co o tom soudíš?}
\begin{defn}[Semivalid representation]\label{def:ex_represent}
Given a connected graph $G=(V,E)$ we call a linear subspace $L\subseteq\R^V$ a \emph{semivalid representation}\footnote{In the first version of the
present work \citep{kaluza_tancer_soda}, we were using a notion of an \emph{extended
representation} with a very similar definition: it had the same properties as in the
current definition, but in addition it was assumed to lie in $\ker(M)$ of some
$M\in\mc{M}(G)$. We found this extra assumption somewhat unpleasant, thus
we spent an extra effort on removing it from this key definition. But this doesn't mean that the proofs
of the main results are more complicated---only a few details are slightly
different.}
of $G$ if, for every nonzero $x\in L$,
\begin{enumerate}
	\item[{\normalfont(\mylabel{def:ex_rep_0}{i})}] both $\supp_+(x)$ and $\supp_-(x)$ are nonempty,
	\item[{\normalfont(\mylabel{def:ex_rep_1}{ii})}] the graph $\indg{G}{\supp_+(x)}$ is either connected, or $\indg{G}{\supp_+(x)}$ has two connected components and $\indg{G}{\supp_-(x)}$ is connected,
	\item[{\normalfont(\mylabel{def:ex_rep_2}{iii})}] if $x$ has inclusion-minimal support in $L$, both $\indg{G}{\supp_+(x)}$ and $\indg{G}{\supp_-(x)}$ are nonempty and connected,
	\item[{\normalfont(\mylabel{def:ex_rep_4}{iv})}] if $\indg{G}{\supp_+(x)}$ is disconnected, then there is no edge between $\supp_+(x)$ and $\supp_-(x)$, and moreover, for every connected component $C$ of $\indg{G}{\supp(x)}$ we have $N(C)=N(\supp(x))$.
\end{enumerate}
\end{defn}

%\martin{Jina poznamka: Odstavec nize a lemma pod nim se snazi tak nejak neprimo
%rict, ze kdyz $M$ splnuje navic SAH, potom je to semivalid representation.
%Navrhoval bych to zkusit trosku preformulovat/preusporadat, aby se tohle
%explicitne reklo jako motivace. Klidne se o to pokusim ja, pokud budes
%souhlasit.}\vk{Trochu jsem tu formulaci níže rozšířil, aby to tam bylo explicitně zmíněno. Klidně to přeformuluj, pokud s tím nejsi spokojen.}

We will use semivalid representations of $G$ as a substitute for $\ker(M)$ in case we want to work with $M$ not necessarily satisfying SAH. This is enabled by the following lemma taken from \citet{Pendavingh_separation}, which together with Lemma~\ref{lemma:basic} implies that the kernel of $M\in\mc{M}(G)$ satisfying SAH defines a semivalid representation of $G$:
\begin{lemma}[{\citep[Lem.~3]{Pendavingh_separation}}]\label{lemma:Pendavingh}
Let $G$ be a connected graph and $M\in\mc{M}(G)$. Let $x\in\ker(M)$ and set
\[
D:=\set{y\in\ker(M)\colon\supp(y)\subseteq\supp(x)}.
\]
If $\indg{G}{\supp(x)}$ is disconnected, it has exactly $\dim(D)+1$ connected components. If, in addition, $M$ satisfies SAH, then $\dim(D)\leq 2$.
\end{lemma}

Similarly to the graph parameter $\lambda$ introduced by \citet{lambda}, we
define an analogous parameter $\eta(G)$:
%\vk{Necitovat tady radši SODA verzi? (Proceedings už jsou publikované)}
%\vk{Pokud extended reprezentaci přejmenujeme, je to tady potřeba zmínit.}
%\martin{Souhlas s obojim, presna zmena zavisi tedy na finalnim jmenu.}
%\martin{Pozmenil jsem poznamku pod carou i s nejakym vysvetlenim ke zmene
%definice. Ale jeste si myslim, ze by sa tahle poznamka mela presunout primo k
%definici semivalidni reprezentace. Souhlasis?}
%\vk{Ano, provedeno.}
\begin{defn}
Let $G$ be a graph. If $G$ is connected, we define
\[
\eta(G):=\max\set{\dim(L)\colon L\text{ is a semivalid representation of }G}.
\]
For convenience, we also extend the definition to disconnected graphs $G$. If
  $G$ has at least one edge, then we define
\[
\eta(G):=\max_C\eta\br*{\indg{G}{C}},
\]
where the maximum is taken over connected components $C$ of $G$.
If $G$ is disconnected and does not have any edge, then we set $\eta(G):=1$.

\end{defn}
Lemmas~\ref{lemma:basic} and \ref{lemma:Pendavingh} show that $\mu(G)\leq\eta(G)$ for every connected graph $G$.
The definition of $\eta$ for disconnected graphs is chosen in a way that agrees precisely with the behavior of $\mu$:
In \citep[Thm.~2.5]{CdV_main} it is shown that $\mu(G)$ is equal to the maximum of $\mu$ over the connected components of $G$ if $G$ has at least one edge. Moreover, it is easy to see that $\mu$ of the empty graph on $n\geq 2$ vertices is $1$ (or see, e.g., \citep[Sec.~1.2]{CdV_main}).

Comparing the definitions of valid and semivalid representations, it is clear that every valid representation is also a semivalid representation.
%Since for disconnected graphs $\lambda$ behaves exactly as $\mu$, which is easy to see directly from the definition of $\lambda$, we get that $\eta(G)$ is always an upper bound on $\lambda(G)$ for any graph $G$.
Since for disconnected graphs $\lambda$ exhibits exactly the same type of
behavior as $\mu$ and $\eta$ with respect to the connected components, which is easy to see directly from the definition of $\lambda$, we get that $\eta(G)$ is always an upper bound on $\lambda(G)$ for any graph $G$.
Summarizing, we get the following:
%\martin{Neni veta, ze $\lambda$ se chova presne jako $\mu$ pro nesouvisle
%grafy, trochu zavadejici? Dalo by se ji mozna rozumet tak, ze pro nesouvisle grafy se
%$\lambda$ a $\mu$ vzdy rovnaji.}
%\vk{Asi to bylo opravdu zavádějící, pokusil jsem se to opravit. Díky!
%Je to teď lepší?}
%\martin{OK. Pridal jsem `and $\eta$' za $\mu$, protoze se $\lambda$ hlavne
%srovnava s $\eta$. Asi je to jen hnidopisstvi, tak
%doufam, ze Ti nevadi.}

\begin{obs}\label{obs:mu_eta}
For every graph $G$ it holds that $\max\set{\mu(G),\lambda(G)}\leq\eta(G)$.\qed
\end{obs}

%\martin{Je tohle tvrzeni opravdu `dusledek'? Asi to neni jednoznacne dane, ale
%mel jsem pocit, ze `dusledek' se pouziva predevsim, kdyz plyne z nejakeho
%predchoziho (znaceneho) tvrzeni. Mozna by to mohlo byt tvrzeni/lemma se
%symbolem pro konec dukazu.}
%\vk{OK, změnil jsem na pozorování.}

\subsection{Topological preliminaries}\label{s:prelim}
%\martin{Usporadani do podsekci je docasne pro prehlednost ted. To pak muzeme
%preusporadat.}
\paragraph{Polyhedra.} A set $\tau' \subset \R^k$ is a \emph{closed}
(convex) \emph{polyhedron} if it is an intersection of finitely many closed
half-spaces. A \emph{closed face} of a polyhedron $\tau'$ is a subset $\eta' \subseteq
\tau'$%\vk{Tahle definicie implikuje, že stěna je rel.\ uzavřená. Pokud chápu
%konvenci o odstavec níže správně, tak tam se naopak říká, že stěny jsou rel.
%otevřené (s čímž bych souhlasil, jen byto mělo být konzistentní).}\martin{Moje
%predstava byla: Prvne se rekne standardni definice steny. Tu nevim jak zavest
%bez uzavrenosti, kdyz se to chce pomoci pruniku. Potom se venuju relativne
%otevrenemu pripadu nize, a nakonec se rekne dulezita konvence, ze uvazujeme ten
%relativne otevreny pripad. Specialne jsou ty definice poskladane tak, ze stena
%muze byt relativne otevrena i uzavrena. Jestli mas konkretni navrh, jak to
%poradne zavest jinak, tak sem s nim :-)}
%\vk{Myslím, že by stačilo, kdyby se definice výše pojmenovala \emph{closed
%face} místo jenom \emph{face}.}\martin{Jsi ochotny pristoupit na `(closed) face'? Face
%je dost standardni pojem a mohlo by to byt podivne nez se ctenar dostane ke
%konvenci. Doufam, ze to projde, takze uz jsem to zmenil.}
such that there exists a hyperplane $h$ satisfying that $\eta' = h \cap \tau'$ and
$\tau'$ belongs to one of the closed half-spaces determined by $h$. A
\emph{relatively open polyhedron} is the relative interior $\tau$ of a closed
polyhedron $\tau'$ (the relative interior is taken with respect to the affine hull
of $\tau'$).
%\vk{Tady by se asi měla definovat faseta, když je tu definice stěny. Taky by se někde mělo říct, co jsou to antipodální stěny a případně i co jsou paralelní stěny.}
%\martin{Urcite souhlasim, ze na nejakem vhodnem miste by se mely ant. steny
%popr. paralelni steny uvest. (Mozna to staci, az se budou pouzivat.) Tady asi
%vhodne misto neni, protoze se tady zatim vubec neresi ani mnohosteny. A myslim,
%ze ani tady neni vhodne misto pro fasety (radeji u komplexu? u prvniho vyskytu?). 
%Mimochodem mnohosteny jsem v tetosekci vubec nerozebiral a predpokladal jsem
%aspon zakladni znalosti u ctenare. Steny jsem zavedl hlavne kvuli konvenci s
%relativnim vnitrkem.}
%\vk{OK, je mi celkem jedno, kde se ty pojmy zavedou. Tady (ale i třeba na začátku sekce 2.4) mi to dává smysl, protože pojmy faseta, paralelní a antipodální stěny dávají smysl i pro polyédr, nejen pro mnohostěn. A navazovalo by to na definici stěny. Pokud chceme předpokládat základní znalosti u čtenáře (což myslím můžeme), je otázka, jestli by se to celé výše nemohlo zredukovat na poznámku, že polyédry a jejich stěny uvažujeme relativně otevřené, pokud není řečeno jinak.)}
%\martin{Facety se ted vyskytuji pouze v definici stellarni subdivize, takze
%jsem je upresnil jen tam.}

\paragraph{Important convention.} In the sequel, when we say \emph{polyhedron}, we mean
relatively open polyhedron. This is nonstandard, but it will be very
convenient for our considerations. Given a polyhedron $\tau$, by $\cl{\tau}$ we
denote the closure of $\tau$, that is, the corresponding closed polyhedron. 
We also say that a (relatively open)
%\vk{Pro zdůraznění bych tady ještě přidal do závorky
%(relatively open)}\martin{O tom jsem taky premyslel i predtim, nezs to
%napsal :-) Ja pak provedl uvahu, ze je dobre zacit zvykat ctenare v kontextu,
%kdy je to vicemene jasne. Ale pokud chces zduraznit, tak to klidne pridej.}
polyhedron $\eta$ is a face of $\tau$ if $\cl \eta$ is a %\vk{přidat "closed"?
%viz můj návrh výše}\martin{Pridano v zavorce, viz poznamka vyse. Ale pokud se
%Ti to nebude libit, tak zavorky odstranim} 
closed face of $\cl \tau$.

\paragraph{Polyhedral complexes.} A \emph{polyhedral complex} is a collection
$\C$ of polyhedra satisfying:
\begin{enumerate}[(i)]
\item
If $\tau \in \C$ and $\eta$ is a face of $\tau$, then $\eta \in \C$.
\item
  If $\theta, \tau \in \C$, then $\cl \theta \cap \cl \tau$ is a closed face of $\cl
  \theta$ as well as a closed face of $\cl \tau$.
\end{enumerate}
%\vk{Tady se používá "face" jak jako "closed face", tak jako "rel.\ open face".
%Dle mého návrhu výše, navrhuju psát pro uzavřenou stěnu "closed
%face".}\martin{OK. Opet v zavorce...}

The \emph{body} of a polyhedral complex $\C$ is defined as $|\C| := \bigcup \C$. Due to our convention that we consider relatively open polyhedra, $|\C|$ is a
disjoint union of polyhedra contained in $\C$.

Given a polyhedron $\tau$, by $\partial \tau$ we denote the boundary of $\tau$.
With a slight abuse of notation, depending on the context, this may be
understood both as a polyhedral complex formed by the proper faces of $\tau$ as well as
the topological boundary of $\tau$, that is, the body of the former one. 

The \emph{$k$-skeleton} of a polyhedral complex $\C$ is the subcomplex $\C^{(k)}$
consisting of all faces of $\C$ of dimension at most $k$. 

In our considerations, we will need two special classes of polyhedra:
simplicial complexes and fans.

\paragraph{Simplicial complexes.} A polyhedral complex is a \emph{simplicial
complex} if each polyhedron in the complex is a simplex. (Consistently with our
convention above, by a simplex we mean a relatively open simplex.)

\paragraph{Fans.} 
A cone is a polyhedron $\alpha\subseteq\R^k$ such that $rx\in\alpha$ whenever
$x\in\alpha$ and $r\in(0,\infty)$.
A polyhedral complex $\F$ is a \emph{fan} if each polyhedron
in $\F$ is a cone, and moreover, if $\F$ contains a nonempty
polyhedron, then $\F$ contains the origin as a polyhedron.
A fan is \emph{complete} if $|\F| = \R^k$.

%\martin{Maybe: define simplicial cone here?}

\paragraph{Subdivisions.} Let $\C$ be a polyhedral complex. A polyhedral
complex $\D$ is a \emph{subdivision} of $\C$ if $|\C| = |\D|$ and for every $\eta \in \D$, there is $\tau$ in $\C$ containing $\eta$.

\paragraph{Fans and polytopes.} By a \emph{polytope} we mean a bounded polyhedron. Let $\PP \subseteq \R^k$ be a polytope such
that the origin is in the interior of $\PP$. Then $\PP$ defines a complete
fan $\F(\PP)$ formed by the cones over the proper faces of $\PP$ (plus the empty set). Again,
we consider the faces of $\PP$ relatively open. With a slight abuse of
terminology, we say that $\PP$ \emph{subdivides} a fan $\F'$ if $\F(\PP)$ subdivides
$\F'$; see Figure~\ref{f:subdivided_fan}.

\begin{figure}
\begin{center}
\includegraphics{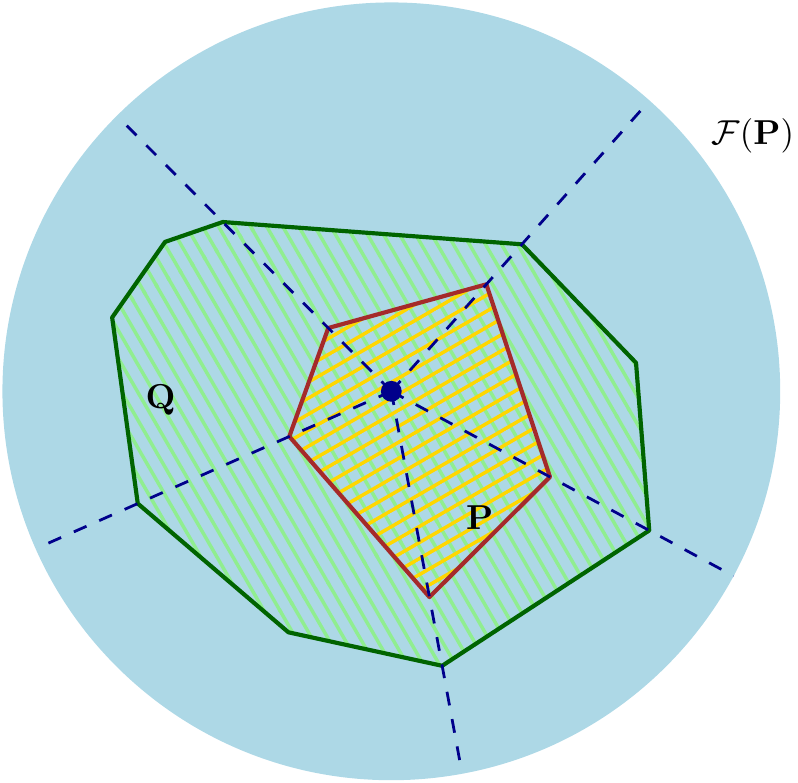}
  \caption{A polytope $\PP$, the fan $\F(\PP)$ and a polytope $\QQ$ subdividing
  $\F(\PP)$.}
\label{f:subdivided_fan}
\end{center}
\end{figure}

\paragraph{Barycentric subdivisions.} Now let $\K$ be a simplicial complex. For
every nonempty simplex $\tau \in K$ let $b_\tau$ be the barycenter of $\tau$.
For two faces $\eta$ and $\tau$ of $\K$, let $\eta \prec \tau$ denote that
$\eta$ is a proper face of $\tau$.
The barycentric subdivision of $\K$, denoted $\sd \K$, is a simplicial complex
obtained so that for every chain $\Gamma = \theta_1 \prec \theta_2 \prec
\cdots \prec \theta_m$ of nonempty faces of $\K$ we add a simplex, denoted
$\Delta(\Gamma)$, with vertices $b_{\theta_1}, \dots, b_{\theta_m}$ into $\sd
\K$. It is well known that $\sd \K$ subdivides $\K$. In particular,
$\Delta(\Gamma) \subset \theta_m$. 

%\martin{Starsi verze pozorovani a dukaz zakomentovana, aby se k tomu dalo
%kdyztak snadno vratit.}

\begin{obs}
  \label{o:edge_in_edge}
  Let $\K$ be a simplicial complex and $\Delta$ be a simplex of the barycentric
  subdivision $\sd \K$. Let $\Delta_1$ and $\Delta_2$ be two faces of $\Delta$
  and $\eta_1 \supseteq \Delta_1$ and $\eta_2 \supseteq \Delta_2$ be two faces of
  $\K$. Then either $\eta_1$ is a face of $\eta_2$ or $\eta_2$ is a face of
  $\eta_1$.  \end{obs}

\begin{proof}
The face $\Delta$ corresponds to a chain $\Gamma = \theta_1 \prec \cdots \prec
  \theta_m$ of faces of $\K$. 
% 
%For the first item, an edge of $\Delta$ corresponds to a subchain
%  $\theta_i \prec \theta_j$ and it is contained in an edge of $\K$ only if
%  $\theta_i$ is a vertex of $\K$ and $\theta_j$ is an edge of $\K$. This may
%  happen only if $(i,j) = (1,2)$.
%
%For the second item, we get that 
  Then $\Delta_1$ corresponds to a subchain $\Gamma_1$ of $\Gamma$ with maximal face
  $\theta_i$ (for some $i$). Then $\theta_i$ is the (unique) face of $\K$ containing $\Delta_1 =
  \Delta(\Gamma_1)$. Therefore $\eta_1 = \theta_i$. 
  Similarly, $\eta_2 = \theta_j$ for some $j$, from which the
  conclusion follows.
\end{proof}

%\begin{obs}
%  \label{o:edge_in_edge}
%  Let $\K$ be a simplicial complex and $\Delta$ be a simplex of the barycentric
%  subdivision $\sd \K$. Then 
%  
%\begin{enumerate}  
%\item[\mylabel{o:edge_in_edge_1}{1}.]
%  there is at most one edge of $\Delta$ contained in an
%  edge of $\K$, and
%\item[\mylabel{o:edge_in_edge_2}{2}.] if $\Delta_1$ and $\Delta_2$ are two faces of $\Delta$ and $\eta_1 \supseteq
%  \Delta_1$ and $\eta_2 \supseteq \Delta_2$ two faces of $\K$,
%    then either $\eta_1$ is a face of $\eta_2$ or $\eta_2$ is a face of  $\eta_1$.
%\end{enumerate}
%\end{obs}
%
%\begin{proof}
%The face $\Delta$ corresponds to a chain $\Gamma = \theta_1 \prec \cdots \prec
%  \theta_m$ of faces of $\K$. 
%  
%  
%For the first item, an edge of $\Delta$ corresponds to a subchain
%  $\theta_i \prec \theta_j$ and it is contained in an edge of $\K$ only if
%  $\theta_i$ is a vertex of $\K$ and $\theta_j$ is an edge of $\K$. This may
%  happen only if $(i,j) = (1,2)$.
%
%For the second item, we get that $\eta_1 = \theta_i$ for some $i$. Indeed,
%  $\Delta_1$ corresponds to a subchain $\Gamma_1$ of $\Gamma$ with maximal face
%  $\theta_i$. Then $\theta_i$ is the (unique) face of $\K$ containing $\Delta_1 =
%  \Delta(\Gamma_1)$. Similarly, $\eta_2 = \theta_j$ for some $j$, from which the
%  conclusion follows.
%\end{proof}

%\martin{Mozna nekde definovat uzavrenou hvezdu nize}

Before we state the next lemma, we introduce two more well-known notions. Let
$\K$ be a simplicial complex and $|\L|$ be the body of some subcomplex
$\L$ of $\K$. We define the
\emph{simplicial neighborhood} of $|\L|$ in $\K$ as \[\Nc(|\L|, \K) := \{\eta \in
\K\colon \eta \subset \cl \tau \hbox{ for some }\tau \hbox{ with } \cl \tau
\cap |\L| \neq \emptyset\}.\] If $\L$ consists of a single vertex $a$, then the
simplicial neighborhood is known as (closed) \emph{star} of $a$ in $\K$,
denoted by $\st(a; \K)$.

\begin{lemma}
\label{l:derived_neighborhoods}
Let $\K$ be a simplicial complex and let $\L_1, \L_2$ be two subcomplexes of
  $\K$ with $|\L_1| \cap |\L_2| = \emptyset$. Let $a$ be a vertex of the second
  barycentric subdivision $\sd^2 \K$. Then the closed star $\st(a; \sd^2 \K)$
  cannot intersect both $|\L_1|$ and $|\L_2|$. 
\end{lemma}

\begin{proof}
%  Let $\Nc_i := \{\eta \in \K\colon \eta \subset \cl \tau \hbox{ for some
%  }\tau \hbox{ with } \cl \tau \cap |\L_i| \neq \emptyset\}$ be the neighborhood
% of $\L_i$. 
  The closed star $\st(a; \sd^2 \K)$ intersects $|\L_i|$ only if $a$
  belongs to $\Nc(|\L_i|, \sd^2 \K) = \Nc(|\sd^2 \L_i|, \sd^2 \K)$.
  %\footnote{Note that $\L_i$
  %is not a subcomplex of $\sd^2 \K$ but $|\L_i| = |\sd^2 \L_i|$; thus,
 % $\Nc(|\L_i|, \sd^2 \K)$ is well defined. The main reason why we use
 % the notation $\Nc(|\L_i|, \sd^2 \K)$ rather than $\Nc(|\sd^2 \L_i|, \sd^2 \K)$ is that we want to keep the notation as
 % close as possible to the reference~\citep{tancer-tonkonog13} given below.} 
  The lemma follows from the fact that
  $\Nc(|\L_1|, \sd^2 \K)$ and $\Nc(|\L_2|, \sd^2 \K)$
  are disjoint. (This is a simple exercise on properties of
  simplicial/derived/regular neighborhoods using the tools
  from~\citep{rourke-sanderson82}. An explicit reference for this claim we are
  aware of is Corollary~4.5 in~\citep{tancer-tonkonog13}---embedding in a
  manifold assumed in~\citep{tancer-tonkonog13} plays no role in the proof.)
%  \martin{Explicitni reference kterou si uvedomuju pro posledni vetu je
%  Tancer-Tonkonog, Corollary~4.5, pricemz to, ze se jsou ve zneni toho 4.5 dane
%  komplexy vnorene ve variete nehraje zadnou roli v dukazu. (Doplnim, az budu
%  mit snazsi pristup k mathscinetu.)}
\end{proof}

\paragraph{Stellar subdivisions of polytopes.}
Let $\PP \subseteq \R^k$ be a polytope such that the origin belongs to the
interior of $\PP$ and let $\FF$ be a face of $\PP$. Let $\aa$ be a point beyond
all facets (i.e. maximal faces) $\FF'$ of $\PP$ such that $\cl \FF \subseteq \cl {\FF'}$ (that is,
$\aa$ and the origin are on different sides of the hyperplane defining $\FF'$)
whereas $\aa$ is beneath all other facets ($\aa$ and the origin are on the same
side of the defining hyperplane). Then the polytope $\PP'$ obtained as the
convex hull of the set of vertices of $\PP$ and $\aa$ is called a \emph{geometric
stellar subdivision} of $\PP$~\citep{ewald-shephard74}. For any
$\FF$, we can pick $\aa$ as above lying inside the cone of $\F(\PP)$ containing
$\FF$. 
%\vk{Tady jsem se chvíli zaseknul, co se tím myslí. Asi bych navrhnul to trošku přeformulovat: " For any
%$\FF$, we can choose $\aa$ as above lying inside the cone of $\F(\PP)$
%containing $\FF$"}\martin{Preformulovano. Jenom teda, je choose v necem lepsi
%nez pick?}
%\vk{OK. Myslím, že není. To nebyl důvod toho návrhu změny.}
Let $p
\colon \partial \PP' \to \partial \PP$ be the projection towards the origin.
Then the complex 
$p(\partial \PP') := \{p(\FF') \colon \FF' \hbox{ is a proper face of } \PP'\}$ 
is a subdivision of the boundary of $\PP$.\footnote{Considering $\partial \PP$
as a polytopal complex, $p(\partial \PP')$ is exactly the stellar subdivision
of $\partial \PP$ as defined in~\cite{ewald-shephard74}
on the level
of polytopal complexes; see also Exercise~3.0 in~\cite{ziegler95}. However, we do not need the exact formula explicitly. It is
sufficient for us that $p(\partial \PP')$ is a subdivision.} Consequently,
$\F(\PP')$ subdivides $\F(\PP)$.

%Moreover, if we perform stellar subdivisions gradually on all proper faces of a
%polytope $\PP$ ordered by non-increasing dimension, we obtain the barycentric
%subdivision 
%\vk{Tu druhou část předchozí věty bych možná smazal, a rovnou bych
%napojil následující větu. Protože to není úplně pravda (dostane se jen
%izomorfní komplex) a navíc se to opakuje ob jednu větu níže (kde je vysvětleno,
%co přesně to dělá).}. \martin{S timhle jsem trochu bojoval, jak to formulovat.
%Moje logika byla: vzdycky dostaneme barycentricke podrozdeleni. Kdyz to neni
%simplicialni polytop, tak to nechceme vysvetlovat, ale dusledek je, ze
%dostaneme simplicialni polytop. Kdyz zacneme u simplicialniho polytopu, tak to
%vysvetlime podrobneji. Nevim, jak to uplne navazat, aby se tam kus te informace
%neztratil. Vyresilo by to, kdyby se pridalo `up to isomorphism' a/nebo dala
%dvojtecka na konec predchozi vety a/nebo pridat "suitably" pred perform?}
%\vk{Není mi jasné, proč je potřeba vyloučit případ, že $\PP$ není simpliciální, protože ten popis níže platí totožně i pro tento případ, alespoň tedy nevidím problém. Ale pro případ, že je to tak potřeba, přidávám návrh, jak to celé zformulovat (zeleně):}
%\martin{OK. Odbarvil jsem zelenou a zakomentoval puvodni verzi. Pridal jsem jen
%slovicko `already' (zelene). Byl jsem ovlivneny postupem
%z~\cite{ewald-shephard74}: Vysvetli se v podstate, ze to dela barycentrickou
%subdivizi, a proto to vychazi simplicialni. Ale uznavam, ze do toho az takhle
%zabihat nemusime.}
If we perform stellar subdivisions gradually on all proper faces of a polytope
$\PP$ ordered by nonincreasing dimension, we obtain a simplicial polytope. In
fact, we get a polytope isomorphic to a barycentric subdivision of $\PP$;
however, we will use this stronger conclusion only when $\PP$ is already simplicial.
That is, in this case we obtain a polytope $\PP'$ such that the projection
$p\colon \partial \PP' \to \partial \PP$ is a simplicial isomorphism between
$\partial \PP'$ and $\sd \partial \PP$ provided in each step, when performing
individual stellar subdivisions over face $\FF$, the newly added point $\aa$ is
on the ray from the origin containing the barycenter of $\FF$.  For more
details on stellar and barycentric subdivisions of polytopes, we refer
to~\cite{ewald-shephard74}.
%If $\PP$ is not yet simplicial, our only desired conclusion is that we obtain
%a simplicial polytope after such subdivisions. On the other hand, if $\PP$ is
%already simplicial and we perform such stellar subdivsions, we obtain a
%polytope $\PP'$ such that the boundary of $\PP'$ is simplicially isomorphic to
%the barycentric subdivision of the boundary of $\PP$. That is, the projection
%$p\colon \partial \PP' \to \partial \PP$ is a simplicial isomorphism between
%$\partial \PP'$ and $\sd \partial \PP$ provided that in each step, when
%performing individual stellar subdivisions over face $\FF$, the newly added
%point $\aa$ is on the ray from the origin containing the barycenter of $\FF$.
%For more details on stellar and barycentric subdivisions of polytopes, we
%refer to~\cite{ewald-shephard74}.

\subsection{Fan of a semivalid representation}
%\martin{Je novy nadpis OK? (Viz poznamka k "conical representation".)}
%\vk{Může být. Alternativně, možná by šlo říkat "Fan of an extended representation", je to trochu kratší.}

Given a semivalid representation $L$ of $G$ we now aim to build a fan $\P =
\P(L)$ (complete in $L$) formed by convex polyhedral cones in a way that corresponds to splitting $L$ by
hyperplanes passing through the origin and perpendicular to the standard basis
vectors of $\R^V$.

%Let us recall that a \emph{fan} is a polyhedral complex
%where each polyhedron is a (convex) polyhedral cone centered in the origin. 
%For a fan $\F$ formed by cones in $L$, we say that $\F$ is \emph{complete} if
%$\bigcup \F = L$. 
%\martin{Potrebujeme vysvetlovat i polyhedralni komplex? Doufam, ze ne. Pismenko
%$\F$ nakonec pouzivam pro obecny vejir.}

\begin{defn}[Fan $\P(L)$]
  \label{d:PL}
%  \martin{Chces explicitne nazev "conical representation"? Myslim, ze by se bez
%  nej dalo obejit a mohlo to byt i snazsi pro ctenare nemit tolik ruznych
%reprezentaci.}\vk{Tak asi na tom netrvám. Ale přijde mi, že obvykle dát značení jméno zapamatování pomáhá, nikoliv škodí. Já si snáze vybavím, co je to conical representation, než co je $\mc{P}(L)$.}
%\martin{Ja obvykle mam taky radeji nejaky nazev nez symbol. Nicmene tech
%ruznych reprezentaci uz je hodne. Nebylo by treba uzitecnejsi psat `the fan
%$\P(L)$ coming from $L$', nebo dokonce `the fan $\P(L)$', kdyz to bude
%nekolikate opakovani blizko sebe?}
%\vk{Mírně mi přijde lepší nechat název "Conical representation", ale uznávám,
%  že je to zejména věc vkusu a netrvám na tom, aby byl zachován. Jestli chceš,
%  můžeš jej odstranit.}\martin{Ve chvili, kdy souhlasis s odstranenim, tak tady
%  se asi priklonim ke svemu nazoru. Tedy provedl jsem odstraneni.}

Let $L$ be a semivalid representation of $G$ and let us define an equivalence
relation $\sim$ on $\R^V$ by
\[
x\sim y \qquad \Longleftrightarrow\qquad \supp_+(x)=\supp_+(y) \text{ and }
\supp_-(x)=\supp_-(y).
\]
  Each equivalence class $[x]_{\sim}$ is a convex cone in $\R^V$ (relatively
  open), and we define $\E$ to be the fan formed by these cones. 

Then we define $\P = \P(L)$ as the fan obtained by intersecting $\E$ with $L$.
  In other words, the cones of $\P$ are the equivalence classes of $\sim$
  restricted to $L$.
  
%  as the collection of closures $\overline{[x]}_{\sim}$ of
%all such equivalence classes.\footnote{Note that each $\overline{[x]}_{\sim}$
%is indeed a convex polyhedral cone as it is intersection of $L$ and the
%polyhedral cone in $\R^V$ formed by all vectors $y \in R^V$ with $\supp_+(y)
%\subseteq \supp_+(x)$ and $\supp_-(y) \subseteq \supp_-(x)$.}

If the semivalid representation $L$ is irrelevant or understood from the context, we omit it from the notation and write just $\mc{P}$. We refer to a $k$-dimensional cone as to a \emph{$k$-cone}.
\end{defn}

%\martin{Vynechal odstavecek "Geometrically" z puvodni verze, protoze ta
%informace je obsazena v definici $\E$. Take si nejsem jisty, jestli je pokazde
%nutne zminovat, jaky typ fontu kdy budeme pouzivat. Nestacilo by to jen
%konzistentne pouzivat?}\vk{Tak pokud vím, zmínil jsem explicitně naši konvenci ohledně fontu pouze jednou. Ale jestli Ti to tak moc vadí, určitě to bez toho přežiju.}

We extend the notation of support to cones in $\P$, i.e., if $\alpha \in \P$,
then $\supp_\pm(\alpha) := \supp_\pm(x)$ for some $x \in \alpha$. Also, if $A
\subseteq \alpha$ for some $\alpha$ in $\P$, then $\supp_\pm(A) :=
\supp_\pm(\alpha)$. 
%Finally, given $\alpha \in \P$, by $\partial \alpha$ we
%denote the boundary of $\alpha$, that is the union of all proper faces of
%$\alpha$. 

%\martin{Definici $\partial \alpha$ jsem dale pozmenil, protoze mi prislo, ze se
%pouziva spis jako mnozina nez jako podkomplex (soubor kuzelu).}
%\martin{Tady jsem vynechal definici $\partial \alpha$. Je ted v topologickych
%predbeznostech.}

We continue with several observations on properties of $\P$. %about conical representations.
%\vk{Níže jsem změnil znění z implikace na ekvivalenci.}
\begin{obs}\label{obs:cones}
Let $\alpha,\beta$ be two cones of $\P$. Then $\alpha\subseteq\partial \beta$ if and only if 
\[
\supp_+(\alpha)\subseteq\supp_+(\beta) \quad\text{ and }\quad \supp_-(\alpha)\subseteq\supp_-(\beta)
\]
and at least one of the inclusions is strict.
\end{obs}
\begin{proof}
The equivalence follows immediately from the facts that $\partial
  \beta \subseteq \cl \beta$ and $\cl \beta$
contains all $y \in L$ with $\supp_+(y) \subseteq \supp_+(\beta)$ and
$\supp_-(y) \subseteq \supp_-(\beta)$. At least one of the inclusions is strict if and only if
$\alpha \neq \beta$. 
\end{proof}

%\martin{Tento dukaz je vyrazne zkraceny oproti Tve puvodni verzi. Souhlasis? Pokud
%to trochu chces vysvetlit, tak bych jen poznamenal neco jako: Indeed, if $x_v
%> 0$ for all $x$ in the relative interior of $\beta$, then $x_v \geq 0$ for all $x$ in $\beta$.}
%\vk{OK.}

\begin{cor}\label{cor:cones}
Whenever $\alpha,\beta$ are two cones of $\P$ such that $\alpha\subseteq\partial \beta$, then $\supp_+(\beta)\subseteq V\setminus \supp_-(\alpha)$.
\end{cor}
\begin{proof}
Indeed, $\supp_+(\beta) \subseteq V\setminus \supp_-(\beta) \subseteq V \setminus \supp_-(\alpha)$.
\end{proof}
%\martin{I tenhle dukaz jsem malinko zkratil - vyhnul se dukazu sporem. Viz
%komentar v starsi verzi.}
%\vk{OK}

%\vk{Tady je nové pomocné pozorováníčko.} \martin{Je to jen drobnost, ale mozna
%je trosku nestastne, ze Observation~\ref{obs:cones} se v dukazu pouziva s
%$\beta$ rovne $\alpha$ ze zneni tohoto dusledku, pokud rozumim spravne.}
%\vk{Myslím, že to záleží na tom, jak by člověk ten důkaz napsal. Když by použil Observation~\ref{obs:cones} sporem, tak ano, pokud to ale člověk napíše přímo, myslím, že pak $\alpha$ odpovídá $\alpha$ v Observation~\ref{obs:cones}.
%Řekl bych ale, že tohle nemůže nikoho zmást.}\martin{Asi to tolik nevadi. Ale
%ja to opravdu neumim odargumentovat tak, aby $\alpha$ odpovídalo $\alpha$ v
%Observation~\ref{obs:cones}. Prijde mi, ze u toho, aby $\alpha$ melo minimalni
%nosic jde o to, co je uvnitr $\alpha$, nikoliv vne $\alpha$. Takze mirne
%zmateny jsem a pouzit $\beta$ by mi prislo prirozenejsi, ale odargumentovat to
%umim a nevim, jestli by to pak melo nejake dusledky pozdeji.}
%\vk{Když bych to tady přeznačil z $\alpha$ na $\beta$, tak pak by vznikl v zásadě stejný problém v Observation~\ref{obs:broken_bd_1-cone} níže, které používá Corollary~\ref{cor:1-cones} na kužel $\alpha$.}
\begin{cor}\label{cor:1-cones}
A cone $\alpha$ of $\mc{P}=\mc{P}(L)$ is a 1-cone if and only if the vectors of $\alpha$ have inclusion-minimal support among nonzero vectors in $L$.
\end{cor}
\begin{proof}
If $\alpha$ is a 1-cone, then every vector in $\alpha$ has to have inclusion-minimal support among nonzero vectors in $L$ due to Observation~\ref{obs:cones}.
On the other hand, if $\alpha$ contains vectors $x$ with inclusion-minimal supports, then $\dim(\alpha)=1$, otherwise there were two linearly independent vectors $x,y\in\alpha$ and $x-\varepsilon y\in L$ would have strictly smaller support than $x,y$ for an appropriate choice of $\varepsilon>0$.
\end{proof}

\begin{defn}
If $\indg{G}{\supp_+(x)}$ is disconnected for a nonzero $x\in \R^V$, we call
  $x$ a \emph{broken vector}. The cones of %a conical representation 
  $\mc{P}$ consisting of broken vectors are called \emph{broken cones}.
\end{defn}

In the remainder of the present subsection, we always assume that $G=(V,E)$ is a
connected graph, $L\subseteq\R^{V}$ is a semivalid representation of $G$ and
$\mc{P}:=\mc{P}(L)$ is the fan corresponding to $L$.

%\vk{Další nové pozorování níže.}\martin{Tady to $\indg{G}{\supp_\pm(\alpha)}$
%pusobi trochu viceznacne, zni to jako jeden objekt a dalo by se zamenit s
%$\indg{G}{\supp(\alpha)}$. Mozna by to bylo lepsi rozepsat. Take neni jasne,
%jestli to rika, ze ve dvou vyskytech $\pm$ ve zneni maji byt obe volby $+$ a
%$-$ stejne, nebo mohou byt rozdilne.}
%\vk{OK, rozepsal jsem to.}

\begin{lemma}
\label{l:broken_bd_1-cone}
  Let $\beta$ be a broken cone of $\mc{P}$ and $\alpha$ be a cone of $\mc{P}$
  with $\alpha\subseteq\partial \beta$. Then
\begin{enumerate}[(i)]
  \item[{\normalfont(\mylabel{l:broken_bd_1-cone_supp-}{i})}] $\supp_-(\alpha) = \supp_-(\beta)$,
  \item[{\normalfont(\mylabel{l:broken_bd_1-cone_supp+}{ii})}] $\indg{G}{\supp_+(\alpha)}$ is equal to a single connected component of
    $\indg{G}{\supp_+(\beta)}$, and
  \item[{\normalfont(\mylabel{l:broken_bd_1-cone_1-cone}{iii})}] $\alpha$ is a 1-cone.
\end{enumerate}
%  \martin{Nemame moc konzistneni, kdy cislujeme jednotlive body 1. 2. 3., a kdy
%  (i), (ii), (iii). Nevedel jsem, co si vybrat.}\vk{Pravda. Sjednotil jsem to, všude teď používám malé římské číslice.}
\end{lemma}
%\vk{V důkazu níže byla nějaká chybějící slova, občas se používalo $B$ místo $b$, tak jsem to opravil. Taky jsem trochu zpřesnil asi dvě formulace (předpoklad pro spor ve 2. odstavci a úplně poslední věta důkazu).}
\begin{proof}
  For $W\subseteq V=V(G)$ by a (connected) \emph{component} of $W$ we always
  mean the vertex set of a connected component of $\indg{G}{W}$.
  Observation~\ref{obs:cones} says that
  $\supp_+(\alpha)\subseteq\supp_+(\beta)$,
  $\supp_-(\alpha)\subseteq\supp_-(\beta)$, and at least on of the inclusions
  is strict. Throughout the proof, $a$ is a vector in $\alpha$ and $b$ in
  $\beta$. 

  First, we deduce that $\supp_+(\alpha)$ contains at least one of
  the two components of $\supp_+(\beta)$. For contradiction, we assume that
  $\supp_+(\beta) \setminus \supp_+(\alpha)$ is not contained in a single component of $\supp_+(\beta)$. 
%  Let $a$ be a vector in $\alpha$
%  and $b$ a vector in $\beta$. 
  Consider the vector $b - Ka$ for $K > 0$
  sufficiently large. Then 
\[
\supp_+(b-Ka)=\br*{\supp_+(\beta)\setminus\supp_+(a)}\cup\supp_-(a).
\]
Definition~\ref{def:ex_represent}\eqref{def:ex_rep_4} applied to $b$ implies
  that $\supp_-(a)$ is in a different component of $\supp_+(b-Ka)$ than
  $\supp_+(\beta)\setminus\supp_+(a)$. Together with the assumption above this means that
  $\supp_+(b-Ka)$ has at least three components, which contradicts
  Definition~\ref{def:ex_represent}\eqref{def:ex_rep_1}. Let us denote by $C_1$
  a component of $\supp_+(\beta)$ contained in $\supp_+(\alpha)$ and let $C_2$
  be the other component.

  By similar ideas we deduce \eqref{l:broken_bd_1-cone_supp-}. For contradiction, suppose that
  $\supp_-(\alpha)\subsetneq\supp_-(\beta)$.  This time, we consider the vector
  $Ka - b$ for $K > 0$ sufficiently large. The component $\supp_-(\beta)$ thus
  contributes both to $\supp_+(Ka-b)$ and $\supp_-(Ka-b)$. The component $C_1$ of
  $\supp_+(\beta)$ is inside
  $\supp_+(Ka-b)$. No matter whether the component $C_2$ of
  $\supp_+(\beta)$ contributes to $\supp_+(Ka-b)$ or $\supp_-(Ka-b)$ or both, in each
  case, we deduce a contradiction with
  Definition~\ref{def:ex_represent}\eqref{def:ex_rep_1}---either we have three
  components in the positive support or at least two components in the positive support
  and two components in the negative support. In particular,
  $\supp_-(\alpha)=\supp_-(\beta)$  implies that
  $\supp_+(\alpha)\subsetneq\supp_+(\beta)$.

  Again by similar ideas we deduce \eqref{l:broken_bd_1-cone_supp+}. Now we know that $C_2$ is not contained
  in $\supp_+(\alpha)$ because $\supp_+(\alpha)\subsetneq\supp_+(\beta)$. Then
  $C_2 \cap \supp_+(\alpha)=\emptyset$, otherwise for $K > 0$
    sufficiently large $\supp_+(b-Ka)$ and $\supp_-(b-Ka)$ have two components
    each.

  Finally, all this implies that $\alpha$ is a $1$-cone via
  Corollary~\ref{cor:1-cones} as \eqref{l:broken_bd_1-cone_supp-} and \eqref{l:broken_bd_1-cone_supp+} (applied to $\alpha'$) show
  that there is no $\alpha'$ with its support strictly included in $\alpha$.
\end{proof}

The following observation is a generalization of part (8) from the proof of \citet[Thm.~3]{CdV_linkless}, but the present proof is different to that of \citep{CdV_linkless}, since we work with a more general object than a kernel of a matrix in $\mc{M}(G)$.
\begin{obs}[generalized {\citep{CdV_linkless}}]\label{obs:broken_2dim}
%Let $G$ be a connected graph, $M\in\mc{M}(G)$ and $L\subseteq\ker(M)$ be a semivalid representation of $G$.
Let $\beta$ be a broken cone of $\mc{P}(L)$. Then
\begin{enumerate}
\item[{\normalfont(\mylabel{obs:broken_2dim_1}{i})}] $\dim(\beta)= 2$ and
\item[{\normalfont(\mylabel{obs:broken_2dim_2}{ii})}] $\partial \beta$ consists of two $1$-cones, which correspond to vectors $x\in L$ for which $\supp_-(x)=\supp_-(\beta)$ and $\supp_+(x)$ is identical with one of the connected components induced by $\supp_+(\beta)$. 
\end{enumerate}
\end{obs}
\begin{proof}
First we observe that if $\beta$ was only a 1-cone, then any $x\in \beta$ would have inclusion-minimal support in $L$ due to Corollary~\ref{cor:1-cones}, which would be a contradiction to Definition~\ref{def:ex_represent}\eqref{def:ex_rep_2}.

Now let $\alpha$ be a cone from $\partial \beta$. Then
Lemma~\ref{l:broken_bd_1-cone} implies that $\alpha$ is a $1$-cone,
$\supp_-(\alpha)=\supp_-(\beta)$ and $\supp_+(\alpha)$ is equal to exactly
one of two connected components induced by $\supp_+(\beta)$. Therefore, we see
that there are at most two different choices for $\alpha$, which are
  necessarily $1$-cones. Given that $\partial \beta
  \neq \emptyset$, because the closure of $\beta$ does not contain (nonzero) opposite points of $L$, and since $\dim(\beta) \geq 2$ we deduce that $\dim(\beta) = 2$ by comparing the dimensions of the intersections of $\beta$ and $\partial\beta$ with the unit sphere in $L$ centered at the origin. 
Consequently, $\partial \beta$ consists of two $1$-cones.  
\end{proof}

\paragraph*{Notation.}
For $x\in L$ we write $S(x):=N(\supp(x))$ and
$R(x):=V\setminus\br*{\supp(x)\cup S(x)}$. Let $\beta\in\mc{P}$. We write
$S(\beta):=S(x)$ and $R(\beta):=R(x)$ for any $x\in \beta$. 
The notation is
motivated by the fact that $S(x)$ is a `separator' if $x$ is a broken vector
and $R(x)$ is the set of vertices of $G$ `remote' from
$\supp(x)$; see Figure~\ref{f:separation}.

\begin{figure}
\begin{center}
  \includegraphics[page=1]{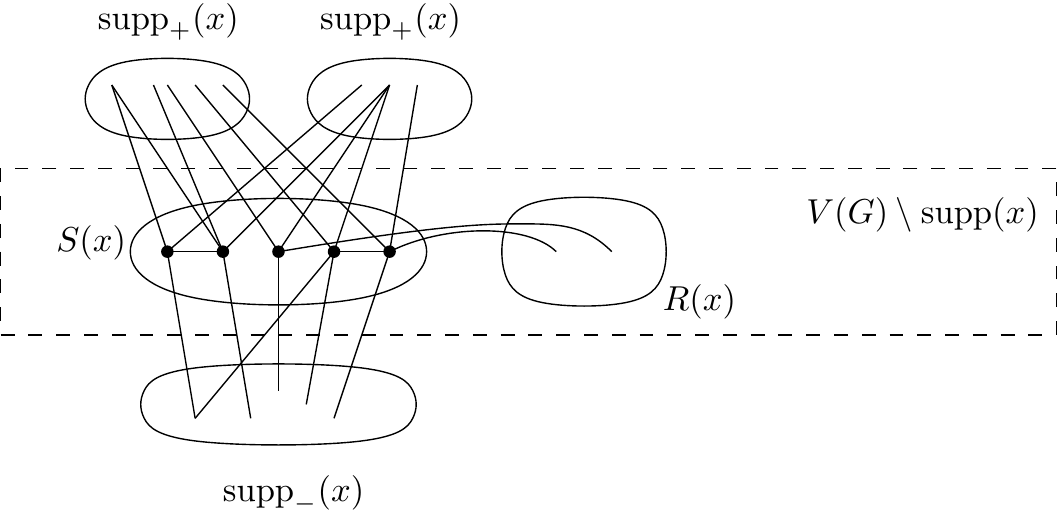}
  \caption{A typical picture of a separation of $G$ by $S(x)$ when $x$ is a broken
  vector in a semivalid representation. Compare with Lemma~\ref{lemma:basic}\eqref{lemma:basic_2}
and Definition~\ref{def:ex_represent}.}
 \label{f:separation}
\end{center}
\end{figure}
%\vk{K Figure~\ref{f:separation}: dal jsem pár hran do $S(x)$, aby si někdo nemyslel, že tam žádné nemohou být.}

\begin{obs}\label{obs:S_determines_R}
Let $\beta\in\mc{P}$ be broken. Then for every $x\in L$ such that $x_{S(\beta)}=0$ we have $x_{R(\beta)}=0$.
\end{obs}
\begin{proof}
Let $y\in \beta$ and assume, for contradiction, that there is $x\in L$ such that $x_{S(\beta)}=0$ and $x_{R(\beta)}\neq 0$. Since there is no edge between $R(\beta)$ and $\supp(y)$ in $G$, the set $\supp_+(y+\varepsilon x)$ is still disconnected and $\supp(y+\varepsilon x)$ induces at least four connected components, for all $\varepsilon>0$ small enough. This is incompatible with Definition~\ref{def:ex_represent}\eqref{def:ex_rep_1}. 
\end{proof}

A crucial observation for our subsequent considerations is the following:
\begin{obs}\label{obs:broken_not_adjacent}
Let $\alpha$ be a 1-cone of $\mc{P}$. Then there is at most one broken cone $\beta\in\mc{P}$ such that $\alpha\subseteq\cl{\beta}$.
\end{obs}
\begin{proof}
Let $\beta,\gamma$ be two broken cones such that $\alpha\subseteq \cl{\beta}\cap\cl{\gamma}$.
%Let $K_1,K_2$ be the connected components induced by $\supp_+(\beta)$.
By Observation~\ref{obs:broken_2dim}\eqref{obs:broken_2dim_2}
%, we can assume, up to renaming, that $\supp_+(\alpha)=K_1$, and also,
we get that $\supp_-(\beta)=\supp_-(\alpha)=\supp_-(\gamma)$. Definition~\ref{def:ex_represent}\eqref{def:ex_rep_4} then implies that $S(\gamma)=S(\beta)$.
%The same argument also shows that $S(\gamma)=S(\alpha)$, and thus, we get $S(\gamma)=S(\beta)$.
Applying Observation~\ref{obs:S_determines_R} finishes the argument.
\end{proof}

%\martin{Tady jsem se zacal ztracet v tom, jak jsi odvodil, ze $S(\gamma)\cap
%K_2=\emptyset$? (Oznaceno cervene.) Ale hlavne mi prijde, nerika hned na
%zacatku $\supp_+(\alpha) = K_1$ a $\supp_-(\alpha) = \supp_-\beta$, ze
%$S(\alpha) = S(\beta)$? Potom analogicky mame $S(\alpha) = S(\gamma)$, coz
%chceme. }\vk{Ano, máš pravdu, zjednodušil jsem to dle Tvého návrhu.}
%\martin{Opet si to jeste jednou projdu. Pak smazu komentare. (Tohle je asi
%sandne, ale chci si byt jisty, ze jsem si pripomnel vsechny definice, kdyz na
%to budu mit malinko delsi casovy usek.)}
%\vk{Tady končí veškeré změny v sekci 2.}

\subsection{Polytopal representation}
In analogy with the approach of \citet{CdV_linkless}, we utilize semivalid
representations $L$ of a given connected graph $G$ to build convex polytopes of
dimension $\dim(L)$. By a \emph{$k$-face} (or a \emph{$k$-cell}) we mean a face (or
a cell) of dimension $k$. We refer to a $d$-dimensional polytope as to a
\emph{$d$-polytope}.

\begin{defn}[Polytopal representation]\label{def:pol_represent}
  Let $L$ be a semivalid representation of $G$, and $\P = \P(L)$ be the complete fan
corresponding to $L$. We say that a polytope $\PP \subset L$ containing the
origin in its interior (relative in $L$) is \emph{polytopal representation} of
$G$ if it satisfies the following conditions.
%\green{The (open) edges of $\PP$ lying inside a broken cone are called \emph{broken edges}.}
%  We choose sufficiently dense set of \emph{unit} vectors in $L$ that will form the vertex set of a convex $\dim(L)$-polytope $\PP$ such that
 
%  \martin{Opravil jsem zneni teto defince, aby to resilo problemy s dostatecne
%  silnym bodem (vi). Navrzene opravy jsou zelene. Naopak cervene je to, co
%  navrhuju smazat (ale nechal jsem Ti to jeste pro srovnani).}

\begin{enumerate}
	\item[{\normalfont(\mylabel{d:pr_symmetric}{i})}] The vertex set of $\PP$ is centrally symmetric.
	\item[{\normalfont(\mylabel{d:pr_subdivides}{ii})}] $\PP$ subdivides $\P$.
	  This in particular means, that for every face $\FF$ of $\PP$,
	  there is a unique cone of $\P$ which contains $\FF$. We
	  denote this cone by $\gamma(\FF)$.
	\item[{\normalfont(\mylabel{d:pr_simplicial}{iii})}] $\PP$ is simplicial, that is, all faces of $\PP$ are simplices.
%	\item[(\mylabel{d:pr_one_edge}{iv})] Every face of $\PP$ contains at most one edge lying inside a
%	  2-cone of $\P$.
%\martin{Rozmyslet: `contains' je zavadejici.}
%\vk{Návrh na reformulaci: At most one of the edges of a face of $\PP$ lies inside a
%	  2-cone of $\P$.}
%	\item[(\mylabel{d:present_v}{v})] No two antipodal faces of $\PP$ can intersect. \martin{Navrhuju
%	  vynechat; melo by stacit, ze $\PP$ obsahuje pocatek}\vk{OK, ale udělejme to, až budeme mít zase jen jednu verzi. Posune se číslování a bojím se chaosu v natvrdo zadrátovaných referencích.}
	\item[{\normalfont(\mylabel{d:pr_inclusions}{iv})}] 
	  Let $\EE$, $\FF$ be faces of $\partial \PP$ which are faces of a
	  common face of $\partial \PP$.
	  %contained in a common proper face of $\PP$. 
	  Then either $\gamma(\EE)$ is a face of $\gamma(\FF)$ or
	  $\gamma(\FF)$ is a face of $\gamma(\EE)$. (This includes the
	  option $\gamma(\EE) = \gamma(\FF)$.)

%\martin{Opet `contained' je asi zavadejici.}
%\vk{Tady by šlo "contained in" nahradit ", which are also faces of".}
%    \martin{Preformulovano v podobnem duchu. Jeste jsem upresnil $\PP
%    \rightarrow \partial \PP$ (v jednom pripade jsem pak ubral `proper').}

%	  \red{Whenever a facet of $\PP$
%	  contains a $k$-face lying in a $(k+1)$-cone $\alpha$ and a $j$-face
%	  lying in a $(j+1)$-cone $\beta$, where $1\leq k\leq j\leq\dim(L)-1$,
%	  then $\alpha\subseteq\cl\beta$.}\martin{Chtel bych se zamyslet, jestli by
%	  tuhle podminku neslo formulovat nejak jednoduseji.}\vk{$k,j$ mají jít
%	  už od hodnoty 1, ne od hodnoty 2. Opravil jsem to tady i ve staré
%	  verzi. Nezkontroloval jsem ale, jestli to nějak neovlivňuje důkaz
%	  níže. Tahle vlastnost tímpádem taky implikuje vlastnost
%	  \eqref{def:pol_represent_iv} (tak to bylo vždycky myšleno, akorát
%	  předchozí opravou překlepu jsem vytvořil tuhle chybu), ale možná je
%	  pořád výhodné vlastnost \eqref{def:pol_represent_iv} explicitně
%	  nechat.}\martin{Navrhuju opravu vyse, coz mj. tak trochu resi i moji
%	  otazku s jednodussi formulaci. 
%	  Nemam nic proti tomu, kdyz (iv) bude
%	  zvlast. Bylo by fajn zminit, ze tento bod implikuje (iv), ale je
%	  potreba byt opatrnejsi. Striktne receno, jasny dusledek je, ze kdyz
%	  mame dve takove hrany, tak jsou ve stejnem kuzeli. A jeste by se melo
%	  zminit, ze to ale nejde.}
	\item[{\normalfont(\mylabel{d:pr_broken}{v})}] Let us define a \emph{broken
	  edge} as an edge of $\PP$ lying in a broken cone of $\P$. Then we
	  require: %The edges of $\PP$ lying inside a broken cone are called \emph{broken edges}.}
    For every $\mb{a}\in \PP^{(0)}$ all broken edges of $\PP$ in
    $\st(\mb{a}; \PP)$ belong to the same broken cone.
%    \martin{Makro `mylabel' nejak kazi cislovani.} \vk{Nevím, čím to tady bylo, momentálně jsem to vyřešil tím, že jsem každý bod přepsal pomocí toho makra. Ono stejně je rozumné se snažit vyhnout natvrdo zadrátovaným referencím jak jen to je možné (to je ostatně smysl toho makra, byť jej zatím neumím udělat tak, aby fungovalo optimálně).}

\end{enumerate}
%We call such a polytope $\PP$ a \emph{polytopal representation} of $G$.
%\vk{Tohle je napsáno už nad výčtem vlastnostní výše, zbytečné opakování.}
\end{defn}

%\martin{Moje osobni zkusenost je, ze pri cislovanych seznamech ve zneni
%vet/lemmat a podobne, nema cenu bojovat s tim, aby odkazy nebyly natvrdo
%zadratovane, protoze kdyz se zmeni zneni a neco se posune, tak to obvykle
%znamena, ze to clovek stejne chce prekontrolovat i vyznamove. (To je neco
%jineho, nez kdyz jsou cislovane vety, kde se treba jedno lemma vynecha neb
%presune.) Coz tady treba urcite jeste prekontrolovat chci, az budeme mit
%ujasneni posledniho bodu. (Riziko nestandardniho makra by taky mohlo byt, az to
%prejde do jineho stylu, ale to nevim---nezkoumal jsem, jak to mas udelane.)
%Pokud ale chces cislovat systematicky, tak se budu snazit to nekazit. Byt je
%riziko, ze to nekde prehlednu.}

%It should be more or less clear that polytopal representation always exists.
We, of course, need to know that a polytopal representation exists.
\citet{CdV_linkless} build a polytope $\PP$ satisfying \eqref{d:pr_symmetric}--\eqref{d:pr_simplicial} and a weaker version of \eqref{d:pr_inclusions} as a convex
hull of a sufficiently dense set of unit vectors taken from every cone, without
going into details about how to choose this set. As we add extra properties, we want to be more careful
and check that all of them can be satisfied.

\begin{prop}
\label{p:polytopal_exists}
Given a semivalid representation $L$, a corresponding polytopal representation
$\PP$ always exists.
\end{prop}

\begin{proof}
We start with considering the crosspolytope $\CC \subseteq \R^V$ whose vertices
  are the standard basis vectors $e_v \in \R^V$ and their negatives $-e_v$ for $v \in
  V(G)$. Then the fan of the crosspolytope $\F(\CC)$ is exactly the fan $\E$
  defined in Definition~\ref{d:PL}. Next we consider the auxiliary polytope $\QQ :=
  \CC \cap L$ and we get $\P = \F(\QQ)$. In particular, $\QQ$ subdivides $\P$. 

%  \martin{Mozna presnejsi odkaz na sekci nize, az bude jasna
%  struktura.}\vk{Změnil jsem odkaz z 'preliminaries' na
%  sekci~\ref{s:prelim}.}\martin{Jeste to tu chvili necham---rad bych z tech
%  podsekci udelal sekce.}\vk{OK.}
%  \martin{Pokud nakonec budou podsekce, nemam
%  nazor na to, jestli se na ne odvolavat jako na `Section' nebo `Subsection'.
%  Klidne to muze zustat jak to je. Jenom jsem zatim tyto komentare nesmazal,
%  kdybys Ty mel nazor.}
%  \vk{V tom případě bych mírně preferoval říkat podsekcím "Subsection". Změnil jsem to už podle toho všude.}

Subsequently, we apply a series of geometric stellar subdivisions on $\QQ$ as described
  in Subsection~\ref{s:prelim}. First we get a simplicial polytope $\QQ'$ which
  subdivides $\P$. Then we take $\PP$ as the second barycentric subdivision of
  $\QQ'$, again by a series of stellar subdivisions. We perform all stellar
  subdivisions in a centrally symmetric fashion so that we obtain centrally
  symmetric $\PP$. 

  It remains to verify the properties from Definition~\ref{def:pol_represent}.
  The properties \eqref{d:pr_symmetric}, \eqref{d:pr_subdivides}, and \eqref{d:pr_simplicial} follow immediately from the construction.

  We will show that \eqref{d:pr_inclusions} follows from
  Observation~\ref{o:edge_in_edge}. Let $\QQ''$ be the polytope obtained from $\QQ'$
  after the first barycentric subdivision and let $p''\colon \partial
  \PP \to \partial \QQ''$ be the projection towards the origin, as
  in Subsection~\ref{s:prelim}. 
%  \martin {Jenom k cislovani se sekci
%  Section~\ref{s:prelim} plati stejna poznamka jako vyse.}
  Then $p''(\partial \PP)$ is a barycentric subdivision of
  $\partial \QQ''$. 
  Now, let $\EE''$ be the face of $\QQ''$ containing
  $p''(\EE)$ and let $\FF''$ be the face of $\QQ''$ containing $p''(\FF)$; see
  Figure~\ref{f:inclusion_cones}. Note that
  $\EE'' \subseteq \gamma(\EE)$. Indeed, $\QQ''$ subdivides
    $\P$, therefore $\EE''$ is contained in some cone of $\P$, and
    $\gamma(\EE)$ is the only option. Similarly, $\FF'' \subseteq \gamma(\FF)$.
%  Therefore $\AA''$ is a $k$-face and $\BB''$ is a $j$-face. 
  By
  Observation~\ref{o:edge_in_edge}, $\EE''$ is a face of $\FF''$ or vice versa
  (the observation is applied with $\eta_1 = \EE''$, $\eta_2 = \FF''$, $\Delta_1
  = \EE$, and $\Delta_2 = \FF$). Therefore
  $\gamma(\EE)$ is a face of $\gamma(\FF)$ or vice versa.

  \begin{figure}
\begin{center}
  \includegraphics{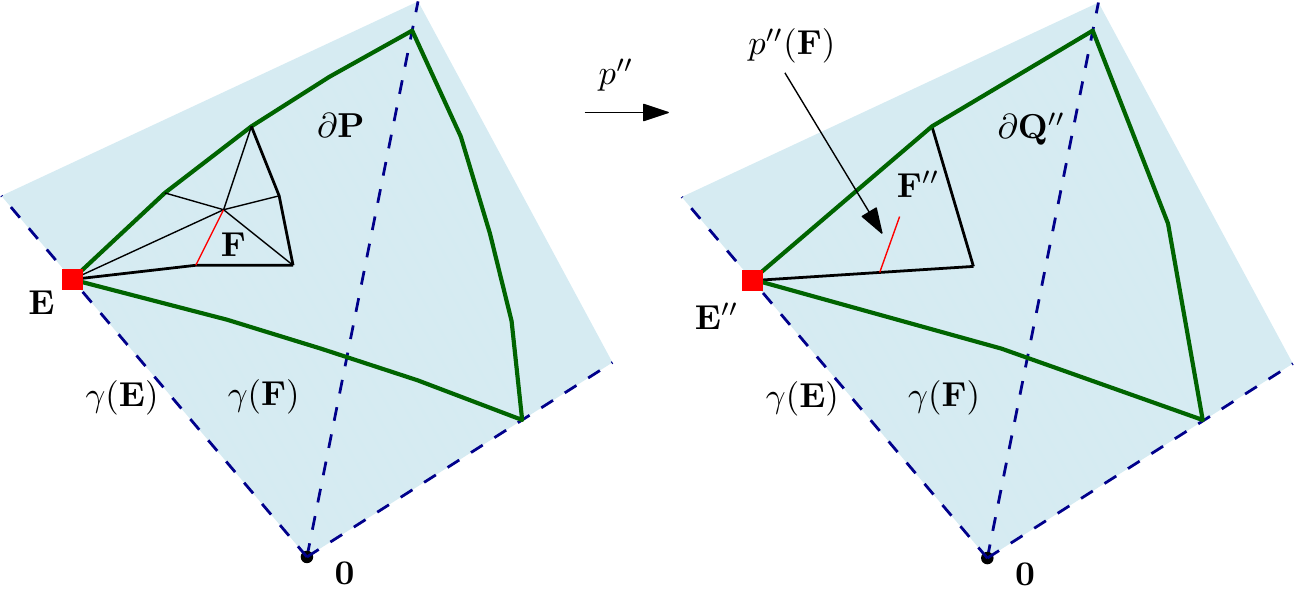}
  \caption{A picture ilustrating property \eqref{d:pr_inclusions}. In this picture, $L$ is
  $3$-dimensional. 
  Left: %The picture depicts $\partial \PP$.
  The faces $\EE$ and $\FF$ are a vertex and an edge in a common
  (small) triangle of $\partial \PP$. The larger (black) subdivided triangle containing
  both $\EE$ and $\FF$ is a result of applying a barycentric subdivision to
  (some triangle of) $\QQ''$. 
  The green outer `almost' triangle depicts the
  intersection of $\partial \PP$ and the cone $\gamma(\FF)$. Right: The picture
  shows $\EE''$ and $\FF''$ obtained as faces of $\QQ''$ containing $p''(\EE)$
  and $p''(\FF)$. In the specific case on the picture $p''(\EE)$ concides with
  $\EE$ and $\EE''$, thus only $p''(\FF)$ is depicted.}
\label{f:inclusion_cones}
\end{center}
  \end{figure}
%\red{
%  Similarly (v) follows from item~2 of Observation~\ref{o:edge_in_edge}.
%  Indeed, let $\FF$ be a facet of $\PP$ containing a $k$-face $\AA$ in $\alpha$
%  and $j$-face $\BB$ in $\beta$. Let $\AA''$ be the face of $\QQ''$ containing
%  $p''(\AA)$ and let $\BB''$ be the face of $\QQ''$ containing $p''(\BB)$. Note that
%  $\AA'' \subseteq \alpha$ and $\BB'' \subseteq \beta$ as $\QQ''$ subdivides
%  $\P$. Therefore $\AA''$ is a $k$-face and $\BB''$ is a $j$-face. By
%  Observation~\ref{o:edge_in_edge}, $\AA''$ is a face of $\BB''$. Therefore
%  $\alpha$ is a face of~$\beta$.
%}

  Finally, we derive \eqref{d:pr_broken} from Lemma~\ref{l:derived_neighborhoods}. This time,
  we consider the projection $p'\colon \PP \to \QQ'$. Then $p'(\partial \PP)$
  is the second barycentric subdivision of $\partial \QQ'$. For contradiction,
  assume that the edges of $\st(\aa; \PP)$ belong to two broken cones $\beta_1$
  and $\beta_2$. Equivalently, the edges of $\st(p'(\aa); p'(\partial \PP)) =
  \st(p'(\aa); \sd^2(\partial \QQ'))$ belong to $\beta_1$ and $\beta_2$. Let
  $\L_1$ and $\L_2$ be subcomplexes of $\partial \QQ'$
%  \vk{Nemají to
%  být subkomplexy jen $\partial\QQ'$? Jinak nevidím, že
%  Lemma~\ref{l:derived_neighborhoods} je použitelné v této
%  situaci.}\martin{Zmeneno. Nejak jsem implicitne chtle lemma pouzivat na
%  jejich podrozdeleni.}
  triangulating $\cl{\beta_1}$ and $\cl{\beta_2}$,
%  \vk{nemají to být spíše
%  $\cl{\beta_1}$ a $\cl{\beta_2}$?}\martin{OK. Uzavery pridany. Myslel jsem to
%  tak implicitne, ale plne uznavam, ze je lepsi to zduraznit.}
  respectively.
  Observation~\ref{obs:broken_not_adjacent} implies that $|\L_1| \cap |\L_2| =
  \emptyset$. Then, by Lemma~\ref{l:derived_neighborhoods}, $\st(p'(\aa);
  \sd^2(\partial \QQ'))$ cannot intersect both $|\L_1|$ and $|\L_2|$, a
  contradiction.
\end{proof}

\section{On the relation \texorpdfstring{$\mu(G)\leq \sigma(G)$}{\unichar{"1D707}(G)<=\unichar{"1D70E}(G)}}
\label{s:mu_sigma}
%\vk{V této sekci nebylo potřeba měnit nic kromě nějakých odkazů. Jen jsem
%trochu upravil důkaz Theoremu~\ref{t:main}. Je otázka, jestli upravit i znění
%té věty v úvodu, aby zahrnul $\eta$, či případně i $\lambda$, nebo ne. Nebo
%jestli v téhle sekci třeba nezopakovat znění v rozšířené verzi.}\martin{Asi
%bych zneni v uvodu spis nemenil, at tam neni prilis mnoho (zatim
%nedefinovanych) parametru. Ale jestli chces, muze se tady na zacatek dukazu
%zminit, co se vlastne dokazuje (silnejsi verze). Ale nechal bych to jen
%soucasti textu, nejspis.}
%\vk{OK, zmínil jsem to tady dvěma větami v textu.}

%\martin{Text nize by se asi v nejake forme mel casem presunout do uvodu
%spolecne se vsemi vysledky, co budeme mit.}
%\martin{Tohle by podle me urcite mela byt samostatna sekce.}
%\vk{Souhlas.}

%\martin{Navrhuju cerveny text nize nahradit zelenym kvuli uvodu.}
%\red{It is known that the values of $\mu(G)$ and $\sigma(G)$ coincide for all graphs with $\mu(G)\leq 4$; see \citep{sigma}. It is also known that, in general, the two parameters are different. \citet{Pendavingh_separation} demonstrated an example of a graph $H$ such that $\mu(H)\leq 18$ and $\sigma(H)\geq 20$.
%The aim of this section is to prove that $\sigma(G)$ is always and upper bound
%on $\mu(G)$, which settles the conjecture of \citet[Conj.~42]{sigma}.}
The aim of this section is to prove Theorem~\ref{t:main}. In fact, we prove that $\eta(G)\leq\sigma(G)$ for every graph $G$. This immediately implies Theorem~\ref{t:main} thanks to Observation~\ref{obs:mu_eta}.
%\par
%\medskip
%\begingroup
%\leftskip4em
%\rightskip\leftskip
%\par
%\centerline{
%\textit{For every graph $G$, $\eta(G)\leq\sigma(G)$.}
%}
%\endgroup
%\medskip

%\martin{Tady asi `citet' nefunguje moc dobre, protoze vypisuje pouze `Holst'.}\vk{Je to pokažené všude. Musí to být nějaký problém v kombinaci s biblatexem. Mám vytištěnou starší verzi těchto poznámek udělanou ještě bibtexem, kde je to správně.}\vk{UPDATE: vyřešeno tím, že jsem v bibliografii uzavřel `van der Holst' do složených závorek.}
%\martin{
%  Jeste se na to chvilicku podivam. Ted to zase kazi zkratky v
%literature, kdybychom to chteli pouzit treba na arXivu. Ale velkou prioritu to
%nema.}\martin{Tak jsem to snad opravil pomoci useprefix=true (po spouste
%zkouseni). Van der Holsta to zkracuje na vH misto vdH, ale to uz bych asi
%nechal...}
%\vk{Nevím, jestli se nám to nekompiluje nějak trochu jinak. Já třeba hned v odstavci nad těmihle poznámkami vidím van der Holsta zkráceného na "vdH", a ne "vH". Každopádně ať už je to tak či onak, obojí mi přijde OK.}
%\martin{Diky za upozorneni. Kompiluje se mi to jinak na skolnim a domacim
%pocitaci. Uplne radost z toho nemam, ze to takhle zavisi. Doufam, ze to nema
%dalsi dusledky. (Tipnul bych, ze na
%KAMu je zastaralejsi biber...)}
To make our exposition more readable, in the present section we refer to
vertices and edges of a graph as to \emph{nodes} and \emph{arcs},
respectively, and reserve the terms vertices and edges for the 0- and 1-faces of polytopes.

%\vk{Konvence přesunuta zde, vysvětlení níže.}

\begin{prop}\label{p:sigma_ex_rep_new}
  Let $G$ be a connected graph and $L$ be a semivalid representation of
  $G$. Then $\dim(L)\leq\sigma(G)$.
\end{prop}

%\vk{V proposition nad i pod jsem smazal "$=(V,E)$" z $G=(V,E)$, protože $V$ a $E$ se nikde níže nepoužívá (to jsem ověřil).}

The key step for the proof of Proposition~\ref{p:sigma_ex_rep_new} is to deduce Proposition~\ref{p:lemma37} below.
Given a polytope $\QQ$, two faces $\FF$ and $\FF'$ are
\emph{antipodal} if there exist two distinct parallel hyperplanes (relatively in the
affine hull of $\QQ$) $h$ and $h'$ such that $\FF \subset h$, $\FF' \subset h'$
and $\QQ$ is `between' $h$ and $h'$, that is, it belongs to one of the closed
halfspaces bounded by $h$ as well as one of the closed halfspaces bounded by
$h'$.
If $\QQ$ is centrally symmetric, then $\FF$ and $\FF'$ are antipodal if
and only if $\FF$ and $-\FF'$ belong to the closure of some proper face of $\QQ$.

Given two polyhedral complexes $\C$ and $\D$, a map $f\colon |\C| \to |\D|$ is
\emph{cellular} if $f(\C^{(k)}) \subseteq \D^{(k)}$ for every $k \geq 0$. If $\C$ and
$\D$ are graphs, which is the only case we are interested in, then this
condition means that every vertex of $\C$ is mapped to a vertex of $\D$.

\begin{prop}
\label{p:lemma37}
Let $G$ be a connected graph and $\PP$ a polytopal representation of
$G$ %\vk{Možná by pro tyhle účely bylo užitečné zavést značení $\PP(L)$ podobně
%  jako máme $\mc{P}(L)$. Pak by se tady $\mc{P}(L)$ vůbec nemusela zmiňovat.
%  Souhlasíš?}
(arising from the fan $\P = \P(L)$, where $L$ is a
semivalid representation of $G$).
Then, there is a cellular map $f\colon \PP^{(1)} \to G$ such that for every pair
  of antipodal faces $\FF$ and $\FF'$, the smallest subgraphs of $G$ containing
  $f(\FF^{(1)})$ and $f(\FF'^{(1)})$, respectively, have no common nodes.
\end{prop}
%\vk{Možná bych pro případ nějaké budoucí reference přesunul polytopální
%reprezentaci $\PP$ z notace výše do znění tvrzení \ref{p:lemma37}. Tak aby to
%znění dávalo smysl samo o sobě.}\martin{Klidne souhlasim.}\martin{Uz jsem to i
%zapracoval. Ovlivnilo to i presun znaceni---na novem miste o trochu nize je to
%zelene.}
%\vk{OK, akorát posun notace níže způsobil, že konvence zmiňující používání termínů vertex/node se objevila až za zněním Proposition~\ref{p:lemma37}, ve kterém se už používá. Proto jsem přesunul ustavení téhle konvence na začátek této sekce (zeleně).}

Using the tools of van der Holst and Pendavingh~\cite{sigma},
Proposition~\ref{p:lemma37} implies Proposition~\ref{p:sigma_ex_rep_new} quite
straightforwardly. As this proof is short, we present it before a proof of
Proposition~\ref{p:lemma37}. Here, we essentially only repeat the proof
of~\citep[Thm.~40]{sigma}.

\begin{proof}[Proof of Proposition~\ref{p:sigma_ex_rep_new}.]
  The main tool for this proof is Lemma 37 from \citep{sigma}.
%of van der Holst and Pendavingh. 
%\martin{Vyse jsem jeste udelal upravu. Bylo tam `is \citep[Lem. 37]{sigma}' ale
%  v takovem pripade je podle me vetnym clenem stale ten clanek, nikoliv lemma.
%  Domnivam se, ze jsi tohle opravoval po mne z jeste jine formulace. Tak radeji
%  upozornuju na zmenu.}
%  \vk{Ano, tu změnu jsem dělal já (i na jiných místech), abychom používal
%  konzistentní způsob citování. Nerozumím tomu, co je špatně na "The main tool
%  for this proof is \citep[Lem.~37]{sigma}." Sémanticky mi to přijde
%  ekvivalentní tomu, cos napsal výše, "\citep[Lem.~37]{sigma}" je podle mě jen
%  zkratka za "Lemma 37 from \citep{sigma}". Ale můžeš to nechat tak, jak to
%  je.}\martin{Ja jenom dovysvetlim: Alespon ja to ctu jako `Hlavnim nastrojem
%  pro tento dukaz je clanek \citep{sigma} (viz Lemma 37)'. A pak mi to
%  syntakticky do te vety nesedi. Nevim, jestli je na tohle vylozene konvence,
%  ale zase myslim, ze by to mohlo vest k mirnemu zmateni ctenare, a pak je
%  lepsi se tomu vyhnout.}
This lemma says that, under the additional assumption that
$\PP$ does not contain parallel faces (that is, faces with disjoint
  affine hulls such that $\FF - \FF$ and $\FF' - \FF'$ contain a common nonzero
  vector),
%  \vk{Přijde mi, že tahle zmínka (v zeleném) jen posouvá problém na otázku, co jsou paralelní afinní podprostory. Já si to musel najít na wiki, a navíc se mi zdá, že to nemá úplně standardizovaný význam. Raději bych napsal jako LS: "i.e., faces $\FF, \FF'$ such that $\FF-\FF$ and $\FF'-\FF'$ share a nonzero vector"}
%\vk{Hm, tady se mi zdá, že zmizela část naší diskuze (alespoň jsem přesvědčen,
%  že jsem viděl už i Tvou reakci na tuhle poznámku, ale  v gitu jsem to
%  nenašel...), kde jsi poznamenal, že sousední stěny krychle by pak byly
%  paralelní. Což je pravda, do mého návrhu má být ještě přidáno, že $\FF, \FF'$
%  mají disjunktní afinní obaly.)}\martin{Asi to s tou krychli bylo jen v mailu.
%  Podival jsem se do Lovasze-Schrijvera, a ted uz souhlasim i ja. Zmeneno na
%  novou zelenou verzi. Snad nemusime vysvetlovat $\FF - \FF$.}
the existence of $f$ from
  Proposition~\ref{p:lemma37} implies $\sigma(G) \geq \dim \PP$. (Note that
  $\dim \PP = \dim L$.)
%  \vk{Co je $d$?
%  Asi to byla $\dim(L)$. Pak by ta nerovnost asi měla být
%  obráceně.}\martin{Diky, opravena nerovnost, a $d$ nahrazeno dimenzi $\PP$.}
  Our $\PP$ contains
parallel faces. However, as van der Holst and Pendavingh point out, $\PP$ can be perturbed by a projective transformation
to a polytope without antipodal parallel faces preserving the
  combinatorial structure of the polytope. Similarly as van der Holst and
  Pendavingh do, we refer to the proof of \citep[Thm.~1]{CdV_linkless} for
  details.
\end{proof}

%\red{
\paragraph*{Notation.} 
%For the rest of the present section, we fix a connected graph $G=(V,E)$.
%Let $L$ be a semivalid representation of $G$ and $d:=\dim L$. Then there exists a
%polytopal representation of $G$ of dimension $d$ arising from the conical representation $\mc{P}(L)$. 
%In addition, we will assume that $\PP$ satisfies the conclusion of
%Proposition~\ref{prop:fine_polytope}.}
Given $G$, $L$, $\P$ and $\PP$ as in the statement of  
Proposition~\ref{p:lemma37}, we extend the notation $R(\gamma)$ and
$S(\gamma)$ from cones to faces of $\PP$. Let $\FF$ be a face of $\PP$, which lies in a unique cone $\gamma(\FF)\in\mc{P}$ by
Definition~\ref{def:pol_represent}. We define $S(\FF):=S(\gamma(\FF))$ and
$R(\FF):=R(\gamma(\FF))$. Note also that $\supp(\FF) = \supp(\gamma(\FF))$ and
$\supp_\pm(\FF) = \supp_\pm(\gamma(\FF))$ according to our convention above
Observation~\ref{obs:cones}.

\begin{proof}[Proof of Proposition~\ref{p:lemma37}.]
%\green{Let $\mc{P}:=\mc{P}(L)$ and let $\PP$ be a polytopal representation of $G$ arising from $\mc{P}$.}
%\vk{Přidal jsem tady to, co předtím bylo v odstavci "Notation", kde to ale přestalo dávat smysl.}
%First, we will build a CW-complex $\C$ such that $G = \C^{(1)}$ and a map $f
%\colon \partial \PP \to \C$. We define $f$ and $\C$ inductively, skeleton by
%skeleton. 
During the construction, for each face $\FF$ of $\PP$ we will
introduce a set $W(\FF)$, which will be a subset of nodes of $G$ such that 
%, and a subcomplex $\C(\FF)$ such that
%$\C(\FF)^{(1)} \subseteq G[W(\FF)]$ and 
$f(\FF^{(1)}) \subseteq G[W(\FF)]$. The key
property of the construction will be that $W(\FF)$ and $W(\FF')$ are disjoint
if $\FF$ and $\FF'$ are antipodal faces of $\PP$.
We first define $f$ and $W$ on the vertices of $\PP$ and then on
%  we extend the definition to
the edges of $\PP$. Finally, we extend the definition of $W$ to higher-dimensional
faces and verify the required disjointness condition.

Throughout the proof, we repeatedly use the fact that every broken cone is
  $2$-dimensional according to
Observation~\ref{obs:broken_2dim}\eqref{obs:broken_2dim_1}. In particular,
faces of $\PP$ lying in a broken cone are either broken edges, or `inner'
vertices in a broken $2$-cone.

%\martin{Ja bych psal jen
%  Observation~\ref{obs:broken_2dim}\eqref{obs:broken_2dim_1}, namisto
%\ref{obs:broken_2dim}, part \eqref{obs:broken_2dim_1}. Je to kratke a
%srozumitelne. Jestli je to OK, udelam to konzistentne v celem clanku.}
%\vk{Souhlasím, sám jsem si už říkal, že by to tak bylo lepší.}
%\vk{Už jsem to udělal.}

Before we start the construction, for every broken cone $\beta$
we fix a node $v(\beta) \in S(\beta)$. We also use the notation $v(\bb) :=
v(\beta)$, where $\bb$ is an arbitrary broken edge lying in $\beta$, that is,
$\gamma(\bb) = \beta$.

%\martin{Myslel jsem si, ze tady nekde budu chtit pouzit
%pozorovani~\ref{obs:S_determines_C}. Bud v definici $v(\bb)$ nebo potom v
%`Claimu' nize. Ale nakonec ho primo nevyuzivam. Je nekde schoane, ze ho
%vyuzivam implicitne?}\vk{Myslím, že se používá asi jen nepřímo (ale vynechat podle mě nejde). Používáš Corollary~\ref{cor:incident_broken}, které plyne z Observation~\ref{obs:broken_not_adjacent}, k jehož důkazu je právě potřeba Observation~\ref{obs:S_determines_C}.}

%\martin{V puvodnim textu se nize tvrdilo, ze se pouziva jeste
%Observation~\ref{obs:broken_2dim}. To se myslim nepouziva, ale je pro kontext
%dulezite si uvedomovat, ze rozbite kuzele jsou 2-rozmerne. Tak jsem to napsal
%zelene o kousek vyse a tady jsem odkaz na pozorovani zrusil.}

\paragraph{Dimension $0$.} Given $\uu \in \PP^{(0)}$, 
%Corollary~\ref{cor:incident_broken} and Observation~\ref{obs:broken_2dim}, part \eqref{obs:broken_2dim_1} imply that
Definition~\ref{def:pol_represent}\eqref{d:pr_broken} applied to $\aa =
-\uu$ implies that either there is no broken
edge antipodal to $\uu$, or there is a unique $2$-cone $\beta = \beta(\uu) \in \P$ such
that all broken edges antipodal to $\uu$ lie in $\beta$.
%\martin{Rika dusledek~\ref{cor:incident_broken}, ze $\beta$ je $2$-rozmerny?\vk{Neříká to přímo, ale implikuje to díky Observation~\ref{obs:broken_2dim}}.}
In the former case, we let $f(\uu)$ be an arbitrary
node of $\supp_+(\uu)$. 
%$R(\aa) := \supp_+(\aa)$
%and we let $f(\aa)$ be an arbitrary node in $R(\aa)$. 
In the latter case, we want to avoid $R(\beta)$ and $v(\beta)$; thus, we need to check that
we can do so.

%know that
%$\supp_+(\uu)$ contains a vertex different from $v(\beta(\uu)$. 

\begin{claim}
\label{c:avoid_v}
  If $\beta = \beta(\uu)$ exists, then there is a node in $\supp_+(\uu)
  \setminus R(\beta)$ different from $v(\beta)$.
\end{claim}

%\martin{V dukazu nize ale vesmes v celem textu pouzivam uzavrene kuzele a
%  steny. Ma to vyhody i nevyhody (nejake formulace jsou jednodussi, nejake
%  slozitejsi). Asi mi hlavne pripadaji uzavrene steny v
%simplicialnim komplexu nebo CW komplexu standardnejsi. Ale take se bunky CW
%komplexu dost hodi uzavrene pro popis `nalepovaciho' zobrazeni. V nejakou
%chvili se na tomhle potrebujeme dohodnout a sjednotit zapis.}
%\vk{Uzavřené stěny jsou OK, ty jsem taky tak používal. Moje konvence je/byla otevřené kužely, uzavřené stěny a buňky.}
%\vk{UPDATE: Předělal jsem to na rel.\ otevřené kužely. Je otázka, jestli vzhledem k současné zjednodušené verzi by nebylo nejrozumnější mít otevřené i stěny mnohostěnu.}

\begin{proof}
  We distinguish two cases according to whether $\gamma(-\uu)\subseteq \cl{\beta}$ or not. 

If $\gamma(-\uu)\subseteq \cl{\beta}$, we get
\[
  \supp_+(\uu) = \supp_-(-\uu) \subseteq \supp_-(\beta)
\]
  whereas $v(\beta)$ does not belong to $\supp(\beta)$. Therefore the claim
  follows from the facts that $\supp_+(\uu)$ is nonempty by Definition~\ref{def:ex_represent}\eqref{def:ex_rep_0} and $R(\beta) \cap
  \supp(\beta) = \emptyset$.

Now we assume that $\gamma(-\uu)\not\subseteq \cl{\beta}$. Let $\bb$ be an
  arbitrary broken edge antipodal to $\uu$. We know that $\beta = \gamma(\bb)$. We also know that there is
a proper face $\FF$ of $\PP$ such that $\bb$ and $-\uu$ belong to $\FF$.
%\red{
%This, by Definition~\ref{def:pol_represent}\eqref{d:pr_inclusions}, means that $\gamma(-\uu)$ is at least 3-dimensional cone
%  satisfying $\beta \subseteq \cl{\gamma(-\uu)}$. Indeed, $\beta$ is a
%    face of $\gamma(-\uu)$ or vice versa, but $\gamma(-\uu)\not\subseteq
%  \cl{\beta}$.
% }
% \green{
 Definition~\ref{def:pol_represent}\eqref{d:pr_inclusions} implies that $\beta$ is a
    face of $\gamma(-\uu)$ or vice versa. Since $\gamma(-\uu)\not\subseteq
  \cl{\beta}$, we obtain that $\gamma(-\uu)$ is at least 3-dimensional cone
  satisfying $\beta \subseteq \cl{\gamma(-\uu)}$.

Now we get $\supp_+(\uu) = \supp_-(-\uu) \supseteq \supp_-(\beta)$. We also again
  use that $v(\beta)$ does not belong to $\supp(\beta)$. Therefore, the
  claim follows from the fact that $\supp_-(\beta)$ is nonempty and $R(\beta)
  \cap \supp(\beta) = \emptyset$. 
\end{proof}

Therefore, if $\beta = \beta(\uu)$ exists, by Claim~\ref{c:avoid_v}, we may
set $f(\uu)$ to be an arbitrary node of $\supp_+(\uu) \setminus R(\beta)$
different from $v(\beta)$.

We also set, somewhat trivially, $W(\uu):=\set{f(\uu)}$.
% to be the set with the only node $f(\uu)$.
%and $\C(\uu)$ the $0$-dimensional 
%subcomplex with the only node $f(\uu)$.

%know by Observation~\ref{obs:S_determines_C} that $\supp_+(\aa)\setminus
%C(\Gamma(\aa))\neq \emptyset$, thus we can let $f(\aa)$ be an arbitrary node in
%$\supp_+(\aa)\setminus C(\Gamma(\aa))$. In both cases, we set, somewhat
%trivially, $\C(\aa) = R(\aa) = R'(\aa)$ to be the graph with the only node
%$f(\aa)$. Then the properties (P1)-(P4) are satisfied.
%However, we
%will frequently only use that $R(\aa) \subseteq \supp_+(\aa)$, or slightly
%stronger $R(\aa) \subseteq \supp_+(\aa)\setminus C(\Gamma(\aa))$ in the latter case.

\paragraph{Dimension $1$.} 
Let $\ee = \uu\ww$ be an edge of $\PP$. We want to define $f$ on $\ee$ as well as
$W(\ee)$.
%and $\C(\ee)$. 
We proceed so that
%\vk{Jazyková věc: docela často používáš "so that" způsobem, o kterém si myslím, že není správný. Podle mě "so that" je v podstatě ekvivalentní českému "aby", "tak, aby", případně "ať", tedy uvozuje vedlejší větu účelovou. Zatímco Ty často používáš "so that" k uvození vedlejší věty způsobové/prostředkové, "tak, že...".
%Chtěl jsem najít autoritativní referenci, ale zatím se mi to nepodařilo, zejména kvůli tomu, že nevím jak se anglicky řekne "příslovečné určení způsobu". Matoušek to explicitně zmiňuje ve svém textu o matematické angličtině \url{https://kam.mff.cuni.cz/~matousek/angli.pdf} v bodě 12. Tam zmiňuje jako možné řešení použít "in such a way that", osobně to někdy používám, i když mám pochybnosti o správnosti i této možnosti (našel jsem na nějakých jazykových fórech příspěvky rodilých mluvčích, kteří tvrdí, že tyhle dvě fráze jsou v podstatě ekvivalentní.).
%Nějaké jasné řešení by mohlo být třeba něco jako "We proceed in the following way: We first suitably define..."}
for every edge $\ee =
\uu\ww$ of $\PP$ we first suitably define $W(\ee)$ in such a way
that $f(\uu)$ and $f(\ww)$ are nodes in the same connected component of
$G[W(\ee)]$. %We also set $\C(\ee) = G[R(\ee)]$. 
Then we set $f(\ee)$ to be an arbitrary
path connecting $f(\uu)$ and $f(\ww)$ inside $G[W(\ee)]$.
%Usually, we will actually check that $R(\ee)$ is connected; however, there will
%be a single exception. This will take care of properties (P1) and (P2). We will
%also suitably define $R'(\ee)$ as a subgraph of $R'(\ee)$ which is (P3). Then
%(P4) is trivially satisfied. (Of course, we aim to choose $R'(\ee)$ so that (P5)
%will be satisfied after we finish the construction, thus we do not have
%complete freedom how to choose $R'(\ee)$.)

If $\ee = \bb$ is a broken edge, then we set $W(\bb) := \supp_+(\bb) \cup
\{v(\bb)\}$. Then $f(\uu)$ and $f(\ww)$ are nodes in $W(\bb)$ as
$\supp_+(\uu), \supp_+(\ww) \subseteq \supp_+(\bb)$. Also, $\indg{G}{W(\bb)}$ is
connected as $v(\bb)$ is adjacent to every component of $\indg{G}{\supp(\bb)}$ by Definition~\ref{def:ex_represent}\eqref{def:ex_rep_4}. 
%Also,
%$R(\uu) \cup R(\ww) \subseteq R(\bb)$ as $R(\uu) \subseteq \supp_+(\uu)
%\subseteq \supp_+(\bb)$ and we have analogous inclusions for $R(\ww)$.

Now, let us assume that $\ee$ is not broken. 
For the connectedness of $\indg{G}{W(\ee)}$ it would suffice to set $W(\ee)
= \supp_+(\ee)$. We know that $G[\supp_+(\ee)]$ is 
connected as $\ee$ is not broken, and also, $f(\uu)$ and $f(\ww)$ are nodes of
$G[W(\ee)]$ by the same argument as above. 
%In this case, we could set $R(\ee)
%= \supp_+(\ee)$. We know that $G[\supp_+(\ee)]$ is 
%connected as $\ee$ is not broken and also $f(\uu)$ and $f(\ww)$ are vertices of
%$G[R(\ee)]$ by the same argument as above.
%$R(\uu) \cup R(\ww) \subseteq R(\ee)$
%by the same inclusions as in the previous case. 
However, in some cases we 
want $W(\ee)$ to be smaller; namely, if there is a broken edge $\bb$
antipodal to $\ee$, we want to avoid $v(\bb)$.
Note that the cone $\beta := \gamma(\bb)$ is
independent of the choice of $\bb$, if $\bb$ exists,
by Definition~\ref{def:pol_represent}\eqref{d:pr_broken} applied to
an arbitrary vertex of $-\ee$ in place of $\aa$. Then $v(\bb) = v(\beta)$, $R(\bb) =
R(\beta)$ and $S(\bb) = S(\beta)$ are independent of $\bb$ as
well. 
%\martin{Tady jsem zakomentoval vetu `Recall also that all broken cones are
%  $2$-cones by Observation~\ref{obs:broken_2dim}.' Myslim, ze uz neni potreba,
%vzhledem k pripomenuti pozorovani na zacatku.}
%Recall also that all broken cones are
%$2$-cones by Observation~\ref{obs:broken_2dim}.
So,
we set $W(\ee) :=  \supp_+(\ee)$ if there is no broken edge antipodal to
$\ee$, but we set $W(\ee) :=  \supp_+(\ee) \setminus \set{v(\bb)}$ if there is a
broken edge $\bb$ antipodal to $\ee$. 

We want to check that $f(\uu)$ and $f(\ww)$ belong to the same
connected component of $\indg{G}{W(\ee)}$. This we already did in the former case, thus
it remains to consider the latter case, when $\bb$ exists. 
%In this case we set
%$\beta := \gamma(\bb)$. 
We observe that since $\ee$ is antipodal to $\bb$, the
vertices $\uu$ and $\ww$ are antipodal to $\bb$ as well. Therefore, both
$f(\uu)$ and $f(\ww)$ are distinct from $v(\bb) = v(\beta)$. In other words,
$f(\uu)$ and $f(\ww)$ indeed lie in $W(\ee)$. It remains to show that they
belong to the same connected component of $G[W(\ee)]$.

\begin{claim}
\label{c:cone}
%If there is an antipodal edge $\bb = \bb(\ee)\in B(\PP)$ (which is then unique), then
  Either $\bb=-\ee$, or $\gamma(\ee)$ is at least $3$-dimensional, 
  and $-\beta \subsetneq \cl{\gamma(\ee)}$. 
\end{claim}

\begin{proof}
Assume that $\bb \neq -\ee$. Because $\bb$ and $\ee$ are antipodal,
  there is a face $\DD$ of $\partial \PP$ containing $-\bb$ and $\ee$. Therefore $\gamma(-\bb) =
-\beta$
is a face of $\gamma(\ee)$ or vice versa according to
Definition~\ref{def:pol_represent}\eqref{d:pr_inclusions}. Since $-\beta$ is a
$2$-cone and $\gamma(\ee)$ is at least $2$-dimensional, $-\beta$ must be a face
  of $\gamma(\ee)$. It remains to observe that $-\beta \neq \gamma(\ee)$. For contradiction assume
  $-\beta = \gamma(\ee)$. Consider the defining hyperplane for $\DD$; it
  contains $-\bb$ and $\ee$.
  Therefore it contains
  $-\beta$ because $-\beta$ is in the
    affine hull of $\bb\cup-\ee$ if $\bb \neq -\ee$ and 
      $-\beta = \gamma(-\bb) =  \gamma(\ee)$.
  Consequently, it contains the origin,
  which is a contradiction.
%\red{because $-\beta$ is $2$-dimensional, $-\bb \neq \ee$, and $-\bb$ and $\ee$ belong to a common
%face of $\PP$.}
%\red{because of Definition~\ref{def:pol_represent}\eqref{d:pr_one_edge}.}
%\vk{Navrhuji nahradit červený text zeleným textem. Viz diskuze níže.}
%  \martin{Nynejsi cerveny text je Tvuj puvodni zeleny. Naopak zelene jsem
%  napsal novy navrh. Snazil jsem se to opravdu podrobne vysvetlit, tak je to
%  mirne delsi. (Kdyz jsem psal mail, tak jsem si myslel, ze to pujde rict jeste
%  malinko jinak, takze se ted mozna mirne odlisuju. Nicmene, porad myslim, ze
%  zarazeni do separatniho bodu v definici se nevyplati, protoze tahle uvaha je potreba
%  prave jednou.}
%  \vk{Zakomentoval jsem červenou verzi a odbarvil zelenou. Ale narhoval bych
%  odstranit opakování rovnosti $-\beta = \gamma(-\bb)$, je to tam čyřikrát a
%  dělá to zbytečně dlouhé formule. Myslím, že jednou nebo dvakrát by to
%  stačilo.}\martin{Diky. Viz reakce (ob jedna) vyse.}
\end{proof}

We remark that if the former case $\bb=-\ee$ occurs, then $v(\bb) \not \in
\supp_+(\ee)$ as $v(\bb) \not \in \supp(\bb) = \supp(\ee)$; we already
resolved this situation. Thus it remains to consider the case that $\gamma(\ee)$ is
at least $3$-dimensional and $-\beta \subsetneq \cl{\gamma(\ee)}$. In
addition, we can assume that $v(\bb) \in \supp_+(\ee)$ (again, the opposite case was already resolved).
%\vk{Navrhuji přidat zelený text, aby se někdo nemyslel, že to tak musí být vždycky.}

%Let first, in addition, assume that $f(\uu), f(\ww) \neq v(\bb)$. 

%te that $\Gamma$ is the same as cones $\Gamma(\uu)$ and $\Gamma(\ww)$ used in the
%definition of $f(\uu)$ and $f(\ww)$. Thus, we get 

Now note that $f(\uu) \in \supp_+(\uu)
\setminus R(\beta)$ and $f(\ww) \in \supp_+(\ww) \setminus R(\beta)$ due to the
definition of $f(\uu)$ and $f(\ww)$. 
This gives $f(\uu), f(\ww) \in \supp_+(\ee) \setminus R(\beta)$.
%= \supp_+(\ee) \setminus C(\bb)$. 

From $-\beta \subseteq \cl{\gamma(\ee)}$ we also get
\[\supp_+(\beta) = \supp_-(-\beta) \subseteq \supp_-(\gamma(\ee)) = \supp_-(\ee).\]
Therefore $f(\uu), f(\ww) \not \in \supp_+(\beta)$, because they
belong to $\supp_+(\ee)$. Altogether, both $f(\uu), f(\ww) \in \supp_{-}(\beta) \cup
S(\beta)$ as they also do not belong to $R(\beta)$.
Moreover, each of $f(\uu)$ and
$f(\ww)$ either belongs to $\supp_{-}(\beta)$ or has a neighbor in
$\supp_{-}(\beta)$, since each vertex of $S(\beta)$ is connected to every component of
$G[\supp(\beta)]$. We also know that $G[\supp_{-}(\beta)]$ is connected by Definition~\ref{def:ex_represent}\eqref{def:ex_rep_1} as
$\beta$ is broken, that $\supp_{-}(\beta) = \supp_+(-\beta) \subseteq \supp_+(\gamma(\ee)) =
\supp_+(\ee)$, and that $v(\beta) \notin \supp_{-}(\beta)$. Altogether, $f(\uu)$
and $f(\ww)$ can indeed be connected inside $G[\supp_+(\ee) \setminus
\{v(\bb)\}]$. (See Figure~\ref{f:connecting_uw} as an example.)

%\martin{Mozna se pokusit o nejaky obrazek?}
%\martin{Jeste jsem doplnil obrazek pro tento pripad.}
%\vk{Díky!}

\begin{figure}
\begin{center}
  \includegraphics[page=2]{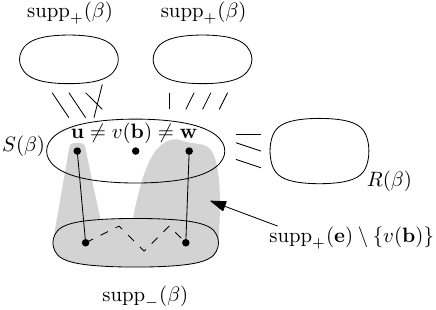}
  \caption{Connecting $f(\uu)$ and $f(\ww)$ inside $G[\supp_+(\ee) \setminus
  \{v(\bb)\}]$ in the case that neither $\uu$ nor $\ww$ belongs to
  $\supp_{-}(\beta)$.}
 \label{f:connecting_uw}
\end{center}
\end{figure}

\paragraph{Higher dimensions.}
%Given $k \geq 2$, we now aim to define $\C^{(k)}$ and $R(\FF)$ and $\C(\FF)$ for $k$-faces
%$\FF$ of $\PP$, assuming that the definition is already provided for lower
%values of $k$. 
It remains to define $W(\FF)$ for faces $\FF$ of $\PP$ of higher dimensions. 
We inductively set $W(\FF) := \bigcup_{\HH} W(\HH)$,
where the union is
over all proper subfaces $\HH$ of $\FF$. As the definition is given inductively,
this is equivalent with setting $W(\FF)$ to $\bigcup_\ee W(\ee)$ where the
union is over the edges $\ee$ in $\FF$.
Then we easily get $f(\FF^{(1)}) \subseteq G[W(\FF)]$ for any face $\FF$ of
$\PP$, as required.

  It remains to prove that $W(\FF)$ and $W(\FF')$
    are disjoint for any pair $\FF$ and $\FF'$ of antipodal faces of $\PP$.

  For contradiction, let us assume that $W(\FF) \cap W(\FF') \neq \emptyset$.
  Due to the definition of $W(\FF)$ and $W(\FF')$, there are faces $\ee$ in
  $\FF$ and $\ee'$ in $\FF'$ of dimension at most 1 such that $W(\ee) \cap
  W(\ee') \neq \emptyset$. (We use the edge notation $\ee$ and
  $\ee'$, which corresponds to the `typical case'; however, one of $\ee, \ee'$ may be a vertex, if $\FF$ or
  $\FF'$ is $0$-dimensional.) We remark that $\ee$ and $\ee'$ are
  antipodal as $\FF$ and $\FF'$ are antipodal. Therefore, there is a proper face $\DD$
  containing $\ee$ and $-\ee'$. 

  If neither $\ee$ nor $\ee'$ is a broken edge, then $W(\ee) \subseteq \supp_+(\ee)
  \subseteq \supp_+(\DD)$, and $W(\ee') \subseteq \supp_+(\ee')
    \subseteq \supp_-(\DD)$, which is a contradiction.
    %\vk{No lemma to prove here.}

%\martin{V textu nize jsem zmenil odkaz z byvale Definice (iv) na byvalou Definici
%  (v), nyni Definici (iv). Dodal jsem poradnejsi zduvodneni jak se pouzije, coz
%  nezaviselo na tom, jaka verze se pouzije, ale spis chybely jine kroky.
%  (Odpustis mi jedno `indeed' v zavorce nize? Take jsem odstranil jedno
%  preznaceni $\ee'$ na $\bb$. Myslim, ze v soucasne verzi by uz bylo spis jen
%  matouci. Uz v minulem pruchodu jsem to zvazoval odstranit, protoze je to
%  relikt toho, ze pred zjednodusenim teto casti dukazu se odkazovalo na druhy
%  Claim. Vse je prebarveno na aktualni navrh.}

  Therefore, we can assume that $\ee$ or $\ee'$ is a broken edge; say
  $\ee'$ is broken. Then $\ee$ cannot be broken. (Indeed, if
  $\ee$ were broken, it would have to be an edge. Therefore, by
  Definition~\ref{def:pol_represent}\eqref{d:pr_inclusions} and Observation~\ref{obs:broken_2dim}\eqref{obs:broken_2dim_1}, $\gamma(\ee) =
  \gamma(-\ee')$, but $\gamma(\ee')$ and $-\gamma(\ee')$ cannot be both broken due
  to Definition~\ref{def:ex_represent}\eqref{def:ex_rep_1}.)
  We get $W(\ee') %= W(\bb) 
  \subseteq
  \supp_+(\ee')  \cup \set{v(\ee')} \subseteq\supp_-(\DD)\cup\set{v(\ee')}$. On the other hand, $W(\ee) \subseteq
  \supp_+(\ee) \setminus \set{v(\ee')}
  \subseteq\supp_+(\DD)\setminus\set{v(\ee')}$.
%	\vk{Nešlo by to tady už ukončit? Poslední dvě věty výše bych malinko
%	rozšířil o inkluzi do supportu $\DD$ (zeleně), což podle mě dává
%	výsledek.}\martin{Nevidim v tom problem. Diky! Takze jsem to udelal
%	podle Tveho.}
%  \red{
%  Now, we are in the setting of
%  Claim~\ref{c:cone}. Thus we get that either $\bb = -\ee$, or $\gamma(\ee)$
%  is at least $3$-dimensional and $-\beta \subseteq \cl{\gamma(\ee)}$, where $\beta =
%  \gamma(\bb)$. In the former case, we immediately get that $W(\bb)$ and
%  $W(\ee)$ are disjoint. Thus, let us focus on the latter case. Then we further 
%  get
%  $$W(\ee') \subseteq \supp_+(\bb)  \cup v(\bb) \subseteq
%  \supp_-(\gamma(\ee)) \cup v(\bb) = \supp_-(\ee) \cup v(\bb).$$
%  }
  Therefore $W(\ee)$ and $W(\ee')$ are disjoint in this case as well.
\end{proof}
\begin{proof}[Proof of Theorem~\ref{t:main}]
 By Observation~\ref{obs:mu_eta}, $\max\set{\lambda(G), \mu(G)}\leq\eta(G)$ for every graph $G$. To prove the inequality $\eta(G)\leq\sigma(G)$, we can assume that $G$ is connected as for
  disconnected graphs both parameters $\eta$ and $\sigma$ are realized as the
  maximum of the respective parameter over the components of $G$,\footnote{For the parameter $\sigma$ it follows from its definition.} if it contains at least one
  edge (and the claim
%  \vk{Možná bych spíše napsal "already known",
%  protože triviální mi to nepřijde (aby člověk dokázal, že $\mu\leq 1$ pro
%  takové grafy, musí trochu pracovat.) Nebo přehlížím snadný
%  argument?}\martin{Pokud uz do toho chces zabihat, tak bych se odkazal na
%  charakterizace uplne na zacatku v uvodu. Jen by tam asi bylo odbre pridat i
%  variantu $\mu(G) \leq 0$.}
%  \vk{Přidáno.}\martin{Ja teda jeste pridal odkaz na intro.}
  follows from the characterization of classes of graphs with
  $\sigma(G)\leq 0,1$ for graphs without edges; see the
  introduction).
  % Let $M$ be any matrix maximizing $\corank(M)$ among matrices in $\mc{M}(G)$ satisfying SAH; hence, $\mu(G)=\corank(M)$. Since $M$ satisfies SAH, $\ker(M)$ is a semivalid representation of $G$ of dimension $\corank(M)$ by Lemma~\ref{lemma:Pendavingh}.
  Applying Proposition~\ref{p:sigma_ex_rep_new} to any semivalid representation $L$ of $G$ such that $\eta(G)=\dim(L)$, we get that
  $\eta(G)\leq \sigma(G)$.
\end{proof}

\section{On the relations between \texorpdfstring{$\mu, \lambda$}{\unichar{"1D707}, \unichar{"1D706}} and \texorpdfstring{$\sigma$}{\unichar{"1D70E}}}
\label{s:gaps}

In this section, we further investigate the distinction between $\mu, \lambda$ and $\sigma$.
\Citet[Thm.~40]{sigma} proved that $\lambda(G)\leq\sigma(G)$ for every graph
$G$.
Moreover, \citet{Pendavingh_separation} provided an
example of a graph $G$ such that $\mu(G)\leq 18$ and $\lambda(G)\geq 20$. This
is the example that we mentioned in the introduction, which shows that
the parameters $\sigma$ and $\mu$ are different in general.  %\vk{Poslední dvě
%věty výše mají jistý překryv s tím, co je už zmíněno v úvodu, kde to bylo pro
%$\mu$ a $\sigma$. Proto navrhuji přidat zelený text.} 
%\red{Note that a
%characterization of graphs $G$ for which $\lambda(G)\leq d$ was obtained by
%\citet{lambda} for values $d=1,2,3$.  From this characterization it follows
%that the parameters $\lambda$ and $\mu$ differ already for those small values.
%In general, $\lambda$ can be both greater or smaller than $\mu$.} \vk{Zde
%navrhuji smazat, protože jsem to přesunul do odstavce výše.}

%\red{One interesting class of examples arises from the finite projective planes. Another one is comprised of strongly regular graphs.}
%\vk{Tady bych asi smazal červený text a napojil zbývající větu na předchozí odstavec.}
%
%\vk{Zde jsem vymazal lemma o mazání hran pro dolní odhad $\lambda$ a přesunul to do zatím nepoužitých tvrzení na konci zdrojáku. V návaznosti na to navrhuji upravit uvozující formulaci níže.}
%
%\red{An analogous lemma holds for the parameter $\sigma$ using semivalid representations instead of valid representations.}
Motivated by \citep[Lem.~4]{Pendavingh_separation} establishing lower bound on $\lambda(G-e)$ for $e\in E(G)$ with special properties, we present a similar lemma for the parameter $\eta$.
\begin{lemma}\label{lemma:intersect_broken_sigma}
Let $G=(V,E)$ be a connected graph and let $M\in\mc{M}(G)$. Suppose $F\subseteq E$ is such that 
\[
\bigcup F\cap\supp(x)\neq\emptyset
\]
for every broken $x\in\ker(M)$ inducing more than three connected components in $\indg{G}{\supp(x)}$.
Then $\corank(M)-\abs{F}\leq\eta(G)$. If, moreover, $G-F$ is connected, then $\corank(M)-\abs{F}\leq\eta(G-F)$.
%Moreover, if $G-F$ is connected, we get that $\corank(M)-\abs{F}\leq\sigma(G-F)$.
\end{lemma}

%\vk{Je otázka, jestli ve znění nepoužít $\eta$ místo $\sigma$.}
%\martin{Klidne to muze zustat, jak to je, ale nemam na to vubec silny nazor.}
%\vk{Změnil jsem $\sigma$ na $\eta$.}
\begin{proof}
Let $L:=\set{y\in\ker(M)\colon y_{u}+y_{v}=0 \ \forall uv\in F}$. Clearly,
  $\dim(L)\geq\corank(M)-\abs{F}$. We show that $L$ is a semivalid
  representation of $G$, and for the `moreover' part, that it is also a semivalid representation of $G-F$.
%  \martin{Tady bych
%  si myslel, ze do teto vety patri `that'.}
%  \vk{V tomto případě není nutné, protože jde o přímý předmět. Viz např. \url{https://dictionary.cambridge.org/grammar/british-grammar/verb-patterns/verb-patterns-verb-that-clause}.}
%The claim then follows by Proposition~\ref{p:sigma_ex_rep_new}.

To verify that $L$ is a semivalid representation of $G$,
it is immediate that condition~\eqref{def:ex_rep_0} of Definition~\ref{def:ex_represent} is satisfied since it holds for every nonzero vector in $\ker(M)$ (e.g., see Lemma~\ref{lemma:basic}\eqref{lemma:basic_1}).
%\martin{Podminka 1 je trivialne splnena protoze to plati pro kazdy vektor z
%  $\ker M$? (Perron-Frobeniova veta?)}\vk{Ano, případně se místo Perron--Frobenia může člověk odkázat na Lemma~\ref{lemma:basic}\eqref{lemma:basic_1}, nebo na to, že $\ker(M)$ je ext. reprezentace.}
%    \martin{ Mirne bych se primlouval k tomu
%  nepouzivat "trivially". Nekoho by to mohlo provokovat. A jeste, kdyz je to
%  snadne, mozna by se s tim mohlo zacit. Bylo by to i poporade (teda v
%  zavislosti na tom, jestli by se presunulo i zduvodneni pro 4)}
%  \vk{OK, upravil jsem to a změnil to pořadí.}
Next we check condition~\eqref{def:ex_rep_1} of Definition~\ref{def:ex_represent}.
Assume it is not satisfied. Take a broken $y\in L$, which induces more than three connected components in $\indg{G}{\supp(y)}$. By the assumption on $F$, there is $uv\in F$ such that $\set{u,v}\cap\supp(y)\neq\emptyset$. This means that $y_{u}=-y_{v}\neq 0$.
However, this is impossible by Lemma~\ref{lemma:basic}\eqref{lemma:basic_2}.

Now we turn to condition~\eqref{def:ex_rep_2} of Definition~\ref{def:ex_represent}. Again, we assume that the condition is not satisfied. Take $y\in L$ which has inclusion-miminal support among nonzero vectors in $L$, but at least one of the graphs $\indg{G}{\supp_\pm(y)}$ is not connected.
%By considering $-y$ instead of $y$, we may assume that $\indg{G}{\supp_+(y)}$ is disconnected.
By the definition of $L$ and Lemma~\ref{lemma:basic}\eqref{lemma:basic_2}, $\bigcup F\subseteq V\setminus\supp(y)$. However, this means that $D:=\set{x\in\ker(M)\colon \supp(x)\subseteq\supp(y)}$ is a subspace of $L$.
On the other hand, Lemma~\ref{lemma:Pendavingh} says that $\dim(D)+1$ is equal to the number of connected components of $\indg{G}{\supp(y)}$. This means that $\dim(D)\geq 2$, which implies that there is $x\in D$ with strictly smaller support than $y$; a contradiction.

Lemma~\ref{lemma:basic}\eqref{lemma:basic_2} proves that the condition~\eqref{def:ex_rep_4} is satisfied as well; thus, $L$ is a semivalid representation of $G$.
%  \martin{Asi se pouziva i Lemma~\ref{lemma:basic}\eqref{lemma:basic_1}, ne?}\vk{To už bylo dokázáno o dva odstavce výše, nepřijde mi potřebné to tady opakovat.}

To verify that $L$ is a semivalid representation of $G-F$, we first observe that if we take a nonzero $y\in L$, none of the edges $uv\in F$ can have both endpoints in $\supp_+(y)$ or $\supp_-(y)$, since $y_{u}+y_{v}=0$.
Therefore, removing $F$ from $G$ cannot disconnect any of the connected components of $\indg{G}{\supp_+(y)}$. Consequently, $L$ satisfies both requirements \eqref{def:ex_rep_1} and \eqref{def:ex_rep_2} of Definition~\ref{def:ex_represent} for $G-F$.
Moreover, none of the edges of $uv\in F$ can have one endpoint in $\supp(y)$ and the other in $V\setminus \supp(y)$; again, because $y_u+y_v=0$. Thus, removing $F$ cannot change $N(\supp_\pm(y))$ nor $N(C)$ for any of the connected components $C$ induced by $\supp(y)$. Therefore, $L$ also satisfies the requirement~\eqref{def:ex_rep_4} of Definition~\ref{def:ex_represent}; we conclude that $L$ is a semivalid representation of $G-F$.
%  \martin{Jenom zpetna kontrola, rika tohle, ze kdyz mam hranu $uv$ z $F$, tak
%  vlastne $y_u$ i $y_v$ museji byt nulove, protoze mezi kladnym a zapornym
%  supportem nema vest hrana.}\vk{Mezi kladným a záporným supportem nesmí vést hrana jen v případě, že kladný nebo záporný support je nesouvislý, jinak tam hrana být může. O takovém případě je ale přesně podnímka \ref{def:ex_rep_4} definice \ref{def:ex_represent}}
%  \martin{Kde presne se vyuziva, ze $G-F$ je souvisly?}\martin{Tady jsem byl
%  zmateny. Snad OK.}
%  \vk{Že $G - F$ je souvislý se tady myslím nikde nevyužívá. Ale problém je, že
%  pokud je $G - F$ nesouvislý, tak nevím, co výše uvedený argument dává: pořád
%  asi $L$ splňuje body \eqref{def:ex_rep_0}--\eqref{def:ex_rep_4}
%  definice~\ref{def:ex_represent}, ale už nevím, jaký je vztah mezi $\dim(L)$ a
%  $\eta(G -F)$. Jak jsem ale psal už dříve nad zněním tohohle lemmatu, možná se
%  to i v případě $G-F$ nesouvislého dá nějak dokončit. Člověk potřebuje
%  dokázat, že pak je $L$ supportován jen na jedné komponentě $G-F$.}\martin{OK.
%  Asi to klidne muze zustat s predpokladem souvislosti.}
\end{proof}

The next lemma is an easy consequence of Lemma~\ref{lemma:basic}\eqref{lemma:basic_2}. It generalizes \citep[Lem.~5]{Pendavingh_separation}. 
%\vk{Trochu jsem tady zesílil tu formulaci v uvození výše.}
\begin{lemma}\label{lemma:d_reg_broken}
Let $G=(V,E)$ be a connected graph with maximum degree at most $d$ and let $M\in\mc{M}(G)$. Let $x\in\ker(M)$ be a broken vector. Then
\begin{enumerate}
	\item[{\normalfont(\mylabel{lemma:d_reg_broken_1}{i})}] $\indg{G}{\supp(x)}$ has at most $d$ connected components,
	\item[{\normalfont(\mylabel{lemma:d_reg_broken_2}{ii})}] if $\indg{G}{\supp(x)}$ has exactly $d$ connected components, then $\indg{G}{V\setminus\supp(x)}$ has no edges and $V\setminus\supp(x)=N(\supp(x))$.
\end{enumerate}  
\end{lemma}
\begin{proof}
Since $G$ is connected, Lemma~\ref{lemma:basic}\eqref{lemma:basic_2} implies that $N(\supp(x))$ is nonempty, and moreover, that every vertex in $N(\supp(x))$ is connected to each component of $\indg{G}{\supp(x)}$; thus, the number of such components cannot be greater than the maximum degree in $G$. This proves the first part.

For the second part, the above argument shows that $\indg{G}{N(\supp(x))}$ does not contain any edge. Consider a vertex $v\in V\setminus\supp(x)\setminus N(\supp(x))$.  Since $G$ is connected and $v\notin N(\supp(x))$, there must be a path from $v$ to $N(\supp(x))$. However, this is not possible, since all vertices in $N(\supp(x))$ have their neighbours only in $\supp(x)$.
\end{proof}

We restate here the following theorem of \citet[Thm.~5]{Pendavingh_separation}, which is very useful for proving upper bounds on $\mu(G)$.
\begin{thm}[{\citep[Thm.~5]{Pendavingh_separation}}]\label{thm:mu_edge_bound}
Let $G=(V,E)$ be a connected graph. Then either $G=K_{3,3}$, or $\abs{E}\geq\binom{\mu(G)+1}{2}$.
\end{thm}

\paragraph{Finite projective planes.}
Let $H_q$ denote the incidence graph of a finite projective plane of order $q$. It is a $(q+1)$-regular bipartite graph with parts of size $q^2+q+1$.
Using Theorem~\ref{thm:mu_edge_bound}, this implies that
\[
\mu(H_q)\leq\floor*{\frac{-1+\sqrt{1+8\abs{E(H_q)}}}{2}}=\floor*{\frac{-1+\sqrt{1+8(q^2+q+1)(q+1)}}{2}}.
\]

Let $A_q$ be the adjacency matrix of $H_q$. It is known that the spectrum of $A_q$ is
\[\br*{(q+1)^{(1)}, \sqrt{q}^{(q^2+q)}, -\sqrt{q}^{(q^2+q)}, -(q+1)^{(1)}};
\]
for a reference, see, e.g. \citep[Sec.~3.8.1, eq.~(3.38)]{regular_graphs} (for that reference, note that a finite projective plane of order $q$ is a symmetric BIBD with parameters $p,b = q^2 + q + 1$, $k, r = q+ 1$, $\ell = 1$).
We further define $M_q:=\sqrt{q}I-A_q$. Clearly, $M_q\in \mc{M}(H_q)$ and $\corank(M_q)=q^2+q$.
%\martin{Vis o nejake referenci? Ted jsem stravil nejaky cas hledanim a neuspel
%jsem. Spocitat bych to asi umel pres druhou mocninu. Popr. asi bych umel i
%explicitne vlastni vektory: ty by se ziskaly v podstate z toho prikladu, jak
%mit v supportu pouze dve primky a body na techto primkach s vyjimkou spolecneho
%bodu. V knize Andries E. Brouwer, Willem H. Haemers, Spectra of Graphs jsem
%nasel spravne hodnoty do radu 8, V tabulce 14.3 se jedna o hodnoty pro GH(1,q).
%Ale neprijde mi, ze by tam bylo zduvodneni, natoz obecne.}
%\vk{Určitě jsem to viděl v nějaké knize, dokonce myslím že ve více než jedné, když jsem si tak k tomu četl různé věci. Myslím, že dokonce to je i v té knize, kterou zmiňuješ, jen musí člověk namatchovat správnou třídu grafů, která obsahuje incidenční grafy projektivních rovin a potom v příslušné kapitole to myslím je. Časem to tady dodám, teď to nemá prioritu.}
%\vk{No, nemám tak hezkou referenci, jak bych si přál, ale nějaká je třeba toto \citep[Sec.~3.8.1, eq.~(3.38)]{regular_graphs} s poznámkou, že projektivní rovina je symetrický BIBD. Myslíš, že by to tak stačilo?}
%\martin{To je skvela reference! Kdybychom chteli byt honde opatrni, tak muzeme
%napsat parametry $p,b = q^2 + q + 1$, $k, r = q+ 1$, $\ell = 1$. Ale dokonce i
%bez poznamky o BIBDu by to podle me bylo OK.}
%\vk{Byl jsem raději opatrný a přidal to tam. :-)}

\begin{prop}
  \label{p:9_11}
$\mu(H_3)\leq 9$ and $\sigma(H_3)\geq 11$.
\end{prop}
\begin{proof}
The bound on $\mu(H_q)$ above gives $\mu(H_3)\leq 9$. Furthemore, $\corank(M_3)=12$. Now choose any edge $e$ of $H_3$. Since $H_3$ is 4-regular, $e\cap\supp(x)\neq\emptyset$ for every broken $x\in\ker(M_3)$ inducing more than three connected components in $\indg{H_3}{\supp(x)}$ by Lemma~\ref{lemma:d_reg_broken}. Thus, by Lemma~\ref{lemma:intersect_broken_sigma} and Proposition~\ref{p:sigma_ex_rep_new} we see that $\sigma(H_3)\geq\sigma(H_3-e)\geq\eta(H_3-e)\geq 11$. 
\end{proof}

%\vk{I realised that we were not completely efficient when we were thinking about it. We forgot that we are really proving lower bound on $\sigma(G-F)$ instead of $\sigma(G)$, where the set $F$ comprises of the edges removed to kill all broken cones of dim 3 and higher. I recall that at some point we have seen graphs for which we had bounds on $\sigma$ and $\mu$ equal, but where the number of edges was the boundary case; if it was one smaller, we would get separation. I have to look again into 3- and 4- regular graphs we were examining (Pappus, Desargues, etc...).}
%\martin{Je tohle jeste aktualni?}

%\red{We can push the separation between $\mu$ and $\sigma$ a bit further:
%\begin{prop}
%There is a graph $G$ such that $\mu(G)\leq 7$ and $\sigma(G)\geq 8$.
%\end{prop}
%}
%\vk{Tady jsem vypustil znění tvrzení a jen to uvedl slovně.}
The separation between $\mu$ and $\sigma$ can be pushed even further by removing a small part from $H_3$ to obtain a graph $G$ with $\mu(G)\leq 7$ and $\sigma(G)\geq 8$, as was announced in Theorem~\ref{t:gap} in the introduction.
\begin{proof}[Proof of Theorem~\ref{t:gap}]
We choose three vertices $v_1, v_2, v_3$ of $H_3$ corresponding to three points of the finite projective plane of order $3$ not lying on a single line. Let $G':=H_3-\set{v_1, v_2, v_3}$.
We observe that $G'$ contains three vertices of degree two, since every two points of a projective plane lie on a single line. Next, we choose an edge $e\in E(G')$ adjacent to a vertex of degree three in $G'$ and set $G:=G'-e$.

Observe that $G$ contains four vertices of degree two; for each of these four vertices we choose one of the two edges incident to it and put it into a set $F$. We write $G/F$ for the graph resulting from a contraction of the edges of $F$ in $G$.
Since a subdivision of edges preserves $\mu(H)$ for graphs $H$ with $\mu(H)\geq 3$ by \citep[Thm.~2.12]{CdV_main}, we get that $\mu(G)=\mu(G/F)$.
The graph $G/F$ has $4\times 13-12-1-4=35$ edges. This  means that $\mu(G)=\mu(G/F)\leq 7$ by Theorem~\ref{thm:mu_edge_bound}. On the other hand, a removal of a vertex can decrease $\sigma$ by at most $1$ \citep[Thm.~28]{sigma}.
%\red{If we argue that} 
  As $\sigma(H_3-e)\geq 11$ (this was substantiated in the proof of Proposition~\ref{p:9_11} above), we deduce that $\sigma(G)\geq 8$.
\end{proof}

The proof of the following proposition is a direct generalization of the proof of \citep[Thm.~1]{Pendavingh_separation}.
\begin{prop}\label{prop:bound_on_lambda}
Let $G=(V,E)$ be a connected graph of maximum degree at most $d$ and $M\in\mc{M}(G)$. Then $\lambda(G)\geq \corank(M)-d+1$.
\end{prop}
\begin{proof}
Let $x\in\ker(M)$ be a broken vector. The subspace
\[
D(x):=\set{y\in\ker(M)\colon \supp(y)\subseteq \supp(x)}
\]
has dimension at most $d-1$ by Lemma~\ref{lemma:Pendavingh} and Lemma~\ref{lemma:d_reg_broken}\eqref{lemma:d_reg_broken_1}. Let $B\subseteq \ker(M)$ be a set consisting of all broken vectors $x$ with inclusion-maximal support among broken vectors in $\ker(M)$. This implies that for every broken $y\in\ker(M)$ there is $x\in B$ such that $y\in D(x)$. Therefore, every broken vector in $\ker(M)$ is contained in a linear subspace of $\ker(M)$ of dimension at most $d-1$.

Since the number of different subsets $\supp(x)\subseteq V$ is finite, the number of distinct subspaces $D(x)$ for $x\in B$ is finite as well. 
%Let $x,y\in B$. If $S(x)=S(y)$, then $C(x)=C(y)$ by Observation~\ref{obs:S_determines_C_general}. This implies that
%$$
%V\setminus\supp(x)=S(x)\cup C(x)=S(y)\cup C(y)=V\setminus \supp(y).
%$$
%We deduce that $\supp(x)=\supp(y)$, and hence, $D(x)=D(y)$. This shows that $S(x)$ uniquely determines $D(x)$ for every $x\in B$.
%
%Since the set $S(x)$ is a cut in the graph $G$ and the number of different cuts in the graph $G$ is finite, the argument above shows that the number of distinct subspaces $D(x)$ for $x\in B$ is finite as well.
Therefore, there is a linear subspace $L\subseteq \ker(M)$ of dimension at least $\corank(M)-d+1$ such that for every $x\in B$ it holds that $L\cap D(x)=\set{\mb{0}}$.
Consequently, $L$ is a valid representation of $G$, which finishes the proof.
\end{proof}

%\vk{Tady jsem slovně přepsal hlavní část znění, které se bude dokazovat.}
%\martin{OK. Jenom `the introduction' obcas pises s malym `i' a obcas s velkym
%`I'. Mas nejakou preferenci?}\vk{Nejdříve jsem to psal s "I", pak jsem zjistil, že by se to asi mělo psát s malým, tak jsem se to snažil všude přepsat. Měl jsem za to, že jsem to všude změnil, ale ukázalo se, že jsem přehlédl jeden výskyt nad důkazem věty~\ref{t:gap}. Projel jsem to znovu, a už jsem další "Introduction" nenašel, snad je to tedy už všude OK. Díky za upozornění!}
Applying this proposition to the finite projective planes we immediately obtain an asymptotic separation of order $\mu(H_q)\in
O\br*{q^{3/2}}$ and $\sigma(H_q)\geq\lambda(H_q)\geq q^2$, which was stated in Theorem~\ref{t:asymptotic} in the introduction.

%\martin{Zmenil jsem $q^{\frac 32} \to q^{3/2}$ pro lepsi citelnost. (Je to i
%standard v takovych pripadech, ne?)}

%\red{\begin{cor}
%Let $q\in\N$ be such that a finite projective plane of order $q$
%exists\footnote{This includes all prime powers $q$.}. Then $\mu(H_q)\in
%O\br*{q^{3/2}}$, while $\sigma(H_q)\geq\lambda(H_q)\geq q^2$.
%\end{cor}
%}
\begin{proof}[Proof of Theorem~\ref{t:asymptotic}]
Since $\corank(M_q)=q^2+q$ and the degree of every vertex in $H_q$ is $q+1$, Proposition~\ref{prop:bound_on_lambda} implies that $\lambda(H_q)\geq q^2$. The fact that $\lambda(G)\leq\sigma(G)$ for every graph $G$ was proven by \citet[Thm.~40]{sigma}, as was already mentioned before.

The upper bound on $\mu(H_q)$ follows directly from Theorem~\ref{thm:mu_edge_bound}.
\end{proof}

%\martin{Pridany `acknowledgment' nize.}

\section*{Acknowledgment}
We would like to thank Radek Hu\v{s}ek and Robert \v{S}\'{a}mal for general
discussions on the parameter $\mu$. We would also like to thank Arnaud de
Mesmay for pointing us to the paper~\citep{sigma} and Rose McCarty for
pointing us to~\cite{foisy03}.

%\renewcommand{\bibname}{Bibliography}
%\bibliographystyle{plainnat}
%\bibliographystyle{alpha}
%\bibliography{bibliography}

%\addbibresource{bibliography.bib}
\printbibliography

\newcommand{\poly}{\textrm{poly}}

%\bigskip

\appendix

\section{An explicit PSPACE algorithm for \texorpdfstring{$\mu$}{\unichar{"1D707}}}
\label{a:mu}

In this appendix we describe an explicit algorithm that for every graph $G=(V,E)$ on $n$ vertices and every $k\in\set{0,\ldots, n}$ decides in \emph{space} polynomial in $n$ whether $\mu(G)\geq k$ or not.
The strategy is to produce an existential sentence $\phi_{G,k}$ in the language $\mathfrak{L}$ of the first-order theory of the reals\footnote{The language $\mathfrak{L}$ allows one to use real variables and symbols $=, \neq, \leq, \geq, <, >, 0, 1, +, -, \cdot$, logical connectives and quantifiers over the real numbers. Thus, one can use equalities and inequalities of polynomials of several real variables with integer coefficients.} of length polynomial in $n$ which is true if and only if $\mu(G)\geq k$. The rest then follows by the algorithm of \citet{PSPACE_existsR} for deciding the existential theory of the reals ($\exists\R$).

\paragraph{Notation.}
We write $E_{i,j}$ for the matrix with one at the position $(i,j)$ and zero everywhere else.
Let $G=(V,E)$ be a graph and $n:=\abs{V}$. We write $\mc{O}(G)$ for the subset of $\R^{V\times V}$ consisting of symmetric matrices $M$ such that $M_{u,v}<0$ for every $uv\in E$ and $M_{u,v}=0$ for every $uv\notin E$. There is no requirement on the diagonal entries of $M$.

Let $p:=\binom{n}{2}-\abs{E}$. Given a matrix $M\in\mc{O}(G)$, we define a $p\times n^2$ matrix $N(M)$ as follows: the columns of $N(M)$ consist of vectors of the form
\[
\br*{ME_{i,j}+E_{i,j}^T M}_{uv\notin E}
\]
where $1\leq i,j\leq n$. That is, we take the matrix $ME_{i,j}+E_{i,j}^T M$ and turn its entries corresponding to the nonedges of $G$ into a vector (assuming some fixed ordering on the nonedges of $G$), which then constitutes a column of the matrix $N(M)$. The role of $N(M)$ will be explained below.

%\martin{Nemelo by nahore byt radeji $1 \leq i < j \leq n$? Ma byt $\binom n2$
%sloupcu, ne?}

\medskip
The definition of the parameter $\mu$ says that $\mu(G)\geq k$ if and only if there is a symmetric matrix $M\in \mc{O}(G)$ with exactly one negative eigenvalue and corank at least $k$ that satisfies SAH (see Subsection~\ref{s:CdV}).
It is not difficult to see that one can transfer this statement into a formula in the language $\mathfrak{L}$. Additionaly, one gets easily an $\exists\forall$-sentence of length polynomial in $n$.
The reason for the presence of the universal quantifier is the definition of SAH, which is a condition on \emph{all} matrices of certain form. The main ingredient in changing this formula into an existence formula is the following equivalent characterization of SAH by \citet{SAH}:

\begin{thm}[{\citep[Thm.~31(a)]{SAH}}]\label{thm:SAH}
$M\in\mc{O}(G)$ satisfies SAH if and only if the matrix $N(M)$ has a full rank, i.e., its rank is $\binom{n}{2} - \abs{E}$. 
\end{thm}

Informally, this theorem allows us to express that $M$ satisfies SAH as a formula saying `there is a matrix $N$ of full rank such that $N=N(M)$'. Clearly, given $M$, the matrix $N(M)$ can be constructed in time polynomial in the length of the description of $M$. 
In addition, we use a simple trick that enables us to prescribe the signs of the eigenvalues of $M$ and the rank of $N$; instead of searching directly for $M$ and $N$, we look for their eigendecomposition and singular value decomposition, respectively.

\begin{prop}
There is an algorithm working in time polynomial in $n$ that given as an input a graph $G=(V,E)$ with $n=\abs{V}$ and any $k\in\set{0, \ldots, n-1}$ constructs an $\exists$-sentence $\phi=\phi_{G,k}$ in the prenex normal form in the language $\mathfrak{L}$ of size $O(n^6)$ using $O(n^4)$ quantified variables such that $\mu(G)\geq k$ if and only if
$\phi_{G,k}$ is true.
\end{prop}
%\vk{Upravil jsem znění aby říkalo, že tu formuli jde algoritmicky efektivně zkonstruovat, ne že jen existuje. Bylo to sice asi víceméně jasné, ale pro pořádek.}
\begin{proof}
Let $p:=\binom{n}{2}-\abs{E}$. The formula $\phi:=\phi_{G,k}$ will have a form equivalent to the following:
\[
(\exists L,D\in\R^{n\times n}, A\in\R^{p\times p}, B\in\R^{n^2\times n^2}, S\in\R^{p\times n^2})\psi(L,D,A,B,S),
\]
where $\psi$ is a quantifier-free formula formed as a conjunction of polynomial equalities and inequalities with variables corresponding to entries of $L,D,A,B$ and $S$.
Every element of $L, A$ and $B$ will be a real variable. On the other hand, since $D$ and $S$ will always represent diagonal matrices, only their diagonal entries will be real variables, their off-diagonal entries are always assumed to be zero.
%\vk{Tady jsem trošku pozměnil formulaci, aby pak bylo jasné, pro co se definuje proměnná a pro co ne. Později jsem měl formulaci jako "formula saying that D is a diagonal matrix", která by mohla být trochu matoucí. Proto jsem tady přidal tohle podrobnější vysvětlení a níže jsem z těch formulací odstranil "D is a diagonal matrix" a podobně.}

For brevity, we write $M:=LDL^T$; this matrix plays the same role as in the discussions above. That is, $M$ certifies that $\mu(G)\geq k$. The matrices $L$ and $D$ represent the eigendecomposition of $M$---the matrix $D$ is a diagonal matrix with the spectrum of $M$ on its diagonal and $L$ is an orthogonal matrix representing the corresponding eigenbasis of $M$.
%The matrices $M$ and $D$ have the same signature by Sylvester's law of inertia.

Similarly, we write $N:=ASB^T$. The matrix $N$ plays the role of $N(M)$ and $A,S,B$ represent the singular value decomposition of $N$. The singular values of $N$ are the diagonal entries of $S$ and $A,B$ are orthogonal matrices.
Since the rank of $N$ is equal to the rank of $N N^T$ and the singular values of $N$ are the square roots of the eigenvalues of $N N^T$, we see that $N$ has full rank if and only if all the singular values of $N$ are positive.

The formula $\psi(L,D,A,B,S)$ is a conjunction of the formulas expressing the following\footnote{For better readability, we do not write the formulas exactly in the language $\mathfrak{L}$, but it should be evident how to rephrase them in that language.}:
\begin{itemize}
	\item The formula saying that the diagonal of $D$ is 
	\[
	(\lambda_1, 0, \ldots, 0, \lambda_{k+2},\ldots, \lambda_n),
	\]
	where the number of hard-coded zero entries is $k$, together with specifying the requirements $\lambda_1<0$ and $\lambda_i\geq 0$ for $i\in\set{k+2, \ldots, n}$. This subformula has thus size only $O(n)$. Recall that $D$ is assumed to be a diagonal matrix, so its off-diagonal entries are also hard-coded to be zero (i.e., we do not need any formula to specify this).
	\item The formula $LL^T=I_n$. This can be written as a conjunction of $O(n^2)$ formulas of size $O(n)$.
	\item The formula expressing $M\in\mc{O}(G)$. Clearly, this can be written as a conjunction of $O(n^2)$ formulas, each of length $O(n)$.
	\item The formula saying that the diagonal of $S$ is strictly positive. This subformula has size $O(p)$, which is in $O(n^2)$. Recall that the matrix $S$ is also assumed to be diagonal, and thus, its off-diagonal entries are hard-coded to be zero. 
	\item The formula $A^{T}A=I_p$. This is a conjunction of $O(p^2)$ formulas of size $O(p)$. In total, this is an $O(n^6)$-long subformula.
	\item The formula $B^{T}B=I_{n^2}$. This is a subformula of size $O(n^6)$.
	\item The formula saying that $ASB^T=N(M)$. This is a conjunction of $O(n^4)$ formulas of length $O(n^2)$. In total, we again have a $O(n^6)$-long subformula.
\end{itemize}

Consequently, the size of $\phi$ is $O(n^6)$ and it contains only one (existential) quantifier over $O(n^4)$ variables. It is also clear that $\phi$ is constructible in time polynomial in $n$.
\end{proof}

The preceding discussion immediately implies the following corollary: 
%\vk{Všiml jsem si, že já názvy složitostních tříd sázím tučně, zatímco Ty ne. Máš nějakou preferenci?}
%\martin{Pardon, to nebyl umysl. Preferenci nemam, vyber si.}
%\vk{Obvykle bych asi preferoval tučnou variantu, ale vzhledem k tomu, používáme tučné $\PP$ pro polytop už v úvodu, tak je asi lepší ty složitostní třídy psát netučně.}

\begin{cor}
There is an explicit algorithm that computes the value of $\mu(G)$ in space polynomial in $\abs{V}$ for any graph $G=(V,E)$.
\end{cor}
%\vk{Upravil jsem znění důsledku, aby mluvil o \emph{explicitním} algoritmu.}
%\vk{Z hlediska časové složitosti z toho vyleze nějaký algoritmus složitosti $O(2^{n^k})$, kde ale $k>4$. Takže ten Tvůj algoritmus pro počítání $\sigma$ přes obstrukci je asi pořád lepší. Navíc ty asymptoticky nejlepší algoritmy pro rozhodování formulí, alespoň podle Wikipedie, nejsou prakticky vůbec použitelné. Prakticky asi běží nejlíp nějaká dvojitě exponenciální procedura.}

%Let us denote by $N$ the matrix with columns of the form $(ME_{i,j}+E_{i,j}^TM)_{uv\notin E(G)}$, for all $1\leq i,j\leq n$. This is a $\br*{{n \choose 2} - \abs{E(G)}}\times n^2$ matrix. The condition that it has rank \emph{at most} $p\in\N$ can be expressed using its rank decomposition: exist $Y\in\R^{\br*{{n \choose 2} - \abs{E(G)}}\times p}$ and $Z\in\R^{p\times n^2}$ such that $YZ=N$.

%\vk{Úplně za ztracené ještě nepovažuji ani použití semidefinitního programování. Tam asi hlavní překážka je vyjádřit SAH pomocí semidefinitních matic. Přímo to asi nejde, ale je možné, že když člověk použije semidefinitní reformulaci $\mu$, dá se BUNO předpokládat něco i o $X$. Zatím nevím. Formulovat podmínku $\corank(M)\geq r$, pokud se nepletu, půjde podobně jako výše, pomocí varianty Choleského dekompozice, tzv. LDL dekompozice.}
%\vk{UPDATE: Tady by zase mohlo být použitelné něco z tohoto článku: \url{https://www.sciencedirect.com/science/article/pii/S0024379514001591}.}
%\vk{Zakomentoval jsem tady svou starší poznámku o SDP a $\mu$. Teď už si nemyslím, že by šlo počítat $\mu$ pomocí SDP.}

\ifcomb
\else

\newcommand{\del}[1]{\widetilde{#1}}
\newcommand{\eq}{\mathtt{eq}}
\newcommand{\interior}{\mathrm{int} \, }
% Dve definice vyse by mely byt zmeneny pomoci \DeclareMathOperator ale ot jde
% jen v hlavicce.
\newcommand{\ceq}{C_{4, \eq}}
\newcommand{\gteq}{\tilde g_{\#, \eq}}
\newcommand{\ckeq}[1]{C_{#1, \eq}}
\newcommand{\cteq}{C_{3, \eq}}
\newcommand{\zeq}{Z_{4, \eq}}
\newcommand{\zkeq}[1]{Z_{#1, \eq}}
\newcommand{\Ceq}{C^4_{\eq}}
\newcommand{\Ckeq}[1]{C^{#1}_{\eq}}
\newcommand{\obst}{\mathbf{o}}

\section{Recognition of graphs with \texorpdfstring{$\sigma(G) \leq 5$}{\unichar{"1D70E}(G)<=5}}
\label{a:sigma}

%\martin{Tady chci napsat neco o pocitani $\sigma$.}
In this appendix, we show that $\sigma(G) > 5$ can be certified in
polynomial time by an explicit certificate (i.e., not via an unknown
forbidden minor). We recall from the introduction
that we provide here the full details in order to provide a rigorous support forthe motivation mentioned in the introduction. Otherwise, the contents of this appendix should
be regarded as a basis for a future work.
%\vk{Tady navrhuji změny v reakci na Tvůj email.}

Throughout Appendix~\ref{a:sigma}, we change our previous convention and assume that all polyhedra (and their faces) are closed.
%\vk{Tady bych to formuloval jinak, stará formulaci bych asi nepochopil (vypadalo to, jako že původní konvence byla, že jsou uzavřené, a to teď adaptujeme, ale chybí na co.)}
%\martin{Urcite jsem to nenapsal dobre, diky! Asi jsem moc nerozlisoval `adapt' a
%'adapt to' coz zmenilo vyznam. Nicmene podle (mozna mizerneho) internetoveho
%slovniku, co obcas pouzivam, je `adapt' nejen prizpusobit, ale i prevzit.
%Nemohla by i Tvoje verze zpusobit zmateni? Coz takhle radeji pouzit `modify'?}
%\vk{Ano, modify je určitě lepší. Snažil jsem se předtím udělat co nejmenší změnu. Použil jsem nakonec "change", které má asi nejjasnější význam.}

\subsection{Exponential time algorithm}

First we describe the exponential time algorithm mentioned in the
introduction.

\paragraph{Polytopal and polygonal complexes.} By a \emph{polytopal complex} we mean a polyhedral complex where each polyhedron is bounded
(i.e., a polytope). A \emph{polygonal complex} is a polytopal comlex of
dimension at most $2$.

\paragraph{$2$-closure.} A polygonal complex central to the contents of this
section will be so called $2$-closure of a graph. Let $G = (V,E)$ be a graph.
Let $\C_2(G)$ be the $2$-dimensional CW-complex obtained from $G$ by
attaching a polygonal disk $D_s$ to every cycle $s$ in $G$. 
\Citet{sigma} define a \emph{closure} of $G$ as a CW-complex $\C$ such that (i) $\C^{(1)}$ equals to $G$ and (ii) for each $i \geq 0$ and each $U \subseteq V$ that induces a connected subgraph of $G$, the higher homotopy group $\pi_i(\C^{(i+1)}[U])$ is trivial, where $\C[U]$ denotes the subcomplex of $\C$ induced by $U$. The complex $\C_2(G)$ satisfies the condition (i) and it also satisfies the condition (ii) for $i \leq 1$. From the proof of Theorem~19 in~\citep{sigma}, it follows that $\C_2(G)$ can be extended to a closure of $G$, thus it is appropriate to call $\C_2(G)$ a \emph{$2$-closure} of $G$. 

It follows from~\citep{sigma} that $\sigma(G) \leq 5$ if and only if $\C_2(G)$
admits an even map into $\R^4$; see Proposition~\ref{p:Iz} below for precise
statement convenient for our setting. As mentioned in the introduction,
determining whether a $2$-complex admits an even map into $\R^4$ is known to be
easy via equivariant obstruction theory (it is equivalent to vanishing of the
$\Z_2$-reduction of so called van Kampen obstruction). Usually, this is set in
the language of simplicial complexes but the extension to polygonal complexes
is straightforward. Below we provide the details needed for explanation of our
algorithm (and proof of its correctness).

\paragraph{Deleted product.} Given a polygonal complex $\P$, by
$\del{\P}$ we denote the \emph{deleted product} of $\P$. This is the
polytopal complex with faces of the form $\eta \times \tau$ where
$\eta$ and $\tau$ are disjoint faces of $\P$. (Because of the convention
for this section $\eta$ and $\tau$ are closed. Therefore their disjointness
means that they do not share a vertex.) Note that $\del{\P}$ is a
$4$-dimensional complex as soon as $\P$ contains a pair of disjoint $2$-faces.
There is a natural $\Z_2$ action on $\del\P$ swapping $\eta \times \tau$ and
$\tau \times \eta$.
%and $\overline{\C(G)}$ be the quotient of $\del{\C(G)}$ under the involution
%exchanging the coordinates.

\paragraph{Chains and symmetric chains.}
Given a polytopal complex $\P$ by $C_k(\P)$ we denote the space of
\emph{$k$-chains}
of $\P$ (over $\Z_2$; all considerations in this section will be over $\Z_2$).
This means that the elements of $C_k(\P)$ are formal linear combinations
\[
\sum \alpha_\eta \eta
\]
where $\alpha_\eta \in \Z_2$ and the sum is over all $k$-faces $\eta$ of $\P$.
The boundary operator $\partial\colon C_k(\P) \to C_{k-1}(\P)$ is defined so that a
$k$-face $\eta$ is mapped to the sum of all $(k-1)$-faces of $\eta$ and then it
is extended linearly to $C_k(\P)$. An element $z \in C_k(\P)$ is a $k$-cycle if
$\partial z = 0$. The space of $k$-cycles is denoted $Z_k(\P)$. 
Note that we carefully distinguish graph-theoretic cycles in graph $G$
(connected subgraphs where every vertex has degree $2$) and $k$-cycles. For
comparison, subgraphs of $G$ such that every vertex has even degree would be $1$-cycles in $Z_1(G)$, but we
will never need them.

%\martin{TODO: Nekde upozornit na grafove cykly.}

%In a special case when $\del \P$ is the deleted product of a polygonal complex
%$\P$, we also consider the space $\ckeq k (\del \P)$ of \emph{symmetric $k$-chains}. These
%are the chains 
%$$
%\sum \alpha_{\eta \times \tau} \eta \times \tau
%$$
%from $C_k(\del \P)$ which in addition satisfy $\alpha_{\eta \times \tau} =
%\alpha_{\tau \times \eta}$ for every $k$-face $\eta \times \tau$ of $\del \P$. 
%Note that these $k$-chains are equivariant with respect to the aforementioned
%$\Z_2$ action on $\del \P$.
%The subgroup of \emph{symmetric $k$-cycles} is denoted $\zkeq k (\del \P)$ (these
%are symmetric $k$-chains $z$ such that $\partial z = 0$).
%
%We will often encounter the $4$-dimensional \emph{equivariant chain space of
%$\del{\C_2(G)}$}
%denoted  $\ceq(\del{\C_2(G)})$
%which is the group of \emph{symmetric 4-chains over $\del{\C_2(G)}$}, that is, the
%vector space of formal linear combinations

In even more special case when $\P = \C_2(G)$, we simplify the notation for
symmetric chains in $\ckeq k (\C_2(G))$ so that we write them in a form

\begin{equation}
  \label{e:ceq}
  \sum\limits \alpha_{r \times s} \cdot D_r \times D_s.
\end{equation}
That is we simplify  $\alpha_{D_r \times D_s}$ to $\alpha_{r \times s}$
where $r$ and $s$ are disjoint cycles of $G$. If we further set
$\alpha_{\{r,s\}} := \alpha_{r \times s} = \alpha_{s \times r}$,
then~\eqref{e:ceq} can be rewritten as
\begin{equation}
  \label{e:ceq_unordered}
  \sum\limits \alpha_{\{r,s\}} \cdot (D_r \times D_s + D_s \times D_r)
\end{equation}
where the sum is over all unordered pairs $\{r,s\}$ of disjoint cycles.

\paragraph{Symmetric cochains.}
%Finally, b 
Given a polygonal complex $\P$, by $\Ckeq k(\del{\P})$ we denote the space of
corresponding symmetric cochains, that is, linear maps $m \colon
\ckeq k (\del{\P}) \to \Z_2$ satisfying $m(\eta \times \tau)
= m(\tau \times \eta)$ for any $k$-face $\eta \times \tau$ of $\del \P$.
%of disjoint cycles of $G$.

\paragraph{General position and almost general position.}
Let $\P$ be a polygonal complex. We say that a PL (piecewise linear) map
$f\colon |\P| \to \R^4$ is in \emph{general position} if the following two conditions
are satisfied.

\begin{enumerate}[(i)]
 \item Whenever $\eta$ is an edge of $\P$, $x \in \eta$, $\tau$ is a $2$-face
   of $\P$, $y \in \tau$, then $f(x) = f(y)$ implies $x = y$.
 \item Whenever $\eta$ and $\tau$ are distinct $2$-faces of $\P$, then $f(\interior
   \sigma)$ and $f(\interior \tau)$ meet in a finite number of points and each
    such point is a transversal crossing. (The symbol $\interior$ denotes the
    interior.)
\end{enumerate}

We say that $f$ is in \emph{almost general position} if it satisfies (i) and it
satisfies (ii) for every pair $\eta$ and $\tau$ of disjoint (instead of
distinct) $2$-faces of $\P$.

%\martin{Odsud dal muze byt velmi rozvrtane znaceni. Prestal jsem editovat ve
%chvili, kdy se objevil problem s minimalnimi cykly.}

\paragraph{Crossing cocycle.} Given a PL
map $f\colon |\P| \to \R^4$ in almost general position, we
define the \emph{crossing cocycle} $\obst_f \in \Ceq(\del{\P})$ by setting
$\obst_f(\eta \times \tau + \tau \times \eta)$ to be the number of crossings of
$f(\eta)$ and $f(\tau)$ if $\eta$ and $\tau$ are disjoint $2$-faces of $\P$. 
Then we extend $\obst_f$ linearly to $\ceq(\del{\P})$.\footnote{The reader
familiar with the van Kampen obstruction may observe that $\obst_f$ is a
representative of the cohomology class of the van Kampen obstruction (modulo
2).} According to the definition
of even map in~\citep{sigma}, the map $f$ is \emph{even} if and only if $\obst_f = 0$. Given
$z \in \zeq(\del{\P})$, $\obst_f(z)$ coincides with $I(z,f)$ defined
in~\citep[Sec. 4]{sigma} in our special case when $f$ is an almost general position PL
map. As~\citet{sigma} argue $I(z,f)$ is independent of the choice of $f$. Then
it is possible to define $I(z) = I(z,f)$ where $f$ is an arbitrary general
position PL map. Note that $I$ is a linear map from $\zeq(\del{\C_2(G)})$ to $\Z_2$.

%\martin{TODO: Remark on relation to the van Kampen obstruction?}

The following proposition is not explicitly mentioned in~\citep{sigma}.
However, it immediately follows from Theorem~30 in~\citep{sigma}
(used with $n=4$) and the equivalent definition of $\sigma$ via $I(z)$
in~\citep[Sec.~6]{sigma}. 

\begin{prop}[\citep{sigma}]
  \label{p:Iz}
  We get $\sigma(G) \leq 5$ if and only if $I(z) = 0$ for every $z \in
  \zeq(\del{\C_2(G)})$.
\end{prop}

%\begin{prop}[\citep{sigma}]
%  The $2$-closure $\C_2(G)$ has a general position PL even mapping into $\R^4$ if and only if
%  $\sigma(G) \leq 5$.
%\end{prop}
%
%\begin{proof}
%  First of all, Theorem~30 in~\citep{sigma} (used with $n=4$) and the equivalent definition of
%  $\sigma$ via $I(z)$ in~\citep[Sec.~6]{sigma} give that $\sigma(G) \leq 5$ if
%  and only if $I(z) = 0$ for every $z \in \zeq(\del{\C(G)})$
%  
%  First assume that there is a general position PL even mapping $f \colon
%  |\C_2(G)| \to \R^4$. Then $\obst_f = 0$. In particular $\obst_f(z) = 0$ for any
%  $z \in \zeq(\del{\C(G)})$. 
  
%  According to~\citep[Thm. 30]{sigma}, used with
%  $n=4$, and the equivalent definition of $\sigma$ in the beginning of section~6
%  of~\citep{sigma} (via $I(z)$, defined in~\citep{sigma}) this means
%  that $\sigma \leq 5$.

% Now let us assume that $\sigma(G) \leq 5$. Then again via Theorem~30 and the
%  aforementioned equivalent definition of $\sigma$ we get that $\obst_f(z) = 0$
%  for any general position PL mapping $f$. TODO. (SOME?)
%
%\end{proof}

\paragraph{Testing $\sigma(G) \leq 5$ in exponential time.}
Now we explain a simple algorithm for testing whether $\sigma(G) \leq 5$ in
exponential time via Proposition~\ref{p:Iz}.

%Now, let us describe the promised algorithm. 
Let $z^1, \dots, z^t$ be a basis
of $\zeq(\del{\C(G)})$. The value $t$ as well as size of each $z^i$ is
polynomial in the size of $\del{\C(G)}$; however, the size of $\del{\C(G)}$
might be (at most) exponential in size of $G$. 

Because of linearity of $I$, it is sufficient to test whether $I(z^i) = 0$ for
every $i \in [t]$ due to Proposition~\ref{p:Iz}. Each such test can be performed in time polynomial
in size of $z^i$. Indeed, it is sufficient to consider arbitrary general
position PL map $f\colon |\C_2(G)| \to \R^4$. Then we evaluate $\obst_f(z^1),
\dots, \obst_f(z^t)$. A good particular choice when it is easy to evaluate
$\obst_f(z^i)$ is to map the vertices of $G$ to the moment curve (as
in~\citep{matousek_tancer_wagner11}) pick a fixed triangulation of every disk
$D_r$ and extend the map linearly.\footnote{
This map is only in weakly general position which is of course sufficient for
evaluating $\obst_f(z^i)$.  
%  This is not really a general
%position PL map because disks sharing vertices may share triangles in the
%image. However, vertex-disjoint disks are mapped in mutually general position,
%and it is easy to check that this is sufficient for evaluating $I(z)$
%correctly. 
(Alternatively, it would be possible to triangulate each disk $D_r$ so
that we introduce one new vertex in the barycentre obtaining a truly general
position map.)}

%\martin{TO BE CONTINUED HERE.}
%If one of this values is nonzero,
%then an even map $f\colon \C(G) \to \R^4$ does not exist, otherwise, it exists.

%Given a graph $G$, we would like to know whether there exists an even map
%$f\colon \C(G) \to \R^4$. We can relatively easily deduce an exponential time
%algorithm for this problem based on the following lemma.
%
%\begin{lemma}
%    Let $g\colon \C(G) \to \R^4$ be an arbitrary general position PL map.
%      An even map $f\colon \C(G) \to \R^4$ exists if and only if $\obst_g(z) =
%      0$ for every
%        equivariant $4$-cycle $z \in \zeq(\del{\C(G)})$.
%\end{lemma}

\subsection{Speed-up}
%\martin{Tohle je ted aktualni sekce (ale znaceni vychazi z predchozi).}

Let $n$ be the number of vertices of $G = (V,E)$, where $V = [n]$.
Let $\Delta_{n-1}$ be the
$n$-simplex with vertex set $V$. Note that $G$ is a subgraph of the
$1$-skeleton $\Delta_{n-1}^{(1)}$. We will first define a suitable map $g
\colon |\C_2(G)| \to |\Delta_{n-1}^{(2)}|$. We set $g$ as identity on $G =
\C^{(1)}_2(G)$. For every cycle $r$ in $G$ we triangulate $D_r$ so that every
triangle in the triangulation contains the minimal vertex of $r$, and we
correspondingly map $D_r$ to $|\Delta_{n-1}^{(2)}|$, that is, a triangle in
$D_r$ with vertices $i$, $j$, $k$ is mapped to the triangle with vertices $i$,
$j$, $k$ of $\Delta_{n-1}^{(2)}$. (Note that if $r$ and $s$ are two distinct
cycles of $G$, then $g(D_r)$ and $g(D_s)$ may easily overlap in some triangle
although the disks $D_r$ and $D_s$ may overlap only on the boundary.) Note also
that $g(D_r)$ is always a disk. 

%\martin{Tady zacinaji podstatnejsi zmeny. (Ale zavedeni prostoru cyklu apod. se
%taky trochu zmenilo, vyznam by mel ale byt stejny.)}

Given a cycle $r$ of $G$, let $g_{\#}(D_r) \in C_2(\Delta_{n-1}^{(2)})$ be the
chain induced by $g$ (that is, the sum of the triangles triangulating
$g(D_r)$). Also the map $g$ induces a $\Z_2$-equivariant map $\tilde g \colon
|\del{\C_2(G)}| \to |\del{\Delta_{n-1}^{(2)}}|$ given by $\tilde g (x,y) =
(g(x), g(y))$. This map further induces an equivariant chain homomorphism
$\gteq \colon \ceq(\del{\C_2(G)}) \to
\ceq(\del{\Delta_{n-1}^{(2)}})$ which we explicitly describe below.

First let us assume that $c = \tau_1 + \dots + \tau_k$ and $c' = \tau'_1 +
\dots + \tau'_\ell$ are two chains in $C_2(\Delta_{n-1}^{(2)})$ such that for
every $i \in [k]$ and $j \in [\ell]$, $\tau_i$ and $\tau_j$ are triangles which
are disjoint. Then we set 
$c \times c' := \sum_{i,j=1,1}^{k,\ell} \tau_i \times \tau'_j$. We remark that
$c \times c' + c' \times c$ belongs to $\ceq(\del{\Delta_{n-1}^{(2)}})$.

Now, given two disjoint cycles $r$ and $s$ of $G$ %and 
%the chain $c_{\{r, s\}} := D_r \times D_s + D_s \times D_r$, 
we set
\begin{equation}
  \label{e:gteq}
  \gteq (D_r \times D_s + D_s \times D_r) := g_{\#}(D_r) \times g_{\#}(D_s) + g_{\#}(D_r) \times
g_{\#}(D_s)
\end{equation}
(adapting the convention from the previous paragraph). Then we
extend $\gteq$ linearly to $\ceq(\del{\C_2(G)})$. Note that the cycles  $D_r
\times D_s + D_s \times D_r$ generate $\ceq(\del{\C_2(G)})$ via~\eqref{e:ceq_unordered}.

\begin{prop}
\label{p:cycle_to_cycle}
Let $z$
%$$
%z = \sum\limits \alpha_{\{r,s\}} \cdot (D_r \times D_s + D_s \times D_r)
%$$
be a symmetric $4$-cycle from $\zeq(\del{\C_2(G)})$. Then $z' = \gteq (z)$ is a
symmetric $4$-cycle from $\zeq(\del{\Delta^{(2)}_{n-1}})$. In addition, $I(z')
= I(z)$. 
\end{prop}

\begin{proof}
First we verify that $z' \in \zeq(\del{\Delta^{(2)}_{n-1}})$.
  From the definition of $\gteq$, we get that $z'$ belongs to
$\ceq(\del{\Delta^{(2)}_{n-1}})$, thus we only need to verify that $z'$ is a
  $4$-cycle.

Assume that 
\[
z = \sum\limits \alpha_{\{r,s\}} \cdot (D_r \times D_s + D_s \times D_r).
\]

Then,
\begin{linenomath}
\begin{align*}
  \partial z &= \sum\limits \alpha_{\{r,s\}} \cdot (D_r \times \partial D_s +
  \partial D_r \times D_s  + D_s \times  \partial D_r + \partial D_s \times D_r)\\ 
  &= \sum\limits \alpha_{\{r,s\}} \cdot (D_r \times s +
  r \times D_s  + D_s \times  r + s \times D_r)  \\
  &= \sum_r \left( D_r \times \left(\sum_s \alpha_{\{r,s\}} s \right) + \left(\sum_s
\alpha_{\{r,s\}} s \right) \times D_r \right) 
\end{align*}
\end{linenomath}
where the outer sum is over all cycles $r$ of $G$ and the inner sums are over
all cycles $s$ of $G$ disjoint from $r$. Because $\partial z = 0$, we get that
$\sum_s \alpha_{\{r,s\}} s = 0$ for each of the inner sums.

By analogous computation using $\partial g_{\#}(D_r) = r$ we get
\begin{linenomath}
\begin{align*}
  \partial z'  
  %\sum\limits \alpha_{\{r,s\}} \cdot (D_r \times \partial D_s +
  %\partial D_r \times D_s  + D_s \times  \partial D_r + \partial D_s \times D_r)\\ 
  %&= 
%  \sum\limits \alpha_{\{r,s\}} \cdot (g_{\#}(D_r) \times s +
%  r \times g_{\#}(D_s)  + g_{\#}(D_s) \times  r + s \times g_{\#}(D_r)). 
  &= \sum_r \left( g_{\#}(D_r) \times \left(\sum_s \alpha_{\{r,s\}} s \right) + \left(\sum_s
\alpha_{\{r,s\}} s \right) \times g_{\#}(D_r) \right)  \\
  &= 0.
\end{align*}
\end{linenomath}

It remains to show $I(z) = I(z')$. Let $f \colon |\Delta_{n-1}^{(2)}| \to \R^4$
be a general position map. Note that $f \circ g \colon |C_2(G)| \to \R^4$ is in
almost general position. Thus, according to the definition of $I(z)$, we need
to show $\obst_f(z') = \obst_{f \circ g} (z)$. 

Let $r$ and $s$ be disjoint cycles of $G$ such that $f \circ g(D_r)$ and $f
\circ g(D_s)$ intersect in $k_{\{r,s\}}$
crossings. Then those two cycles contribute exactly by
$\alpha_{\{r,s\}} k_{\{r,s\}}$ to $\obst_{f \circ g} (z)$ (according to its
definition). However, the crossings between $f \circ g(D_r)$ and $f
\circ g(D_s)$ are also crossings of triangles in $g_{\#}(D_r)$ and
$g_{\#}(D_s)$ when mapped under $f$. Thus they contribute by the same amount to
$\obst_f(z')$ using that $z' = \gteq(z)$ and formula~\eqref{e:gteq}.
\end{proof}

%Given two disjoint cycles $r$ and $s$ of $G$, then $g(D_r)$ and $g(D_s)$ are
%disjoint disks. Let $g_{\#}(D_r)$ and $g_{\#}(D_s)$ be the induced chains from
%$C_2(\Delta_{n-1}^{(2)})$ (they are formed as the sum of the triangles
%triangulating $g(D_r)$ or $g(D_s)$). Consequently, the $4$-chain 
%\begin{equation}
%\label{e:gdr}
%  g_{\#}(D_r) \times g_{\#}(D_s) +
%g_{\#}(D_s) \times g_{\#}(D_r)
%\end{equation}
%belongs to $\ceq(\del{\Delta_{n-1}^{(2)}})$.
%(Here we identify $(\eta_1 + \cdots + \eta_k) \times (\tau_1 + \cdots +
%\tau_{\ell})$ with $\sum_{i,j=1,1}^{k,\ell} \eta_i \times \tau_j$.)

Now let $Z' := \gteq(\zeq(\del{\C_2(G)}))$. According to
Proposition~\ref{p:cycle_to_cycle}, $Z'$ is a subspace of
$\zeq(\del{\Delta_{n-1}^{(2)}})$. 

%\leq \ceq(\del{\Delta_{n-1}^{(2)}})$ be the span of
%chains~\eqref{e:gdr} over all (unordered) pairs $\{r,s\}$ of disjoint cycles of
%$G$. 

\begin{cor}
\label{c:z'}
  There is $z' \in Z'$ with $I(z') = 1$ if and only if $\sigma(G) > 5$.
\end{cor}

\begin{proof}
First assume that there is $z' \in Z'$ with $I(z') = 1$. Then there is also $z
\in \zeq(\del{\C_2(G)})$ such that $z' = \gteq(z)$. According to
Proposition~\ref{p:cycle_to_cycle}, $I(z) = 1$. Therefore $\sigma(G) > 5$ by
Proposition~\ref{p:Iz}.

On the other hand, let us assume that $\sigma(G) > 5$. Then there is $z
\in \zeq(\del{\C_2(G)})$ with $I(z) = 1$ by Proposition~\ref{p:Iz}. Then
$\gteq(z)$ is the required $z'$ by Proposition~\ref{p:cycle_to_cycle}.
\end{proof}

\begin{thm} 
  \label{t:coNP}
  For any graph $G$, there is an explicit\footnote{One may
  observe that a forbidden minor is a polynomial size certificate showing
$\sigma(G) > 5$. However, we do not regard such a certificate explicit as we do
not know the list of forbidden minors.} polynomial size certificate showing
  $\sigma(G) > 5$. 
\end{thm}

\begin{proof}
  By Corollary~\ref{c:z'}, it is sufficient to certificate an existence of $z'
  \in Z'$ with $I(z') = 1$. We can easily observe that the dimension of $Z'$ is
  polynomially bounded by the size of $G$ because $Z'$ is a subspace of
$\ceq(\del{\Delta^{(2)}_{n-1}})$. A safe bound is that
$\ceq(\del{\Delta^{(2)}_{n-1}})$ is generated by at most $\binom {n}{3}^2$ pairs
of triangles where $n$ is the number of vertices of $G$. Therefore, there is a
chain $z \in \zeq(\del{\C_2(G)})$ with polynomially many nonzero coordinates
satisfying $z' = \gteq(z)$ which thereby certifies that $z' \in Z$. Certifying
$I(z') = 1$ is easy via a suitable general position map (as described for the
exponential time algorithm).
\end{proof}

%\martin{Pridana poznamka nize o tom, jak daleko jsme od polynomialniho
%algoritmu.}

\begin{remark}
 If we knew how to find a basis of $Z'$ in polynomial time, then we would
 immediately get a polynomial time algorithm by evaluating $I(z')$ for all
 basis cycles $z'$.
\end{remark}

\fi

\end{document}